\documentclass[numbers,webpdf]{ima-authoring-template}

\usepackage{multirow,bm}

\theoremstyle{thmstyletwo}%
\newtheorem{theorem}{Theorem}[section]% meant for sectionwise numbers
\newtheorem{lemma}[theorem]{Lemma}
\newtheorem{remark}[theorem]{Remark}%

\numberwithin{equation}{section}
\numberwithin{figure}{section}
\numberwithin{table}{section}

\makeatletter
\DeclareFontEncoding{LS1}{}{}
\DeclareFontSubstitution{LS1}{stix}{m}{n}
\DeclareMathAlphabet{\mathcc}{LS1}{stixcal}{m}{n}
\makeatother

\newcommand{\C}{\mathcal{C}}

\newcommand{\bR}{\mathbf{R}}

\newcommand{\bP}{\mathbf{P}}
\newcommand{\bI}{\mathbf{I}}

\newcommand{\fb}{\boldsymbol{f}}

\newcommand{\bH}{\mathbf{H}}

\newcommand{\bn}{\mbox{\boldmath{$n$}}}

\newcommand{\bW}{\mathbf{W}}

\newcommand{\bs}{\boldsymbol{s}}

\newcommand{\bu}{{\boldsymbol{u}}}

\newcommand{\bS}{\mathbf{S}}
\newcommand{\bT}{\mathbf{T}}

\newcommand{\bv}{\mathbf{v}}
\newcommand{\bV}{\mathbf{V}}

\newcommand{\bL}{\mathbf{L}}

\newcommand{\bw}{\boldsymbol{w}}
\newcommand{\bx}{\boldsymbol{x}}

\newcommand{\bz}{\mbox{\boldmath{$z$}}}

\newcommand{\bepsilon}{\mbox{\boldmath{$\varepsilon$}}}

\newcommand{\brho}{\boldsymbol{\rho}}
\newcommand{\bkappa}{\revX{\boldsymbol{\kappa}}}
\newcommand{\bsigma}{\boldsymbol{\sigma}}

\newcommand{\bPi}{\mbox{\boldmath{$\Pi$}}}

\newcommand{\btau}{\boldsymbol{\tau}}

\newcommand{\bxi}{\boldsymbol{\xi}}

\newcommand{\bbL}{\mathbb{L}}
\newcommand{\bbH}{\mathbb{H}}
\newcommand{\bbLsym}{\mathbb{L}_{\mathrm{sym}}^2(\Omega)}
\newcommand{\bbLskew}{\mathbb{L}_{\mathrm{skew}}^2(\Omega)}
\newcommand{\bbHsym}{\mathbb{H}_{\mathrm{sym}}(\bdiv;\Omega)}

\newcommand{\bzero}{\boldsymbol{0}}
\newcommand{\BBR}{\mbox{$\mathbb{R}$}}

\newfont{\twelvemsb}{msbm10 at 11.6pt}

\renewcommand{\div}{\mathop{\rm div}\nolimits}
\newcommand{\bdiv}{\mathop{\mathbf{div}\,}\,}
\renewcommand{\max}{\mathop{\rm max}\limits}

\newcommand{\tr}{\mathop{\rm tr}}

\newcommand{\Om}{\Omega}
\xdefinecolor{mygrey}{rgb}{0.65,0.65,0.6375}

\newcommand\bbR{\mathbb{R}}
\def\bdiv{\mathbf{div}}
\def\qan{{\quad\hbox{and}\quad}}
\def\qin{{\quad\hbox{in}\quad}}
\newcommand{\ds}{\displaystyle}
\def\qon{{\quad\hbox{on}\quad}}

\newcommand{\bbeta}{{\boldsymbol\eta}}
\def\rskew{\mathrm{skew}}

\newcommand{\bchi}{{\boldsymbol{\chi}}}
\newcommand{\rL}{\mathrm{L}}
\newcommand{\rW}{\mathrm{W}}
\newcommand{\rH}{\mathrm{H}}

\newcommand{\bzeta}{{\boldsymbol\zeta}}

\def\rt{\mathrm{t}}

\def\rd{\mathrm{d}}

\def\wh{\widehat}
\def\dist{\mathrm{dist}\,}
\newcommand{\bbV}{\mathbb{V}}
\newcommand{\bvarphi}{{\boldsymbol{\varphi}}}
\newcommand{\bpsi}{{\boldsymbol{\psi}}}
\def\bcurl{\mathbf{curl}\,}
\newcommand{\rot}{\mathrm{rot}\,}
\newcommand{\cR}{\mathcal{R}}
\newcommand{\curl}{\mathrm{curl}\,}
\newcommand{\cT}{\mathcal{T}}
\newcommand{\bq}{\mbox{\boldmath{$q$}}}
\def\rD{\mathrm{D}}
\def\rN{\mathrm{N}}

\def\wt{\widetilde}

\newcommand\bnabla{\boldsymbol{\nabla}}

\allowdisplaybreaks

\usepackage{colordvi}

\definecolor{lightgreen}{rgb}{0.22,0.50,0.25}
\definecolor{lightblue}{rgb}{0.22,0.45,0.70}
\providecommand{\revX}{}%[1]{{\leavevmode\color{blue}{#1}}}
\providecommand{\revZ}{}%[1]{{\color{orange} #1}}
\providecommand{\revY}{}%[1]{{\color{red} #1}}

\newcommand{\cblue}{}%[1]{{\leavevmode\color{lightblue}{#1}}}

\begin{document}

\DOI{xxx}
\copyrightyear{2024}
\vol{00}
\pubyear{2024}
\access{Advance Access Publication Date: \today}
\appnotes{Paper}
\copyrightstatement{Published by Oxford University Press on behalf of the Institute of Mathematics and its Applications. All rights reserved.}
\firstpage{1}

\title[Mixed FEM for poroelasticity with stress-dependent permeability]{A priori and a posteriori error bounds for the fully mixed FEM formulation of poroelasticity with stress-dependent permeability}

\author{Arbaz Khan\ORCID{0000-0001-6625-700X}
  \address{\orgdiv{Departament of Mathematics}, \orgname{Indian Institute of Technology Roorkee (IITR)},
    \orgaddress{%\street{Street}, 
    \postcode{247667},
    \state{Roorkee}, \country{India}}}}
\author{Bishnu P. Lamichhane\ORCID{0000-0002-9184-8941}
  \address{\orgdiv{School of Information \& Physical Sciences}, \orgname{University of Newcastle}, \orgaddress{\street{University Drive}, \postcode{2308},
      \state{New South Wales}, \country{Australia}}}}
\author{Ricardo Ruiz-Baier*\ORCID{0000-0003-3144-5822} \quad and \quad  
Segundo Villa-Fuentes\ORCID{0000-0002-0377-6555}
\address{\orgdiv{School of Mathematics}, \orgname{Monash University}, \orgaddress{\street{9 Rainforest Walk}, \postcode{3800}, \state{Victoria}, \country{Australia}}}}

\authormark{Khan, Lamichhane, Ruiz-Baier, Villa-Fuentes}

\corresp[*]{Corresponding author: \href{email:ricardo.ruizbaier@monash.edu}{ricardo.ruizbaier@monash.edu}}

\received{19}{02}{2025}
%\revised{Date}{0}{Year}
%\accepted{Date}{0}{Year}

%\editor{Associate Editor: Name}

\abstract{We develop a family of mixed finite element methods for a model of nonlinear poroelasticity where, thanks to a rewriting of the constitutive equations, the permeability depends on the total poroelastic stress and on the fluid pressure and therefore we can use the Hellinger--Reissner principle with weakly imposed stress symmetry for Biot's equations. The problem is adequately structured into a coupled system consisting of one saddle-point formulation, one linearised perturbed saddle-point formulation, and two off-diagonal perturbations. This system's unique solvability requires assumptions on regularity and Lipschitz continuity of the inverse permeability, and the analysis follows  fixed-point arguments and the Babu\v{s}ka--Brezzi theory. The discrete problem is shown uniquely solvable by applying similar fixed-point and saddle-point techniques as for the continuous case. The  method is based on the classical PEERS$_k$ elements, it is exactly \revX{equilibrium} and mass conservative, and it is robust with respect to the nearly incompressible as well as vanishing storativity limits. We derive a priori error estimates, we also propose fully computable residual-based a posteriori error indicators, and show that they are reliable and efficient with respect to the natural norms, and robust in the limit of near incompressibility. These a posteriori error estimates are used to drive adaptive mesh refinement. The theoretical analysis is supported and  illustrated by several numerical examples in 2D and 3D. }

\maketitle

\section{Introduction}
Nonlinear interaction between flow and the mechanical response of saturated porous media is of a great importance in many applications in biophysics, geomechanics, and tissue engineering, for example. One of such models is the equations of nonlinear poroelasticity, whose mathematical properties were studied in great detail, for example, in the references \cite{bociu2016analysis,bociu2021nonlinear,bociu2022weak}. In these works, it becomes clear that a distinctive property of nonlinear poroelasticity models targeted for, e.g., soft tissue (cartilage, trabecular meshwork, brain matter, etc.), is that the nonlinear permeability (the hydraulic conductivity,  defined as how easily pore fluid escapes from the compacted pore spaces) that depends on the evolving total amount of fluid, does not entail a monotone operator, and therefore one cannot readily apply typical tools from monotone saddle-point problems. 

Our interest is in deriving mixed finite element (FE) formulations (solving also for other variables of interest), and for this we can cite in particular \cite{gomez23,lamichhane24}, where  fully mixed formulations based on the Hu--Washizu principle are studied. Writing the poroelasticity equations in terms of the strain tensor was motivated in particular in \cite{lamichhane24} because the permeability -- at least in the regime we focus \revX{on} here -- depends nonlinearly on the total amount of fluid, which is a function of strain. 

The upshot here compared to \cite{lamichhane24} is that we are able to rewrite the constitutive equation for permeability to depend on the total poroelastic stress and on the fluid pressure (similarly as in, e.g., \cite{barbeiro2010priori}). This allows us to revert to the more popular Hellinger--Reissner type of mixed formulations for poroelasticity  \cite{baerland17,lee16,li20,borregales2021iterative,yi2014convergence} (without solving explicitly for the strain). Consequently, another appealing advantage with respect to the formulation in \cite{lamichhane24} is that, as in the Hellinger--Reissner formulation, \revX{the model becomes robust   in the nearly incompressibility regime}. Also in contrast to \cite{lamichhane24}, in this work we use a mixed form for the fluid flow (adding the discharge flux as additional unknown), which gives the additional advantage of mass conservativity. 

Regarding the well-posedness analysis, the aforementioned non-monotonicity of the permeability suggest, for example, to use a fixed-point argument. We opt for freezing the arguments  of permeability, turning the double saddle-point structure with three perturbations coming from the stress trace operator and from the $L^2$ pressure blocks, into two decoupled saddle-point problems whose separate solvability can be established from the classical literature for weakly symmetric elasticity and mixed reaction-diffusion equations. Banach fixed-point theorem is then used to show well-posedness of the overall problem. This analysis needs to verify conditions of ball-mapping and contraction of the fixed-point map, and this imposes a  small data assumption, which can be carried over to the external load, mass source, boundary displacement, and boundary fluid pressure. \revX{Compared to \cite{lamichhane24}, these conditions are less restrictive and imply also a less restrictive discrete analysis (which follows closely the continuous one), due to the analysis being performed using the inverse of the Hooke tensor, which allows us to achieve robustness with respect to the first Lam\'e parameter $\lambda$.} 

Note that at the discrete level we can simply use conforming FE spaces. Discrete inf-sup conditions are already well-known for the chosen FE families of PEERS$_k$ and Raviart--Thomas elements used for the solid and fluid sub-problems (but several other inf-sup stable spaces that satisfy a discrete kernel characterisation are also possible). We emphasise that, similarly to \cite{lee16,li20}, all estimates hold uniformly in the limit of nearly incompressibility (implying that the formulation is Poisson locking-free) as well as when the constrained storage coefficient vanishes (poroelastic locking-free), and therefore they are free of non-physical pressure oscillations. 

An additional goal of this work is to derive efficient and reliable residual a posteriori error estimators for the nonlinear poroelasticity equations. The approach follows a similar treatment as that of \cite{gatica22} (which focuses on mixed formulations of stress-assisted diffusion equations), with the difference that here we do not need to include augmentation terms for the mixed form of the mixed diffusion problem. The main ingredients in the analysis of these estimates are Helmholtz decompositions and a global inf-sup (together with boundedness and Lipschitz continuity of coupling terms), local inverse and trace estimates, bubble-based localisation arguments, and properties of Cl\'ement and Raviart--Thomas interpolators. See also \cite{alvarez16,elyes18} for estimators in a similar multiphysics context and, e.g.,  \cite{carstensen1998posteriori,carstensen2000locking,LV2004} for mixed linear elasticity.  Note that for the reliability of the estimator the aforementioned Helmholtz decompositions -- for both tensor-vector and vector-scalar cases -- should be valid for mixed boundary conditions. For this we follow \cite{alvarez16} and \cite{cov2020c}, from which we \revX{inherit a convexity} assumption on the Neumann sub-boundary (where we impose traction and flux boundary conditions). 

\smallskip 
\noindent\textbf{Outline.} The rest of the paper is organised as follows. The remainder of this Section has a collection of preliminary definitions and notational convention, as well as the statement of the governing partial differential equations. The weak formulation and proofs of the uniform boundedness of the bilinear forms and suitable  inf-sup conditions are shown in Section~\ref{sec:weak}. The fixed-point analysis of the coupled problem is carried out in Section~\ref{sec:analysis}. Section~\ref{sec:FE} then focuses on the Galerkin discretisation, including its well-posedness analysis and definition of specific FE subspaces that provide \revX{equilibrium} and mass conservativity. In Section~\ref{sec:error} we show a C\'ea estimate and using appropriate approximation properties we derive optimal a priori error bounds including also the higher order case. The definition of a residual a posteriori error estimator and the proofs of its reliability and efficiency are presented in Section~\ref{sec:aposteriori}. We display in  Section~\ref{sec:numer} some numerical tests that both validate and underline the theoretical properties of the proposed discretisations, \revX{and close in Section~\ref{sec:concl} with a summary and a discussion on possible extensions.} 

\smallskip 
\noindent\textbf{Notation and preliminaries.}
Let $\rL^2(\Omega)$ be the set of all square-integrable functions in $\Omega \subset \BBR^d$ where $d \in \{2,3\}$ is the spatial dimension, and denote by $\bL^2(\Omega)=\rL^2(\Omega)^d$ its vector-valued counterpart and by $\bbL^2(\Omega)=\rL^2(\Omega)^{d\times d}$ its tensor-valued counterpart. We also write  
\[\bbLskew:=\{\btau \in \bbL^2(\Omega):\,\btau = -\btau^{\tt t} \},\]
to represent the skew-symmetric tensors in $\Omega$ with each  component being square-integrable. Standard notation will be employed for Sobolev spaces $\rH^m(\Omega)$ with $m\geq 0$ (and we note that $\rH^0(\Omega)=\rL^2(\Omega)$). Their norms and seminorms are denoted as $\|\cdot\|_{m,\Omega}$ and $|\cdot|_{m,\Omega}$, respectively (as well as for their vector and tensor-valued counterparts $\bH^m(\Omega)$, $\bbH^m(\Omega)$) see, e.g., \cite{BS94}. 

As usual $\revX{\mathbf{I}}$ stands for the identity tensor in $\bbR^{d\times d}$, 
and $|\cdot|$ denotes the Euclidean norm  in $\bbR^d$. Also, for any vector field $\bv=(v_i)_{i=1,d}$  we set the gradient and divergence operators as
$
\bnabla \bv \,:=\, \left(\frac{\partial v_i}{\partial x_j}\right)_{i,j=1,d}$ and  $\div\bv \,:=\, \sum_{j=1}^d \frac{\partial v_j}{\partial x_j}$. 
In addition, for any tensor fields $\btau=(\tau_{ij})_{i,j=1,d}$
and $\bzeta = (\zeta_{ij})_{i,j=1,d}$, we let $\bdiv\,\btau$ be the divergence operator $\div$ acting along the rows of $\btau$, and define the transpose, the trace, the tensor inner product, and the deviatoric tensor  as
$\btau^\rt := (\tau_{ji})_{i,j=1,d}$, 
$\tr(\btau) := \sum_{i=1}^d\tau_{ii}$,  
$\btau:\bzeta := \sum_{i,j=1}^n\tau_{ij}\zeta_{ij}$, and 
$\btau^\rd := \btau - \revX{\frac{1}{d}}\,\tr(\btau)\,\bI$, respectively. 
We also recall the Hilbert space
\[
\bH(\div;\Omega)\,:=\,\big\{\bz \in \bL^2(\Omega):\, \div\,\bz \in \mathrm{L}^2(\Omega)\big\},
\]
with norm $\|\bz\|_{\div;\Omega}^{2}:=\|\bz\|_{0,\Omega}^{2}+\|\div\,\bz\|_{0,\Omega}^{2}$, and introduce its tensor-valued version 
\[
\bbH(\bdiv;\Omega)\,:=\,\big\{\btau \in \bbL^2(\Omega):\, \bdiv\,\btau \in \bL^2(\Omega)\big\}.
\]
%and the tensor space $\bbHsym$, given by
%\[
%\bbHsym\,:=\,\big\{ \btau \in \bbLsym:\  \bdiv\,\btau \in \bL^2(\Omega) \big\}
%\]
%whose norm will be denoted by $\|\cdot\|_{\bdiv;\Omega}$.

%%%%%%%%%%%%%%%%%%%%%%%%%%%%%%%%%%%%%%%%%%%%%%%%%%%%%%%%%%%%%%%%%%%%%%%%%%%%%%%
\noindent\textbf{Governing equations.} 
Let us consider a fully-saturated poroelastic medium (consisting of a mechanically isotropic and homogeneous fluid-solid mixture) occupying the open and bounded domain $\Omega$ in $\BBR^d$, the Lipschitz boundary $\partial\Omega$ is partitioned into disjoint sub-boundaries $\partial\Omega:= \overline{\Gamma_\rD} \cup \overline{\Gamma_\rN}$, and it is assumed for the sake of simplicity that both sub-boundaries are non-empty $|\Gamma_\rD|\cdot|\Gamma_\rN|>0$.
The symbol $\bn$ will stand for the unit outward normal vector on the boundary. Let $\fb \in \bL^2(\Omega)$ be a prescribed body force per unit of volume (acting on the fluid-structure mixture) and let $g \in L^2(\Omega)$ be a net volumetric fluid production rate.  

The \revX{equilibrium (balance of linear momentum)} for the solid-fluid mixture is written as 
\begin{equation}
\label{eq:mom} -\bdiv\, \bsigma = \fb \qquad \text{in $\Omega$},
\end{equation}
with $\bsigma$ being the total Cauchy stress tensor of the mixture (sum of the effective solid and fluid stresses), whose dependence on strain and on fluid pressure  is given by the constitutive assumption (or effective stress principle) 
\begin{equation}\label{eq:constitutive}
  \bsigma = \C \bepsilon (\bu) -\alpha p \revX{\mathbf{I}} \qquad \text{in $\Omega$}.
\end{equation}
Here the skeleton displacement vector $\bu$ from the position $\bx \in \Omega$ is an unknown, the tensor $\bepsilon (\bu) := \frac{1}{2} (\bnabla \bu + [\bnabla \bu]^{\tt t} )$ is the infinitesimal strain, by $\C $ we denote the fourth-order elasticity tensor, also known as Hooke's tensor (symmetric and positive definite and  characterised by \revX{$\C \btau := \lambda(\tr\btau)\mathbf{I} + 2\mu\,\btau$), $\lambda$ and $\mu$} are the Lam\'e parameters (assumed constant and positive), $0\leq \alpha \leq 1$ is the Biot--Willis parameter, and $p$ denotes the Darcy fluid pressure (positive in compression), which is an unknown in the system. 

We also consider the balance of angular momentum, which in this context states that the total poroelastic stress is a symmetric tensor 
%\begin{equation}\label{eq:angular}
$    \bsigma = \bsigma^{\tt t}$.
%\end{equation}
To weakly impose it, it is customary to use the rotation tensor
\begin{equation}\label{eq:vort}
\brho = \frac{1}{2} (\bnabla \bu - [\bnabla \bu]^{\tt t} ) = \bnabla \bu -\bepsilon (\bu).
\end{equation}
The fluid content (due to both fluid saturation and local volume dilation) is given by 
\begin{equation}\label{eq:content}
\zeta = c_0 p + \alpha \div \bu,
\end{equation}
where $c_0 \geq 0$ is the constrained specific storage %(or storativity)
coefficient. 
%, {\color{red} and so $c_0\geq 0$.}
Using Darcy's law to describe the discharge velocity in terms of the fluid pressure gradient,  the balance of mass for the total amount of fluid is $\partial_t\zeta - \div (\bkappa \nabla p)  = g$ in $\Omega \times (0,t_{\mathrm{end}})$, where $\bkappa$ is the intrinsic permeability \revX{tensor} %(divided by the fluid viscosity) 
of %the laminar flow in 
the medium, a nonlinear function of the porosity.  In turn, in the small strains limit the porosity can be approximated by a linear function of the fluid content $\zeta$ (see for example \cite[Section 2.1]{van2023mathematical}), and so, thanks to 
\eqref{eq:content}, we can simply write 
\[\bkappa = \bkappa(\bepsilon (\bu),p).\]  
Furthermore, after a backward Euler semi-discretisation in time with a constant time step and rescaling appropriately, we only consider the type of equations needed to solve at each time step and therefore we will concentrate on the form 
\begin{equation}\label{eq:mass}
c_0 p + \alpha \tr\bepsilon (\bu) - \div (\bkappa(\bepsilon (\bu),p) \nabla p)  = g \qquad \text{in $\Omega$}.
\end{equation}
Typical constitutive relations for permeability are, e.g., exponential or Kozeny--Carman type (cf.  \cite{ateshian10}) 
\begin{equation}\label{eq:kappa}
    \bkappa
    (\bepsilon (\bu),p) 
    = \frac{k_0}{\mu_f}\mathbf{I} + \frac{k_1}{\mu_f}\exp(k_2 (c_0p + \alpha\tr\bepsilon (\bu)))\mathbf{I}, \quad \bkappa(\bepsilon (\bu),p) =   \frac{k_0}{\mu_f}\mathbf{I} +  \frac{k_1(c_0p + \alpha\tr\bepsilon (\bu))^3}{\mu_f(1-(c_0p + \alpha\tr\bepsilon (\bu)))^2}\mathbf{I},
\end{equation}
where $\mu_f$ denotes the viscosity of the interstitial fluid and $k_0,k_1,k_2$ are model constants. We note that in the case of incompressible constituents one has $c_0 = 0 $ and $\alpha = 1$, indicating that permeability depends only on the dilation $\tr\bepsilon (\bu) = \div\bu$ (see, e.g., \cite{bociu2016analysis}). We also note that even in such a scenario (of incompressible phases) the overall mixture is not necessarily incompressible itself. 
More precise assumptions on the behaviour of the permeability are postponed to Section~\ref{sec:stability-properties}.
Next we note that from \eqref{eq:constitutive} we can obtain
%\begin{subequations}
\begin{equation}\label{eq:trace-sigma}
 \tr \bsigma = (d\lambda+2\mu)\div\bu -d\alpha p \qan 
 %\qquad \text{in $\Omega$}, \\
%\label{eq:constitutive-2}
\C^{-1}\bsigma + \dfrac{\alpha}{d\lambda+2\mu} p \mathbf{I}  = \bepsilon (\bu) \qquad \text{in $\Omega$}.
\end{equation}
%\end{subequations}
%\cred{[I don't think we need 1.7b for this step. It suffices to scale 1.7a]} 
Then, from the first equation in 
%\eqref{eq:constitutive-2}, using that $$\C^{-1}\bsigma=\dfrac{1}{2\mu}\bsigma -\dfrac{\lambda}{2\mu(d\lambda +2\mu)}\tr\bsigma\bbI,$$
\eqref{eq:trace-sigma} we get 
\begin{equation*}%\label{eq:constitutive-2}
\tr \bepsilon (\bu) = \dfrac{1}{d\lambda + 2\mu}\tr\bsigma + \dfrac{d\alpha}{d\lambda+2\mu} p,
\end{equation*}
and therefore %it is clear that 
the dependence of $\bkappa$ on $\bepsilon (\bu)$ and $p$ (cf.  \eqref{eq:kappa}) can be written in terms of $\bsigma$ and $p$ as follows 
\begin{equation}\label{const-kappa-2}
\begin{array}{cc}
\ds  \bkappa(\bsigma,p) = \frac{k_0}{\mu_f}\mathbf{I} + \frac{k_1}{\mu_f}\exp\Big(\dfrac{k_2}{d\lambda+2\mu} \big( (c_0(d\lambda+2\mu)+d\alpha^2)\, p + \alpha\tr\bsigma \big) \Big)\mathbf{I},\\[1ex]
\ds \bkappa(\bsigma,p) =   \frac{k_0}{\mu_f}\mathbf{I} +  \frac{k_1\big( (c_0(d\lambda+2\mu)+d\alpha^2)\, p + \alpha\tr\bsigma\big)^3}{(d\lambda+2\mu)\mu_f\big(d\lambda+2\mu-((c_0(d\lambda+2\mu)+d\alpha^2)\, p + \alpha\tr\bsigma)\big)^2}\mathbf{I},
  \end{array}
\end{equation}
\revY{and emphasize that the permeability tensor is the same constitutive law on either \eqref{eq:kappa} or \eqref{const-kappa-2}, so,  making abuse of notation we have 
\[ \bkappa = \bkappa (\zeta)= \bkappa(\bepsilon(\bu),p) = \bkappa(\bsigma,p). \]}
In addition, putting together the second equation in \eqref{eq:trace-sigma} and \eqref{eq:vort} we obtain:
\begin{equation}\label{eq:constitutive-3}
\C^{-1}\bsigma + \dfrac{\alpha}{d\lambda+2\mu} p \mathbf{I} = \bnabla \bu -\brho \qquad \text{in $\Omega$}.
\end{equation}

Finally, we introduce the discharge flux  $\bvarphi$ as an unknown defined by the constitutive relation 
\begin{equation}\label{eq:varphi}
\bkappa(\bsigma,p)^{-1}\bvarphi = \nabla p,
\end{equation}
and combining \eqref{eq:trace-sigma} and \eqref{eq:mass}, we are able to rewrite the mass balance equation as 
\begin{equation}\label{eq:mass2}
c_0 p + \dfrac{\alpha}{d\lambda+2\mu}\tr\bsigma + \dfrac{d\alpha^2}{d\lambda+2\mu} p - \div \bvarphi  = g \qquad \text{in $\Omega$}.
\end{equation}

To close the system, we consider  mixed boundary conditions for a given $\bu_\rD \in \bH^{1/2}(\Gamma_\rD)$ and $p_\rD \in \rH^{1/2}(\Gamma_\rD)$: 
\begin{equation}\label{eq:bc}
\bu = \bu_\rD \qan p = p_\rD  \qon \Gamma_\rD, \qquad 
\bvarphi\cdot \bn = 0 \qan  \bsigma \bn = \boldsymbol{0} \qon \Gamma_\rN.
\end{equation}

%%%%%%%%%%%%%%%%%%%%%%%%%%%%%%%%%%%%%%%%%%%%%%%%%%%%%%%%%%%%%%%%%%%%

\section{Weak formulation and preliminary properties}\label{sec:weak}
\subsection{Derivation of weak forms}
Let us define the following spaces
\begin{align*}
\ds \bbH_\rN(\bdiv;\Omega)&:=\,\big\{\btau \in \bbH(\bdiv;\Omega):\, \btau\bn=\boldsymbol{0} \qon \Gamma_\rN\big\},\\ 
\ds \bH_\rN(\div;\Omega)&:=\,\big\{\bpsi \in \bH(\div;\Omega):\, \bpsi\cdot\bn=0 \qon \Gamma_\rN\big\}.
\end{align*}
We test equation \eqref{eq:mom} against $\bv\in\bL^2(\Omega)$,  equation \eqref{eq:constitutive-3} against $\btau\in\bbH_\rN(\bdiv;\Omega)$, impose the symmetry of $\bsigma$ weakly, test equation \eqref{eq:mass2} against $q\in\rL^2(\Omega)$,   equation \eqref{eq:varphi} against $\bpsi\in\bH_\rN(\div;\Omega)$, integrate by parts and \revX{use} the boundary conditions \eqref{eq:bc} naturally, and then reorder the resulting equations. Then we arrive at 
\begin{align*}
\ds  \int_{\Omega} \C^{-1}\bsigma:\btau + \dfrac{\alpha}{d\lambda+2\mu} \int_{\Omega} p\, \tr\btau  + \int_{\Omega} \bu\cdot\bdiv\btau + \int_{\Omega} \brho:\btau &= \ds \langle \btau\bn,\bu_\rD\rangle_{\Gamma_\rD}&\quad\forall\, \btau\in\bbH_\rN(\bdiv;\Omega),\\
\ds\int_{\Omega} \bv \cdot \bdiv\, \bsigma   &= \ds -\int_{\Omega}  \fb \cdot \bv&\quad \forall\,\bv\in\bL^2(\Omega),\\
\ds \int_{\Omega} \bsigma:\bbeta &=0 &\quad\forall\, \bbeta\in\bbLskew,\\[1ex]
\ds \int_{\Omega} \bkappa(\bsigma,p)^{-1} \bvarphi \cdot \bpsi +  \int_{\Omega}  p \div\bpsi &= \ds \langle \bpsi\cdot\bn,p_\rD\rangle_{\Gamma_\rD}&\quad\forall\, \bpsi\in\bH_\rN(\div;\Omega),\\
\ds \big(c_0 +\dfrac{d\alpha^2}{d\lambda+2\mu}\big)\int_{\Omega} p\,q +  \dfrac{\alpha}{d\lambda+2\mu} \int_{\Omega} q\, \tr\bsigma - \int_{\Omega}q\div\bvarphi&= \ds \int_{\Omega}  g\, q &\quad\forall\, q\in \rL^2(\Omega),
\end{align*}
where $\langle\cdot,\cdot\rangle_{\Gamma_\rD}$ denotes the duality pairing between $\rH^{-1/2}(\Gamma_\rD)$ and its dual $\rH^{1/2}(\Gamma_\rD)$ with respect to the inner product in $L^2(\Gamma_\rD)$, and we use the same notation, $\langle\cdot,\cdot\rangle_{\Gamma_\rD}$, in the vector-valued case. 

Next we proceed to introduce the  bilinear forms $a:\bbH_\rN(\bdiv;\Omega)\times\bbH_\rN(\bdiv;\Omega)\to\bbR$, $b:\revX{\bbH_\rN(\bdiv;\Omega)\times[\bL^2(\Omega) \times \bbLskew]\to\bbR}$, $c:\bbH_\rN(\bdiv;\Omega)\times\rL^2(\Omega)\to\bbR$, the nonlinear form $\wt a_{\wh\bsigma,\wh p}:\bH_\rN(\div;\Omega)\times\bH_\rN(\div;\Omega)\to\bbR$, and the \revX{bilinear forms} $\wt b:\bH_\rN(\div;\Omega) \times \rL^2(\Omega) \to\bbR$ and $\wt c:\rL^2(\Omega) \times \rL^2(\Omega) \to\bbR$, as follows 
\begin{equation}\label{eq:forms}
\begin{array}{c}
\ds a(\bsigma,\btau) := \int_{\Omega} \C^{-1}\bsigma:\btau,\quad 
b(\btau,(\bv,\bbeta)) := \int_{\Omega} \bv \cdot \bdiv\, \btau + \int_{\Omega} \btau:\bbeta,\quad c(\btau,q):=\dfrac{\alpha}{d\lambda+2\mu} \int_{\Omega} q\, \tr\btau, \\ [3ex]
\ds \wt a_{\wh\bsigma,\wh p}(\bvarphi,\bpsi) := \int_{\Omega} \bkappa(\wh\bsigma,\wh p)^{-1} \bvarphi \cdot \bpsi ,\quad
\wt b(\bpsi,q) := \int_{\Omega}  q \div\bpsi,\quad \wt c(p,q):=\big(c_0 +\dfrac{d\alpha^2}{d\lambda+2\mu}\big)\int_{\Omega} p\,q ,
\end{array}
\end{equation}
respectively, and  linear functionals $H\in\bbH_\rN(\bdiv;\Omega)'$, $F\in (\bL^2(\Omega)\times\bbLskew)'$, $\wt H\in\bH_\rN(\div;\Omega)'$,  $\cblue{\tilde{F}}\in \rL^2(\Omega)'$  %by 
\begin{equation*}
H(\btau): = \langle \btau\bn,\bu_\rD\rangle_{\Gamma_\rD}, \quad 
F(\bv,\bbeta):=\ds -\int_{\Omega}  \fb \cdot \bv, \quad \wt H(\bpsi): = \langle \bpsi\cdot\bn,p_\rD\rangle_{\Gamma_\rD}, \quad 
\cblue{\tilde{F}}(q):=\ds -\int_{\Omega}  g\, q ,
\end{equation*}
we arrive at: 
%the fully-mixed variational formulation: 
find $(\bsigma,\bu,\brho,\bvarphi,p)\in \bbH_\rN(\bdiv;\Omega)\times\bL^2(\Omega)\times\bbLskew\times \bH_\rN(\div;\Omega)\times\rL^2(\Omega)$, such that:
\begin{subequations}\label{eq:weak-formulation}
\begin{align}\label{eq:2.2a}
%\begin{array}{rllll}
\ds a(\bsigma,\btau)  & +\quad \ds b(\btau,(\bu,\brho))  &+\quad c(\btau,p)&= \ds H(\btau)&\forall\, \btau\in\bbH_\rN(\bdiv;\Omega),\\[1ex]\label{eq:2.2b}
\ds b(\bsigma,(\bv,\bbeta))  &  & & =\ds F(\bv,\bbeta)&\forall\,\bv\in\bL^2(\Omega),\, \forall\,\bbeta\in\bbLskew,\\[1ex]
\ds \wt a_{\bsigma,p}(\bvarphi,\bpsi)  & +\quad \ds \wt b(\bpsi,p)  & &= \ds \wt H(\bpsi)&\forall\, \bpsi\in\bH_\rN(\div;\Omega),\\[1ex]\label{eq:2.2d}
\ds \wt b(\bvarphi,q)  &-\quad\wt c(p,q) & -\quad c(\bsigma,q) & =\ds \cblue{\tilde{F}}(q)&\forall\, q\in\rL^2(\Omega).
%\end{array}
\end{align}
\end{subequations}
\revY{In what follows, we stress that bilinear forms without tildes will refer to the perturbed saddle-point problem of the solid, and bilinear forms with tildes relate with the perturbed saddle-point sub-system of the interstitial fluid. Functions with hats will typically denote fixed quantities (which will be of importance in the fixed-point setting).}
%%%%%%%%%%%%%%%%%%%%%%%%%%%%%%%%%%%%%%%%%%%%%%%%%%%%%%%%%%%%%%%%%%%%%%%%%%
\subsection{Stability properties and suitable inf-sup conditions}\label{sec:stability-properties}
For the sake of the analysis, we allow the permeability $\bkappa(\bsigma,p)$ to be anisotropic but still require $\bkappa(\bsigma,p)^{-1}$ to be uniformly positive definite  in $\bbL^\infty(\Omega)$ and  Lipschitz continuous with respect to $p\in \rL^2(\Omega)$. That is, there exist \revX{strictly positive} constants $\kappa_1,\kappa_2$ such that 
\begin{equation}\label{prop-kappa}
\kappa_1|\bv|^2 \leq \bv^{\tt t}\bkappa(\cdot,\cdot)^{-1}\bv, %\leq \kappa_2|\bv|^2, 
%\qquad \bw^{\tt t}\bkappa(\cdot,\cdot)\bv \leq \kappa_2 \bw\cdot \bv, 
\qquad 
\|\bkappa(\cdot,p_1)^{-1} - \bkappa(\cdot,p_2)^{-1}\|_{\bbL^\infty(\Omega)} \leq \kappa_2 \|p_1 -p_2\|_{0,\Omega},
\end{equation}
for all $ \bv \in\mathbb{R}^d\setminus\{\bzero\}$, and for all $p_1,p_2\in \rL^2(\Omega)$. 

We start by establishing the boundedness of  the bilinear forms $a$, $b$, $c$, $\wt b$, $\wt c$:
\begin{subequations}
\begin{gather}
\label{eq:bounded-a-b}
\big|a(\bsigma,\btau)\big|\leq \dfrac{1}{\mu}\|\bsigma\|_{\bdiv;\Omega}  \|\btau\|_{\bdiv;\Omega},\qquad \big|b(\btau,(\bv,\bbeta))\big|\leq \|\btau\|_{\bdiv;\Omega}(\|\bv\|_{0,\Omega} + \|\bbeta\|_{0,\Omega}),  \\ \label{eq:bounded-c}
\big|c(\btau,q)\big|\leq \gamma \|\btau\|_{\bdiv;\Omega}\|q\|_{0,\Omega},
\\ \label{eq:bounded-hat-b-c}
\big|\wt b(\bpsi,q)\big|\leq \|\bpsi\|_{\div;\Omega}\|q\|_{0,\Omega} ,\qquad \big|\wt c(p,q)\big|\leq \wt\gamma \|p\|_{0,\Omega}\|q\|_{0,\Omega},
\end{gather}
\end{subequations}
where
\begin{equation}\label{eq:def-gamma}
\ds \gamma:=\dfrac{\alpha\sqrt{d}}{d\lambda+2\mu} \qan \wt\gamma:=c_0 +\dfrac{d\alpha^2}{d\lambda+2\mu}.
\end{equation}

On the other hand, using H\"older's and trace inequalities we can readily observe that the right-hand side functionals are all bounded
\begin{gather*}%\label{des:bound-F-H-G}
  \big|H(\btau)\big|  \leq \|\bu_\rD\|_{1/2,\Gamma_\rD} \|\btau\|_{\bdiv;\Omega}, %\quad \forall \btau \in \bbX,
 \qquad 
 \big|F(\bv,\bbeta)\big|  \leq\|\fb\|_{0,\Omega}\|\bv\|_{0,\Omega}\leq \|\fb\|_{0,\Omega}(\|\bv\|_{0,\Omega} + \|\bbeta\|_{0,\Omega}), %\quad \forall \bv \in \bL^2(\Omega),
 \nonumber \\
\big|\wt H(\bpsi)\big|  \leq \|p_{\rD}\|_{1/2,\Gamma_\rD} \|\bpsi\|_{\div;\Omega}, \qquad \big|\cblue{\tilde{F}}(q)\big|  \leq \|g\|_{0,\Omega}\|q\|_{0,\Omega}.
\end{gather*}
Let us now denote by $\bbV$ and $\bV$ the kernels of $b$ and $\wt b$, respectively. They are characterised, respectively, as
\begin{subequations}
\begin{align}\label{eq:kernel-b}
%\begin{array}{ll}
 \bbV   %:=  \ds \Big\{\btau\in \bbH_\rN(\bdiv;\Omega) :\quad b(\btau,(\bv,\bbeta)) = 0 \quad \forall\,(\bv,\bbeta)\in \bL^2(\Omega)\times\bbLskew \Big\}\\
% \ds\qquad 
 &= \ds \Big\{\btau\in \bbH_\rN(\bdiv;\Omega) :\quad \bdiv\,\btau = \bzero \qan \btau = \btau^{\tt t} \qin \Omega\Big\},
\\
  \bV   %:=   \Big\{\bpsi\in \bH_\rN(\div;\Omega) : \wt b(\bpsi,q) = 0 \quad \forall\,q\in \rL^2(\Omega) \Big\}
 &= \Big\{\bpsi\in \bH_\rN(\div;\Omega) :\quad  \div\,\bpsi = 0 \quad \text{in }\,\Omega\Big\}.
\end{align}\end{subequations}
From \cite[Lemmas 3.1 and 3.2]{agr2015} we easily deduce that there exists $c_a>0$ such that
\begin{equation}\label{eq:ellipticity-a}
\ds a(\btau,\btau) \geq c_a \|\btau\|_{\bdiv;\Omega}^2 \quad \forall\,\btau\in\bbV.
\end{equation}
\begin{remark}
\revY{From \cite[Lemma 2.2]{gatica06elast} and \cite[Lemma 2.2]{gatica14}, we have that the ellipticity constant $c_a$ has the form $c_a=\hat c \frac{1}{\mu}$, where $\hat c$ depends on $\Gamma_\rN$, $|\Omega|$, and the Poincar\'e constant. Based on the above, $c_a$ has a non-zero lower bound. }
\end{remark} 

The following inf-sup conditions are well-known to hold (see, e.g., \cite{cgorv2015}):
\begin{subequations}
 \begin{align}\label{eq:infsup-b}
\sup_{\bzero\neq\btau\in \bbH_\rN(\bdiv;\Omega) } \frac{b(\btau,(\bv,\bbeta))}{\|\btau\|_{\bdiv;\Omega}} &\geq \beta(\|\bv\|_{0,\Omega} +\|\bbeta\|_{0,\Omega})
\quad \forall\,(\bv,\bbeta)\in \bL^2(\Omega)\times\bbLskew,\\
\label{eq:infsup-hat-b}
\ds \sup_{\bzero\neq \bpsi\in \bH_\rN(\div;\Omega)}  \frac{\wt b(\bpsi,q)}{\|\psi\|_{\div;\Omega}} &\geq \wt\beta \|q\|_{0,\Omega}
\quad \forall\,q\in \rL^2(\Omega). 
\end{align}   
\end{subequations}
Finally, we observe that $\wt c$ is elliptic over $\rL^2(\Omega)$
\begin{equation}\label{eq:ellipticity-hat-c}
\ds \wt c(q,q) \geq \wt\gamma\, \|q\|_{0,\Omega}^2.
\end{equation}
\begin{remark}
\revX{Due to the careful choice of the bilinear form $a(\cdot,\cdot)$, constants in continuity estimates 
\eqref{eq:bounded-a-b} --\eqref{eq:bounded-hat-b-c} and inf-sup conditions
\eqref{eq:infsup-b} and \eqref{eq:infsup-hat-b}
do not blow up when $\lambda \to \infty$ and 
$c_0 \to 0$. In particular, the constants $\gamma$ and $\wt\gamma$, which depend on $\lambda$ and $c_0$ (cf. \eqref{eq:def-gamma}), remain bounded.}
\end{remark}

%%%%%%%%%%%%%%%%%%%%%%%%%%%%%%%%%%%%%%%%%%%%%%%%%%%%%%%%%%%%%%%%%%%%%%%
\section{Analysis of the coupled problem}\label{sec:analysis}
We now use a combination of the classical Babu\v ska--Brezzi and Banach fixed-point theorems to establish the well-posedness of \eqref{eq:weak-formulation} under appropriate assumptions on the data.
%%%%%%%%%%%%%%%%%%%%%%%%%%%%%%%%%%%%%%%%%%%%%%%%%%%%%%%%%%%%%%%%%%%%%%%
\subsection{A fixed‑point operator}
We adopt a similar approach to, e.g., \cite{cov2020}. 
%and delineate the fixed-point strategy to establish the well-posedness of \eqref{eq:weak-formulation}. We initiate by defining, %introducing the fixed-point operator corresponding to this approach. %To proceed, 
%for a given $r>0$, %we first define 
\revX{First}, we define a closed ball of $\rL^2(\Omega)$ centred at the origin and of given radius $r>0$
\begin{equation}\label{eq:set-W}
 \rW := \{ \wh p \in\rL^2(\Omega) \,:\quad \|\wh p\|_{0,\Omega} \leq r \}.
\end{equation}
Then, for a given $(\wh\bsigma,\wh p)\in\bbH_\rN(\bdiv;\Omega)\times \rW$, thanks to the assumptions on the nonlinear permeability, we can infer that the form $\wt a_{\wh\bsigma,\wh p}$ (cf. \eqref{eq:forms}) is continuous, as well as coercive over $\bV$
\begin{subequations}
\begin{align}\label{eq:bound-a}
\ds \big|\wt a_{\wh\bsigma,\wh p}(\bvarphi,\bpsi)\big| &\leq C_{\wt a}\, \|\bvarphi\|_{\div;\Omega} \|\bpsi\|_{\div;\Omega},\\
\label{eq:ellipticity-hat-a}
\ds \wt a_{\wh\bsigma,\wh p}(\bpsi,\bpsi) &\geq \kappa_1 \|\bpsi\|_{\div;\Omega}^2 \quad\forall\,\bvarphi,\,\bpsi\in\bV.%\bpsi\in \bH_\Gamma(\div;\Omega),
\end{align}
\end{subequations}
%with $\wt C_a:=\max \Big\{ \kappa_2 r,\, c_0 +\dfrac{d\alpha^2}{d\lambda+2\mu}  \Big\}$ and$\wt c_a:=\min \Big\{ \kappa_1, c_0 +\dfrac{d\alpha^2}{d\lambda+2\mu} \Big\}$.

Then, we define the auxiliary operators 
$\bR:\rW\subseteq\rL^2(\Omega)\to \bbH_\rN(\bdiv;\Omega)\times(\bL^2(\Omega)\times\bbLskew)$ and $\bS: \bbH_\rN(\bdiv;\Omega)\times\rW \to \bH_\rN(\div;\Omega)\times\rL^2(\Omega)$, given by
\begin{equation*}
\bR(\wh p):=\big(R_1(\wh p),(R_2(\wh p),R_3(\wh p))\big)=(\bsigma,(\bu,\brho)) \quad \forall\, \wh p\in\rW,
\end{equation*}
with $(\bsigma,(\bu,\brho))\in \bbH_\rN(\bdiv;\Omega)\times(\bL^2(\Omega)\times\bbLskew)$ satisfying
\begin{equation}\label{eq:weak1}
\begin{array}{rllll}
\ds a(\bsigma,\btau)  & +\quad \ds b(\btau,(\bu,\brho))  &= \ds H(\btau) - c(\btau,\wh p)&\forall\, \btau\in\bbH_\rN(\bdiv;\Omega),\\ [1ex]
\ds b(\bsigma,(\bv,\bbeta))  &  &  =\ds F(\bv,\bbeta)&\forall\,(\bv,\bbeta)\in\bL^2(\Omega)\times\bbLskew,
\end{array}
\end{equation}
and
\begin{equation*}
\bS(\wh \bsigma,\wh p):=\big(S_1(\wh \bsigma,\wh p),S_2(\wh \bsigma,\wh p) \big)=(\bvarphi,p) \quad \forall\, (\wh \bsigma,\wh p)\in\bbH_\rN(\bdiv;\Omega)\times\rW,
\end{equation*}
where $(\bvarphi,p)$ is such that
\begin{equation}\label{eq:weak2}
\begin{array}{rllll}
\ds \wt a_{\wh\bsigma,\wh p}(\bvarphi,\bpsi)  & +\quad \ds \wt b(\bpsi,p)  &=\ds \wt H(\bpsi)&\forall\, \bpsi\in\bH_\rN(\div;\Omega),\\ [1ex]
\ds \wt b(\bvarphi,q)  &-\quad\wt c(p,q) &  =\ds \cblue{\tilde{F}}(q) + c(\wh\bsigma,q)&\forall\, q\in\rL^2(\Omega) .
\end{array}
\end{equation}

By virtue of the above, by defining the operator $\bT:\rW\subseteq\rL^2(\Omega)\to \rL^2(\Omega)$ as
\begin{equation}\label{eq:def-T}
\bT(\wh p):=S_2(R_1(\wh p),\wh p),   
\end{equation}
it is clear that $(\bsigma,\bu,\brho,\bvarphi,p)$ is a solution to \eqref{eq:weak-formulation} if and only if $p\in\rW$ solves the fixed-point problem 
\begin{equation}\label{eq:fixed-point-problem}
\bT(p)=p.    
\end{equation}
Thus, in what follows, we focus on proving the unique solvability of \eqref{eq:fixed-point-problem}. 
%%%%%%%%%%%%%%%%%%%%%%%%%%%%%%%%%%%%%%%%%%%%%%%%%%%%%%%%%%%%%%%%
\subsection{Well‑definedness of \texorpdfstring{$\bT$}{Lg}}
From the definition of $\bT$ in  \eqref{eq:def-T} it is evident that its well-definedness requires the well-posedness of problems \eqref{eq:weak1} and \eqref{eq:weak2}. We begin by analysing that of \eqref{eq:weak1}.
\begin{lemma}\label{lem:well-def-R}
Let	$\wh p \in \rW$ (cf. \eqref{eq:set-W}). Then, there exists a unique $(\bsigma,(\bu,\brho)) \in \bbH_\rN(\bdiv;\Omega)\times \bL^2(\Omega)\times\bbLskew$ solution to \eqref{eq:weak1}.
In addition, there exist $C_1,\,C_2>0$, such that
\begin{equation}\label{eq:dep-R}
\begin{array}{cc}
\|\bsigma\|_{\bdiv;\Omega} \leq C_1\, (\|\bu_\rD\|_{1/2,\Gamma_\rD} + \|\fb\|_{0,\Omega}) + \dfrac{1}{c_a}\,\gamma\,\|\wh p\|_{0,\Omega},\\[1ex]
\|\bu\|_{0,\Omega} + \|\brho\|_{0,\Omega} \leq C_2\, (\|\bu_\rD\|_{1/2,\Gamma_\rD} + \|\fb\|_{0,\Omega}) + \dfrac{1}{\beta}\Big(1 + \dfrac{1}{\mu c_a}  \Big)\,\gamma\,\|\wh p\|_{0,\Omega}.
\end{array}
\end{equation}
\end{lemma}
\begin{proof}
It is a direct consequence of the Babu\v ska--Brezzi theory \cite[Th. 2.34]{ernguermond}, using  \eqref{eq:ellipticity-a} and \eqref{eq:infsup-b}  with
\begin{equation}\label{eq:constant-C1}
C_1:=\Big(\dfrac{1}{c_a} + \dfrac{1}{\beta}  \Big)\Big(1 + \dfrac{1}{\mu c_a} \Big) \qan C_2:= \dfrac{1}{\beta}\Big(1 + \dfrac{1}{\mu c_a}  \Big)\Big(1 + \dfrac{1}{\mu \beta} \Big);
\end{equation}
we omit further details.
\end{proof}

Next, we provide the well-definedness of $\bS$, or equivalently, the well-posedness of \eqref{eq:weak2}.
\begin{lemma}\label{lem:well-def-S}
Let	$(\wh \bsigma,\wh p)\in\bbH_\rN(\bdiv;\Omega)\times\rW$. Then, there exists a unique $(\bvarphi,p) \in \bH_\rN(\div;\Omega)\times\rL^2(\Omega)$ solution to \eqref{eq:weak2}.
In addition, there exist $\wt C>0$ such that
\begin{equation}\label{eq:dep-S}
\|\bvarphi\|_{\div;\Omega} + \|p\|_{0,\Omega} \leq \wt C\, (\|g\|_{0,\Omega} + \|p_{\rD}\|_{1/2,\Gamma_\rD} + \gamma\|\wh\bsigma\|_{\bdiv;\Omega} ).
\end{equation}
\end{lemma}
\begin{proof}
The existence of a unique solution $(\bvarphi,p)$ to \eqref{eq:weak2} is straightforward given the properties of the forms $\wt a$, $\wt b$, and $\wt c$. By examining \eqref{eq:ellipticity-hat-a}, \eqref{eq:ellipticity-hat-c}, and \eqref{eq:infsup-hat-b}, we can confirm that the assumptions of \cite[Th.~3.4]{cg2022} are satisfied. In addition, $\bvarphi$ and $p$ satisfy the following bounds 
\begin{align*}%\label{eq:dep-varphi}
\|\bvarphi\|_{\div;\Omega} &\leq \Big( \dfrac{1}{\kappa_1} + \wt C_1 + \wt C_1\sqrt{\wt\gamma} \Big)\Big(2\max\{\wt C_2,\,\wt C_3\} \Big)^{1/2} (\|g\|_{0,\Omega} + \|p_{\rD}\|_{1/2,\Gamma_\rD} + \gamma\,\|\wh\bsigma\|_{\bdiv;\Omega} ),\\
%\label{eq:dep-p}
\|p\|_{0,\Omega} &\leq \dfrac{1}{\wt\beta}\Big(1 + \kappa_2 r\,\big(\dfrac{1}{\kappa_1} + \wt C_1 + \wt C_1\sqrt{\wt\gamma}\big) \big(2\max\{\wt C_2,\,\wt C_3\} \big)^{1/2}\Big) (\|g\|_{0,\Omega} + \|p_{\rD}\|_{1/2,\Gamma_\rD} + \gamma\,\|\wh\bsigma\|_{\bdiv;\Omega} ),
\end{align*}
%\end{subequations}
where
\begin{equation*}%\label{eq:constant-hat-C1}
\ds \wt C_1:=\dfrac{1}{\wt\beta}\big(1 + \dfrac{C_{\wt a}}{\kappa_1}\big),\quad \wt C_2:= \dfrac{1}{\kappa_1} + \wt C_1 + \wt\gamma \, \wt C_1^2  \qan \wt C_3:= \wt C_1 \Big( 1 + \dfrac{C_{\wt a}}{\wt \beta} + \wt C_1 \dfrac{C_{\wt a}^2\, \wt\gamma}{\wt \beta^2}       \Big);
\end{equation*}
and the above implies \eqref{eq:dep-S} \revY{with 
\begin{equation}
 \label{eq:Ctilde}
 \wt C: = \Big( \dfrac{1}{\kappa_1} + \wt C_1 + \wt C_1\sqrt{\wt\gamma} \Big)\Big(2\max\{\wt C_2,\,\wt C_3\} \Big)^{1/2} + \dfrac{1}{\wt\beta}\Big(1 + \kappa_2 r\,\big(\dfrac{1}{\kappa_1} + \wt C_1 + \wt C_1\sqrt{\wt\gamma}\big) \big(2\max\{\wt C_2,\,\wt C_3\} \big)^{1/2}\Big).
\end{equation}
}
We leave out additional minor details.
\end{proof}
\begin{lemma}\label{lem:well-def-T}
Given $r>0$, let us assume that 
\begin{equation}\label{eq:assumption-T}
\wt C \big(1 +\,\gamma\,C_1 \big)\big( 
 \|g\|_{0,\Omega}+ \|p_{\rD}\|_{1/2,\Gamma_\rD} + \|\bu_\rD\|_{1/2,\Gamma_\rD} + \|\fb\|_{0,\Omega}   \big) + \dfrac{\wt C}{ c_a} \,\gamma^2 r\leq r,
\end{equation}
where $C_1$, \revY{$\widetilde{C}$}, and $\gamma$ are defined in \eqref{eq:constant-C1}, \revY{\eqref{eq:Ctilde}} and \eqref{eq:def-gamma}, respectively. Then, for a given $\widehat{p} \in \rW$ (cf. \eqref{eq:set-W}), there exists a unique $p \in \rW$ such that $\bT(\widehat{p}) = p$.
\end{lemma}
\begin{proof}
From Lemmas \ref{lem:well-def-R} and \ref{lem:well-def-S}, we ascertain that the operators $\bR$ and $\bS$, respectively, are well-defined, thereby ensuring the well-definition of $\bT$. Furthermore, from  \eqref{eq:dep-R} and \eqref{eq:dep-S}, for each $\wh p \in\rW$, we deduce that
\begin{align*}
\|\bT(\wh p)\|_{0,\Omega} & = \|S_2(R_1(\wh p),\wh p)\|_{0,\Omega}\\
& \leq \wt C\, (\|g\|_{0,\Omega} + \|p_{\rD}\|_{1/2,\Gamma_\rD} ) + \wt C\,\gamma\,\|(R_1(\wh p)\|_{\bdiv;\Omega}\\
& \leq \wt C\, (\|g\|_{0,\Omega} + \|p_{\rD}\|_{1/2,\Gamma_\rD} ) + \wt C\,\gamma\,C_1\, (\|\bu_\rD\|_{1/2,\Gamma_\rD} + \|\fb\|_{0,\Omega}) + \dfrac{\wt C}{c_a}\,\gamma^2\|\wh p\|_{0,\Omega},
\end{align*}
this, combined with assumption \eqref{eq:assumption-T}, implies $\bT(\rW)\subseteq\rW$, which concludes the proof.
\end{proof}
\begin{remark}
\revY{Note that the argument used in Lemma \ref{lem:well-def-T} is as follows, for an arbitrary but fixed $r>0$, we define the ball $\rW$ (cf. \eqref{eq:set-W}). Then, for this fixed $r$, we assume that the data $\fb$, $g$, $\bu_\rD$, $p_{\rD}$, and $\gamma$ (cf. \eqref{eq:def-gamma}) are sufficiently small to satisfy Hypothesis \eqref{eq:assumption-T}.}
\end{remark} 
\begin{remark}
Another option for defining the operator $\bS$ (see \eqref{eq:weak2}) is to introduce the perturbation $\wt c$ on the right-hand side of the system, given by 
\begin{equation*}
\begin{array}{rllll}
\ds \wt a_{\wh\bsigma,\wh p}(\bvarphi,\bpsi)  & +\quad \ds \wt b(\bpsi,p)  &=\ds \wt H(\bpsi)&\forall\, \bpsi\in\bH_\rN(\div;\Omega),\\ [1ex]
\ds \wt b(\bvarphi,q)  & &  =\ds \cblue{\tilde{F}}(q) + c(\wh\bsigma,q) + \wt c(\wh p,q)&\forall\, q\in\rL^2(\Omega).
\end{array}
\end{equation*}
But in this case, the assumption of small data in \eqref{eq:assumption-T} (as well as in other instances, later on) would also involve the storativity parameter $c_0$, making the analysis slightly more restrictive.
\end{remark}
%%%%%%%%%%%%%%%%%%%%%%%%%%%%%%%%%%%%%%%%%%%%%%%%%%%%%%%%%%%%%%%%
\subsection{Existence and uniqueness of weak solution}\label{sec:wellp} 
%Here, we provide the main result of this section, namely, the existence and uniqueness of solution of the nonlinear problem \eqref{eq:weak-formulation}. 
We begin by establishing two lemmas deriving conditions under which the operator $\bT$ is a contraction.
\begin{lemma}\label{lem:cotaR1}
Given $\wh p_1,\,\wh p_2, \in \rW$, the following estimate holds
\begin{equation}\label{eq:Lipschitz-continuity-R1}
\|R_1(\wh p_1) - R_1(\wh p_2)\|_{\bdiv;\Omega} \leq \dfrac{1}{c_a} \,\gamma\, \|\wh p_1- \wh p_2\|_{0,\Omega}.
\end{equation}
\end{lemma}
\begin{proof}
Let $(\bsigma_1,(\bu_1,\brho_1)),\,(\bsigma_2,(\bu_2,\brho_2)) \in \bbH_\rN(\bdiv;\Omega)\times(\bL^2(\Omega)\times\bbLskew)$, such that $\bR(\wh p_1)=(\bsigma_1,(\bu_1,\brho_1))$ ad $\bR(\wh p_2)=(\bsigma_2,(\bu_2,\brho_2))$. Then, from the definition of $\bR$ (cf. \eqref{eq:weak1}), we have
\begin{equation}\label{eq:auxiliar-eq-1}
\begin{array}{rllll}
\ds a(\bsigma_1-\bsigma_2,\btau)  & +  \ds b(\btau,(\bu_1-\bu_2,\brho_1-\brho_2))  &= \ds - c(\btau,\wh p_1- \wh p_2)&\forall\, \btau\in\bbH_\rN(\bdiv;\Omega),\\ [1ex]
\ds b(\bsigma_1-\bsigma_2,(\bv,\bbeta))  &  &  =\ds 0&\forall\,(\bv,\bbeta)\in\bL^2(\Omega)\times\bbLskew.
\end{array}
\end{equation}
Since $\bsigma_1-\bsigma_2 \in\bbV$ (cf. \eqref{eq:kernel-b}), taking $\btau = \bsigma_1-\bsigma_2$ in \eqref{eq:auxiliar-eq-1}, and utilising the ellipticity of $a$ on $\bbV$ (cf. \eqref{eq:ellipticity-a}) along with the bound of $c$ (cf. \eqref{eq:bounded-c}), we obtain:
\begin{equation*}
\ds c_a \|\bsigma_1-\bsigma_2\|_{\bdiv;\Omega}^2 \leq a(\bsigma_1-\bsigma_2,\bsigma_1-\bsigma_2)  =  - c(\bsigma_1-\bsigma_2,\wh p_1- \wh p_2) \leq \,\gamma\,\|\bsigma_1-\bsigma_2\|_{\bdiv;\Omega}\|\wh p_1- \wh p_2\|_{0,\Omega},
\end{equation*}
which concludes the proof.
\end{proof}
\begin{lemma}\label{lem:cotaS-1-2}
Given $(\wh\bsigma_1,\,\wh p_1),\,(\wh\bsigma_2,\,\wh p_2), \in \bbH_\rN(\bdiv;\Omega)\times\rW$, the following estimate holds
\begin{equation}\label{eq:Lipschitz-continuity-S-1-2}
\begin{array}{ll}
\|S_2(\wh\bsigma_1,\,\wh p_1) - S_2(\wh\bsigma_2,\,\wh p_2)\|_{0,\Omega}\\[1ex]
\ds\leq \dfrac{2\kappa_2\,  \wt C}{\min\{\wt\gamma,\,\kappa_1\}}\, \big(\|g\|_{0,\Omega} + \|p_{\rD}\|_{1/2,\Gamma_\rD} + \gamma\|\wh\bsigma_2\|_{\bdiv;\Omega} \big) \|\wh p_1 - \wh p_2\|_{0,\Omega} + \dfrac{2}{\min\{\wt\gamma,\,\kappa_1\}}\,\gamma\,\|\wh\bsigma_1-\wh\bsigma_2\|_{\bdiv;\Omega}.
\end{array}
\end{equation}
\end{lemma}
\revX{Note 
that $\min\{\tilde\gamma, \kappa_1\}$ enters in the denominator of the estimate, but 
 we have assumed that $\kappa_1$ is strictly positive and it does not degenerate to zero.}
\begin{proof}
Let $(\bvarphi_1,p_1),\,(\bvarphi_2,p_2)\in \bH_\rN(\div;\Omega)\times\rL^2(\Omega)$, such that $\bS(\wh\bsigma_1,\wh p_1)=(\bvarphi_1,p_1)$ and $\bS(\wh\bsigma_2,\wh p_2)=(\bvarphi_2, p_2)$. Then, from the definition of $\bS$ (cf. \eqref{eq:weak2}), and employing similar arguments to those in Lemma~\ref{lem:cotaR1}, we have
\begin{equation*}
\ds \wt a_{\wh\bsigma_1,\wh p_1}(\bvarphi_1,\bvarphi_1-\bvarphi_2) - \wt a_{\wh\bsigma_2,\wh p_2}(\bvarphi_2,\bvarphi_1-\bvarphi_2) + \wt c(p_1-p_2,p_1-p_2) = - c(\wh\bsigma_1-\wh\bsigma_2,p_1-p_2),
\end{equation*}
by adding $\pm \wt a_{\wh\bsigma_1,\wh p_1}(\bvarphi_2,\bvarphi_1-\bvarphi_2)$ in the last equation, we obtain
\begin{align*}%\label{eq:auxiliar-eq-2}
%\begin{array}{ll}
 &\wt a_{\wh\bsigma_1,\wh p_1}(\bvarphi_1-\bvarphi_2,\bvarphi_1-\bvarphi_2) + \wt c(p_1-p_2,p_1-p_2)%\\[1ex]
 \\
&\qquad =  \wt a_{\wh\bsigma_2,\wh p_2}(\bvarphi_2,\bvarphi_1-\bvarphi_2) - \wt a_{\wh\bsigma_1,\wh p_1}(\bvarphi_2,\bvarphi_1-\bvarphi_2) - c(\wh\bsigma_1-\wh\bsigma_2,p_1-p_2).
%\end{array}
\end{align*}
Then, using the the first assumption for $\bkappa$ (cf. \eqref{prop-kappa}), the ellipticity of $\wt c$ (see \eqref{eq:ellipticity-hat-c}), the definition of $\wt a_{\wh\bsigma,\wh p}$ (cf. \eqref{eq:forms}) and the continuity of the form $c$ (see \eqref{eq:bounded-c}), we deduce
\begin{equation*}
\begin{array}{ll}
\ds \kappa_1 \|\bvarphi_1-\bvarphi_2\|_{0,\Omega}^2 + \wt\gamma \|p_1-p_2\|_{0,\Omega}^2 \leq \wt a_{\wh\bsigma_1,\wh p_1}(\bvarphi_1-\bvarphi_2,\bvarphi_1-\bvarphi_2) + \wt c(p_1-p_2,p_1-p_2)\\[1ex]
\ds \qquad = \int_{\Omega} (\bkappa(\wh \bsigma_2,\wh p_2)^{-1}-\bkappa(\wh \bsigma_1,\wh p_1)^{-1})\, \bvarphi_2 \cdot (\bvarphi_1-\bvarphi_2) - c(\wh\bsigma_1-\wh\bsigma_2,p_1-p_2)\\[2.5ex]
\ds \qquad \leq \|\bkappa(\wh \bsigma_2,\wh p_2)^{-1}- \bkappa(\wh \bsigma_1,\wh p_1)^{-1}\|_{\bbL^\infty(\Omega)} \, \|\bvarphi_2\|_{0,\Omega} \|\bvarphi_1-\bvarphi_2\|_{0,\Omega} + \,\gamma\,\|\wh\bsigma_1-\wh\bsigma_2\|_{\bdiv;\Omega}\|p_1-p_2\|_{0,\Omega}.
\end{array}
\end{equation*}
From the last equation, by utilising the second assumption regarding $\bkappa$ (see \eqref{prop-kappa}), we obtain
\begin{equation*}
\begin{array}{ll}
\frac12\min\{\wt\gamma,\kappa_1\}(\|\bvarphi_1-\bvarphi_2\|_{0,\Omega} + \|p_1-p_2\|_{0,\Omega})^2 \leq \kappa_1 \|\bvarphi_1-\bvarphi_2\|_{0,\Omega}^2 + \wt\gamma \|p_1-p_2\|_{0,\Omega}^2\\[1ex] 
\qquad\leq  \kappa_2\|\wh p_2-\wh p_1\|_{0,\Omega}\|\bvarphi_2\|_{0,\Omega} \|\bvarphi_1-\bvarphi_2\|_{0,\Omega} + \,\gamma\,\|\wh\bsigma_1-\wh\bsigma_2\|_{\bdiv;\Omega}\|p_1-p_2\|_{0,\Omega}\\[1ex]
\ds \qquad \leq \big(\kappa_2\|\wh p_2-\wh p_1\|_{0,\Omega}\|\bvarphi_2\|_{0,\Omega}  + \,\gamma\,\|\wh\bsigma_1-\wh\bsigma_2\|_{\bdiv;\Omega}\big) \big(\|\bvarphi_1-\bvarphi_2\|_{0,\Omega} + \|p_1-p_2\|_{0,\Omega}\big),
\end{array}
\end{equation*}
the last, together with the fact that $\bvarphi_2$ satisfies \eqref{eq:dep-S}, leads to the following bound 
\begin{equation*}
\begin{array}{ll}
\frac12\min\{\wt\gamma,\kappa_1\}(\|\bvarphi_1-\bvarphi_2\|_{0,\Omega} + \|p_1-p_2\|_{0,\Omega}) \leq \kappa_2\|\wh p_2-\wh p_1\|_{0,\Omega}\|\bvarphi_2\|_{0,\Omega}  + \,\gamma\,\|\wh\bsigma_1-\wh\bsigma_2\|_{\bdiv;\Omega}\\[1ex]
\ds \qquad \leq  \kappa_2\|\wh p_2-\wh p_1\|_{0,\Omega}\wt C\, (\|g\|_{0,\Omega} + \|p_{\rD}\|_{1/2,\Gamma_\rD} + \gamma\|\wh\bsigma_2\|_{\bdiv;\Omega} )  + \,\gamma\,\|\wh\bsigma_1-\wh\bsigma_2\|_{\bdiv;\Omega},
\end{array}
\end{equation*}
and this yields \eqref{eq:Lipschitz-continuity-S-1-2}, concluding the proof.
\end{proof}

The following theorem presents the main result of this section, establishing the existence and uniqueness of the solution to the fixed-point problem \eqref{eq:fixed-point-problem}, or equivalently, of \eqref{eq:weak-formulation}.
\begin{theorem}\label{theorem:unique-solution-weak1}
Given $r>0$, assume that $\fb \in \bL^2(\Omega)$, $g \in L^2(\Omega)$, $\bu_\rD \in \bH^{1/2}(\Gamma_\rD)$, $p_{\rD} \in \rH^{1/2}(\Gamma_\rD)$ and $\gamma$ satisfies
\begin{equation}\label{eq:assumption-T-2}
\begin{array}{cc}
\dfrac{2 \max\{1,\,\kappa_2\}  }{\min\{\wt\gamma,\,\kappa_1\,, r\}}\, \bigg\{  \wt C (1 + C_1\gamma)(\|g\|_{0,\Omega} + \|p_{\rD}\|_{1/2,\Gamma_\rD} + \|\bu_\rD\|_{1/2,\Gamma_\rD} + \|\fb\|_{0,\Omega})+\dfrac{\gamma^2}{c_a}(\dfrac{1}{\kappa_2} +\wt C \, r)\bigg\}   < 1.
 \end{array}
\end{equation}
Then, $\bT$ (cf. \eqref{eq:def-T}) has a unique fixed point $p\in\rW$. Equivalently,  \eqref{eq:weak-formulation} has a unique solution $(\bsigma,\bu,\brho,\bvarphi,p)\in \bbH_\rN(\bdiv;\Omega)\times\bL^2(\Omega)\times\bbLskew\times \bH_\rN(\div;\Omega)\times\rW$. In addition, 
there exists $C>0$, such that
\begin{align}\label{eq:stability}
 \nonumber &\|\bsigma\|_{\bdiv;\Omega} + \|\bu\|_{0,\Omega} + \|\brho\|_{0,\Omega} + \|\bvarphi\|_{\div;\Omega} + \|p\|_{0,\Omega} \\
 &\qquad \leq C (\|g\|_{0,\Omega} + \|p_{\rD}\|_{1/2,\Gamma_\rD} + \|\bu_\rD\|_{1/2,\Gamma_\rD} + \|\fb\|_{0,\Omega} + \gamma \,r\,).
\end{align}
\end{theorem}
\begin{proof}
%We begin by recalling from the previous analysis that assumption 
Recall that \eqref{eq:assumption-T-2} ensures the well-definedness of $\bT$. Let $\wh p_1,\,\wh p_2,\, p_1,\, p_2 \in\rW$, such that $\bT(\wh p_1)=p_1$ and $\bT(\wh p_2)=p_2$. From the definition of $\bT$ (see \eqref{eq:def-T}), and the estimates \eqref{eq:Lipschitz-continuity-S-1-2} and \eqref{eq:Lipschitz-continuity-R1}, we deduce
\begin{align*}
& \|p_1-p_2\|_{0,\Omega} = \|\bT(\wh p_1)-\bT(\wh p_2)\|_{0,\Omega} = \|S_2(R_1(\wh p_1),\wh p_1)-S_2(R_1(\wh p_2),\wh p_2)\|_{0,\Omega}\\
&\leq \dfrac{2\kappa_2\,  \wt C}{\min\{\wt\gamma,\,\kappa_1\}}\, \big(\|g\|_{0,\Omega} + \|p_{\rD}\|_{1/2,\Gamma_\rD} + \gamma\|R_1(\wh p_2)\|_{\bdiv;\Omega} \big) \|\wh p_1 - \wh p_2\|_{0,\Omega} \\
&\qquad + \dfrac{2}{\min\{\wt\gamma,\,\kappa_1\}}\,\gamma\,\|R_1(\wh p_1)-R_1(\wh p_2)\|_{\bdiv;\Omega}\\
&\leq \dfrac{2\kappa_2\,  \wt C}{\min\{\wt\gamma,\,\kappa_1\}}\, \big(\|g\|_{0,\Omega} + \|p_{\rD}\|_{1/2,\Gamma_\rD}\big)\|\wh p_1 - \wh p_2\|_{0,\Omega} + \dfrac{2\kappa_2\,  \wt C}{\min\{\wt\gamma,\,\kappa_1\}}\gamma\|R_1(\wh p_2)\|_{\bdiv;\Omega} \|\wh p_1 - \wh p_2\|_{0,\Omega}\\
& \qquad+ \dfrac{2}{c_a\,\min\{\wt\gamma,\,\kappa_1\}}\,\gamma^2\, \|\wh p_1- \wh p_2\|_{0,\Omega},
\end{align*}
the above, along with the fact that $R_1(\wh p_2)$ satisfies \eqref{eq:dep-R} and $\wh p_2\in\rW$, implies
\begin{align*}
 \|p_1-p_2\|_{0,\Omega} & \leq \dfrac{2\kappa_2\,  \wt C}{\min\{\wt\gamma,\,\kappa_1\}}\, \big(\|g\|_{0,\Omega} + \|p_{\rD}\|_{1/2,\Gamma_\rD}\big)\|\wh p_1 - \wh p_2\|_{0,\Omega}+ \dfrac{2}{c_a\,\min\{\wt\gamma,\,\kappa_1\}}\,\gamma^2\, \|\wh p_1- \wh p_2\|_{0,\Omega}\\
 & \qquad +\dfrac{2\kappa_2\,  \wt C}{\min\{\wt\gamma,\,\kappa_1\}}\gamma\Big( C_1\, (\|\bu_\rD\|_{1/2,\Gamma_\rD} + \|\fb\|_{0,\Omega}) + \dfrac{1}{c_a}\,\gamma\,r  \Big) \|\wh p_1 - \wh p_2\|_{0,\Omega}\\
&  \leq \dfrac{2}{\min\{\wt\gamma,\,\kappa_1\}}\, \bigg\{ \kappa_2\,  \wt C (1 + C_1\gamma)(\|g\|_{0,\Omega} + \|p_{\rD}\|_{1/2,\Gamma_\rD} + \|\bu_\rD\|_{1/2,\Gamma_\rD} + \|\fb\|_{0,\Omega})\\
&\qquad \qquad \qquad \qquad +\dfrac{\gamma^2}{c_a}(1+\kappa_2\,\wt C \, r)\bigg\}\|\wh p_1- \wh p_2\|_{0,\Omega}, 
\end{align*}
which together with \eqref{eq:assumption-T-2} and the Banach fixed-point theorem yields that $\bT$ has a unique fixed point in $\rW$.
Finally,  \eqref{eq:stability}  is derived analogously to the estimates in \eqref{eq:dep-R} and \eqref{eq:dep-S}, which completes the proof.
\end{proof}
\begin{remark}
The operator $\bT$ (see \eqref{eq:def-T}) could be also defined, for example $\bT:\bW\to\bW$, with $ \bW := \Big\{ (\wh\bsigma,\wh p) \in\bbH_\rN(\bdiv;\Omega)\times\rL^2(\Omega) \,:\quad \|\wh\bsigma\|_{\bdiv;\Omega} + \|\wh p\|_{0,\Omega} \leq r \Big\}$ and $\bT(\wh\bsigma,\wh p):=(R_1(\wh p),S_2(\wh \bsigma,\wh p))=(\bsigma,p)$, with $R_1$ and $S_2$ defined as in \eqref{eq:weak1} and \eqref{eq:weak2}, respectively.
\end{remark}
%%%%%%%%%%%%%%%%%%%%%%%%%%%%%%%%%%%
\section{Finite element discretisation}\label{sec:FE}
In this section, we present and analyse the Galerkin scheme for problem \eqref{eq:weak-formulation}. It is worth mentioning upfront that the well-posedness analysis can be straightforwardly extended from the continuous problem to the discrete case. Therefore, we omit many of the details.

%%%%%%%%%%%%%%%%%%%%%%%%%%%%%%%%%%%%%%%%%%%%%%%%%%%%%%%%%%%%%%%%%%%%%%%%%
\subsection{Finite element spaces and Galerkin scheme}\label{sec:spaces}
Let us consider a  regular partition $\mathcal T_h$ of $\bar\Omega$ made up of triangles $K$ (in $\mathbb{R}^2$) or tetrahedra $K$ (in $\mathbb{R}^3$) of diameter $h_K$, and denote the mesh size by $h := \max\{ h_K: \ K \in \mathcal T_h\}$. 
%We will begin by defining finite-dimensional subspaces of the functional spaces encountered previously. 
%
Given an integer $\ell \ge 0$ and $K \in \mathcal{T}_h$, we first let $\mathrm{P}_\ell(K)$ be the space of polynomials of degree $\leq \ell$ defined on $K$, whose vector and tensor versions are denoted $\bP_\ell(K) \,:=\, [\mathrm{P}_\ell(K)]^d$ and $\mathbb{P}_\ell(K)
\,=\,[\mathrm{P}_\ell(K)]^{d \times d}$, respectively. Also, we let $\mathbf{RT}_\ell(K) \,:=\, \bP_\ell(K) \oplus \mathrm{P}_\ell(K)\,\bx$ be the local Raviart--Thomas space of order $\ell$ defined on $K$, where $\bx$ stands for a generic vector in $\bbR^d$, and denote by $\mathbb{RT}_k(K)$ the tensor-valued counterpart of this space. 

For each $K\in \mathcal{T}_h$ we consider the bubble space of order $k$, defined as 
\[\mathbf{B}_k(K):=
 \begin{cases}
\revX{\mathrm{curl}(b_K)\mathrm{P}_k(K)}&\textrm{in}\quad\mathbb{R}^2,\\
\nabla \times (b_K\mathbf{P}_k(K))&\textrm{in} \quad \mathbb{R}^3,
\end{cases}\]
where $b_K$ is a suitably normalised cubic polynomial on $K$, which vanishes on the boundary of $K$ (see \cite{ernguermond}).

We recall the classical PEERS$_k$ elements (cf. \cite{arnold84}) to define the discrete subspaces for the stress tensor $\bsigma$, the displacement $\bu$, and the rotation tensor $\brho$
\begin{align}\label{fe:peers}
\bbH^{\bsigma}_{h}&:=\left\{{\btau}_h \in \bbH_\rN(\bdiv;\Omega):\quad {\btau}_{h}|_{K}\in \mathbb{RT}_k(K)\oplus[\mathbf{B}_k(K)]^d\quad \forall\, K\in \mathcal{T}_h\right\},\nonumber \\
\bH^{\bu}_{h}&:= \left\{\bv_h \in \bL^2(\Omega): \quad \bv_h|_{K}\in \textbf{P}_{k}(K)\quad \forall\, K\in \mathcal{T}_h\right\},\\
\bbH^{\brho}_{h}&:=\left\{\bbeta_h\in \bbLskew\cap \mathbb{C}(\overline{\Omega})\quad \mathrm{and}\quad \bbeta_h|_{K}\in \mathbb{P}_{k+1}(K)\quad \forall\, K\in \mathcal{T}_h\right\},\nonumber
\end{align}
\revX{where $\mathbb{C}(\overline{\Omega})$ denotes the space of continuous tensor fields}, 
and the following estimates are proven for the  PEERS$_k$ elements (cf. \cite[Remark~3.3]{LV2004})
\begin{subequations}
\begin{align}\label{eq:infsup-b-h}
\sup_{\bzero\neq\btau_h\in \bbH^{\bsigma}_{h} } \frac{b(\btau_h,(\bv_h,\bbeta_h))}{\|\btau_h\|_{\bdiv;\Omega}} & \geq \beta^*(\|\bv_h\|_{0,\Omega} +\|\bbeta_h\|_{0,\Omega})
\quad \forall\,(\bv_h,\bbeta_h)\in \bH^{\bu}_{h}\times\bbH^{\brho}_{h},\\
\label{eq:ellipticity-a-h}
\ds a(\btau_h,\btau_h) & \geq c_a \|\btau_h\|_{\bdiv;\Omega}^2 \quad \forall\,\btau_h\in\bbV_h,
\end{align}\end{subequations}
where $\bbV_h$ denotes the discrete kernel of $b$, that is
\begin{equation*}%\label{eq:kernel-b-h}
 \ds \bbV_h   :=  \ds \Big\{\btau_h\in \bbH^{\bsigma}_{h} :\quad b(\btau_h,(\bv_h,\bbeta_h)) = 0 \quad \forall\,(\bv_h,\bbeta_h)\in \bH^{\bu}_{h}\times\bbH^{\brho}_{h} \Big\}.
\end{equation*}

Additionally, for $\bvarphi$ and the pressure $p$, we define the FE subspaces
\begin{align}\label{fe:weak-2}
\bH^{\bvarphi}_{h}&:=\left\{\bpsi_h\in \bH_\rN(\div;\Omega): \  \bpsi_h|_{K}\in \mathbf{RT}_k(K)\quad \forall\, K\in \mathcal{T}_h\right\},\nonumber\\
\rH^{p}_{h}&:= \left\{q_h \in \rL^2(\Omega): \  q_h|_{K}\in \mathrm{P}_{k}(K)\quad \forall\, K\in \mathcal{T}_h\right\},
\end{align}
and it is well known that $\wt b$ satisfies the inf-sup condition (see, e.g.,  \cite[Lemma~4.6]{cdgo2020})
\begin{equation}\label{eq:infsup-hat-b-h}
\ds \sup_{\bzero\neq \bpsi_h\in \bH^{\bvarphi}_{h}}  \frac{\wt b(\bpsi_h,q_h)}{\|\bpsi_h\|_{\div;\Omega}} \geq \wt\beta^* \|q_h\|_{0,\Omega}
\quad \forall\,q_h\in \rH^{p}_{h}. 
\end{equation}
Note that it is of course possible to consider other conforming and inf-sup stable spaces such as Arnold--Falk--Winther and Brezzi--Douglas--Marini instead of \eqref{fe:peers} and \eqref{fe:weak-2}, respectively.   
The Galerkin scheme for   \eqref{eq:weak-formulation} reads: find $(\bsigma_h,\bu_h,\brho_h,\bvarphi_h,p_h)\in  \bbH^{\bsigma}_{h}\times\bH^{\bu}_{h}\times\bbH^{\brho}_{h}\times \bH^{\bvarphi}_{h}\times\rH^{p}_{h}$, such that:
\begin{equation}\label{eq:weak-formulation-h}
\begin{array}{rllll}
\ds a(\bsigma_h,\btau_h)  & +\quad \ds b(\btau_h,(\bu_h,\brho_h))  &+\quad c(\btau_h,p_h)&= \ds H(\btau_h)&\forall\, \btau_h\in\bbH^{\bsigma}_{h},\\ [1ex]
\ds b(\bsigma_h,(\bv_h,\bbeta_h))  &  & & =\ds F(\bv_h,\bbeta_h)&\forall(\bv_h,\bbeta_h)\in\bH^{\bu}_{h}\times\bbH^{\brho}_{h} ,\\ [1ex]
\ds \wt a_{\bsigma_h,p_h}(\bvarphi_h,\bpsi_h)  & +\quad \ds \wt b(\bpsi_h,p_h)  & &= \ds \wt H(\bpsi_h)&\forall\, \bpsi_h\in\bH^{\bvarphi}_{h},\\ [1ex]
\ds \wt b(\bvarphi_h,q_h)  &-\quad\wt c(p_h,q_h) & -\quad c(\bsigma_h,q_h) & =\ds \cblue{\tilde{F}}(q_h)&\forall\, q_h\in\rH^{p}_{h}.
\end{array}
\end{equation}

%%%%%%%%%%%%%%%%%%%%%%%%%%%%%%%%%%%%%%%%%%%%%%%%%%%%%%%%%%%%%%%%%
\subsection{Analysis of the discrete problem}\label{sec:Fixed-point-argument-h}
In this section, we analyse the Galerkin scheme \eqref{eq:weak-formulation-h}. It's worth noting that establishing well-posedness can be readily achieved by extending the results derived for the continuous problem to the discrete setting.
Firstly, and similarly to the continuous case, we define the following set
\begin{equation*}%\label{eq:set-W-h}
 \rW_h := \Big\{ \wh p_h \in\rH^{p}_{h} \,:\quad \|\wh p_h\|_{0,\Omega} \leq r \Big\}.
\end{equation*}
Next, for a fixed $\wh p_h$ in $\rW_h$, we have that the bilinear form $\wt a_{\bsigma_h,p_h}$ satisfies
\begin{equation}\label{eq:ellipticity-hat-a-h}
\ds \wt a_{\wh\bsigma_h,\wh p_h}(\bpsi_h,\bpsi_h) \geq \kappa_1 \|\bpsi_h\|_{\div;\Omega}^2 \quad\forall\,\bpsi_h\in\bV_h,
\end{equation}
where $\bV_h$ is the discrete kernel of $\wt b$
\begin{equation*}%\label{eq:kernel-hat-b-h}
 \ds \bV_h   :=  \ds \Big\{\bpsi_h\in \bH^{\bvarphi}_{h} :\quad \wt b(\bpsi_h,q_h) = 0 \quad \forall\,q_h\in \rH^{p}_{h} \Big\}.
\end{equation*}
Additionally, we define the discrete operators $\bR_h:\rW_h\subseteq\rH^{p}_{h}\to \bbH^{\bsigma}_{h}\times(\bH^{\bu}_{h}\times\bbH^{\brho}_{h})$ and $\bS_h: \bbH^{\bsigma}_{h}\times\rW_h \to \bH^{\bvarphi}_{h}\times\rH^{p}_{h}$, respectively, by
\begin{equation*}
\bR_h(\wh p_h):=\big(R_{1,h}(\wh p_h),(R_{2,h}(\wh p_h),R_{3,h}(\wh p_h))\big)=(\bsigma_h,(\bu_h,\brho_h)) \quad \forall\, \wh p_h\in\rW_h,
\end{equation*}
where $(\bsigma_h,(\bu_h,\brho_h))\in \bbH^{\bsigma}_{h}\times(\bH^{\bu}_{h}\times\bbH^{\brho}_{h})$ is the unique solution of 
\begin{equation}\label{eq:weak1-h}
\begin{array}{rllll}
\ds a(\bsigma_h,\btau_h)  & +\quad \ds b(\btau_h,(\bu_h,\brho_h))  &= \ds H(\btau_h) - c(\btau_h,\wh p_h)&\forall\, \btau_h\in\bbH^{\bsigma}_{h},\\ [1ex]
\ds b(\bsigma_h,(\bv_h,\bbeta_h))  &  &  =\ds F(\bv_h,\bbeta_h)&\forall\,(\bv_h,\bbeta_h)\in\bH^{\bu}_{h}\times\bbH^{\brho}_{h},
\end{array}
\end{equation}
and
\begin{equation*}
\bS_h(\wh \bsigma_h,\wh p_h):=\big(S_{1,h}(\wh \bsigma_h,\wh p_h),S_{2,h}(\wh \bsigma_h,\wh p_h) \big)=(\bvarphi_h,p_h) \quad \forall\, (\wh \bsigma_h,\wh p_h)\in\bbH^{\bsigma}_{h}\times\rW_h,
\end{equation*}
where $(\bvarphi_h,p_h)$ is the unique tuple in $\bH^{\bvarphi}_{h}\times\times\rH^{p}_{h}$ such that 
\begin{equation*}%\label{eq:weak2-h}
\begin{array}{rllll}
\ds \wt a_{\wh\bsigma_h,\wh p_h}(\bvarphi_h,\bpsi_h)  & +\quad \ds \wt b(\bpsi_h,p_h)  &= \ds \wt H(\bpsi_h)&\forall\, \bpsi_h\in\bH^{\bvarphi}_{h},\\ [1ex]
\ds \wt b(\bvarphi_h,q_h)  &-\quad\wt c(p_h,q_h)  &  =\ds \cblue{\tilde{F}}(q_h) + c(\wh\bsigma_h,q_h)&\forall\, q_h\in\rH^{p}_{h}.
\end{array}
\end{equation*}

Employing properties \eqref{eq:infsup-b-h}, \eqref{eq:ellipticity-a-h}, \eqref{eq:infsup-hat-b-h}, \eqref{eq:ellipticity-hat-a-h} and \eqref{eq:ellipticity-hat-c} and proceeding exactly as for the continuous case (Lemmas~\ref{lem:well-def-R} and \ref{lem:well-def-S}), it can be
easily deduced that both operators are well-defined. Then, analogously to the continuous case, we define the following fixed-point operator 
\begin{equation}\label{eq:def-T-h}
\bT_h:\rW_h\subseteq\rH^{p}_{h}\to \rH^{p}_{h},\quad \wh p_h\mapsto \bT_h(\wh p_h):=S_{2,h}(R_{1,h}(\wh p_h),\wh p_h), 
\end{equation}
which is clearly well-defined (since $R_h$ and $S_h$ are). Further, it can be easily deduced that $\bT_h(\rW_h)\subseteq\rW_h$ if 
\begin{equation}\label{eq:assumption-T-h}
\wt C^* \big(1 +\,\gamma\,C_1^* \big)\big( 
 \|g\|_{0,\Omega}+ \|p_{\rD}\|_{1/2,\Gamma_\rD} + \|\bu_\rD\|_{1/2,\Gamma_\rD} + \|\fb\|_{0,\Omega}   \big) + \dfrac{\wt C^*}{ c_a} \,\gamma^2 r\leq r,
\end{equation}
where $\wt C^*$ and $C_1^*$ (depending on $c_a$, $\mu$, $\kappa_1$, $\kappa_2$, $C_{\wt a}$, $\beta^*$, $\wt\beta^*$) are the discrete versions of the constants $\wt C$ and $C_1$ (cf. \eqref{eq:dep-R} and \eqref{eq:dep-S}). Finally, it is clear that $(\bsigma_h,\bu_h,\brho_h,\bvarphi_h,p_h)$ is a solution to \eqref{eq:weak-formulation-h} if and only if $p_h$ satisfies 
%$\bT_h(p_h)=p_h$, and consequently, the well-posedness of \eqref{eq:weak-formulation-h} is equivalent to the unique solvability of the fixed-point problem: find $p_h\in\rW_h$ such that
\begin{equation}\label{eq:fixed-point-problem-h}
\bT_h(p_h)=p_h.    
\end{equation}
The main outcome of this section is presented in the following theorem, establishing the existence and uniqueness of a solution to the fixed-point problem \eqref{eq:fixed-point-problem-h},  equivalently proving the well-posedness of problem \eqref{eq:weak-formulation-h}.
\begin{theorem}\label{theorem:unique-solution-weak1-h}
Given $r>0$, assume that 
%$\fb \in \bL^2(\Omega)$, $g \in \rL^2(\Omega)$, $\bu_\rD \in \bH^{1/2}(\Gamma_\rD)$, $p_{\rD} \in \rH^{1/2}(\Gamma_\rD)$ and 
the data and $\gamma$ satisfy
\begin{equation}\label{eq:assumption-T-h-2}
\frac{2 \max\{1,\,\kappa_2\}  }{\min\{\wt\gamma,\,\kappa_1\,, r\}}\, \bigg\{  \wt C^* (1 + C_1^*\gamma)(\|g\|_{0,\Omega} + \|p_{\rD}\|_{1/2,\Gamma_\rD} + \|\bu_\rD\|_{1/2,\Gamma_\rD} + \|\fb\|_{0,\Omega})+\dfrac{\gamma^2}{c_a}\biggl(\frac{1}{\kappa_2} +\wt C^* \, r\biggr)\bigg\}   < 1.
\end{equation}
Then,  $\bT_h$ (cf. \eqref{eq:def-T-h}) has a unique fixed point $p_h\in\rW_h$. Equivalently, problem \eqref{eq:weak-formulation-h} has a unique solution $(\bsigma_h,\bu_h,\brho_h,\bvarphi_h,p_h)\in  \bbH^{\bsigma}_{h}\times\bH^{\bu}_{h}\times\bbH^{\brho}_{h}\times \bH^{\bvarphi}_{h}\times\rW_h$.%\rH^{p}_{h}$ with $p_h\in\rW_h$. 
In addition, 
%we have the following continuous dependence on data. T
there exists $C^*>0$, such that
\begin{align}\label{eq:stability-h}
%\begin{array}{ll}
\nonumber & \ds \|\bsigma_h\|_{\bdiv;\Omega} + \|\bu_h\|_{0,\Omega} + \|\brho_h\|_{0,\Omega} + \|\bvarphi_h\|_{\div;\Omega} + \|p_h\|_{0,\Omega}\\
%\ds \qquad 
&\qquad \leq C^* (\|g\|_{0,\Omega} + \|p_{\rD}\|_{1/2,\Gamma_\rD} + \|\bu_\rD\|_{1/2,\Gamma_\rD} + \|\fb\|_{0,\Omega} + \gamma \,r\,).
%\end{array}
\end{align}
\end{theorem}
\begin{proof}
First, we observe that, similar to the continuous case (as seen in the proof of Theorem~\ref{theorem:unique-solution-weak1}), assumption \eqref{eq:assumption-T-h-2} ensures the well-definedness of $\bT_h$ and that $\bT_h(\rW_h)\subseteq\rW_h$. Now, by adapting the arguments used in Section~\ref{sec:wellp} (cf.  Lemmas~\ref{lem:cotaR1} and \ref{lem:cotaS-1-2}), one can derive the following estimates
%\begin{equation}\label{eq:Lipschitz-continuity-R1-h}
\begin{align*}
\|R_{1,h}(\wh p_1) - R_{1,h}(\wh p_2)\|_{\bdiv;\Omega} 
&\leq \dfrac{1}{c_a} \,\gamma\, \|\wh p_1- \wh p_2\|_{0,\Omega},\\
\|S_{2,h}(\wh\bsigma_1,\,\wh p_1) - S_{2,h}(\wh\bsigma_2,\,\wh p_2)\|_{0,\Omega} & \leq \frac{2\kappa_2\,  \wt C^*}{\min\{\wt\gamma,\,\kappa_1\}}\, \big(\|g\|_{0,\Omega} + \|p_{\rD}\|_{1/2,\Gamma_\rD} + \gamma\|\wh\bsigma_2\|_{\bdiv;\Omega} \big) \|\wh p_1 - \wh p_2\|_{0,\Omega} \\
&\quad + \frac{2}{\min\{\wt\gamma,\,\kappa_1\}}\,\gamma\,\|\wh\bsigma_1-\wh\bsigma_2\|_{\bdiv;\Omega},
\end{align*}
for all $\wh p_1,\, \wh p_2\in\rW_h$ and $\wh\bsigma_1,\,\wh\bsigma_2 \in \bbH^{\bsigma}_{h}$, which together with the definition of $\bT_h$ (see \eqref{eq:def-T-h}), yield
\begin{align*}
 \|\bT_h(\wh p_1)-\bT_h(\wh p_2)\|_{0,\Omega} &\leq \dfrac{2}{\min\{\wt\gamma,\,\kappa_1\}} \bigg\{ \kappa_2\,  \wt C^* (1 + C_1^*\gamma)(\|g\|_{0,\Omega} + \|p_{\rD}\|_{1/2,\Gamma_\rD} + \|\bu_\rD\|_{1/2,\Gamma_\rD} + \|\fb\|_{0,\Omega})\\
& \qquad \qquad \qquad\quad + \dfrac{\gamma^2}{c_a}(1+\kappa_2\,\wt C^* \, r)\bigg\}\|\wh p_1- \wh p_2\|_{0,\Omega}, 
\end{align*}
for all $\wh p_1,\, \wh p_2\in\rW_h$. In this way, using estimate \eqref{eq:assumption-T-h-2} we obtain that $\bT_h$ is a contraction mapping on $\rW_h$, thus problem \eqref{eq:fixed-point-problem-h}, or equivalently \eqref{eq:weak-formulation-h} is well-posed.
Finally, analogously to the proof of Theorem ~\ref{theorem:unique-solution-weak1} (see also Lemmas~\ref{lem:cotaR1} and \ref{lem:cotaS-1-2}) we can obtain \eqref{eq:stability-h}, which concludes the proof.
\end{proof}
%%%%%%%%%%%%%%%%%%%%%%%%%%%%%%%%%%%%%%%%%%%%%%%%%%%%%%%%%%%%%%%%%

\section{A priori error estimates}\label{sec:error}

In this section, we aim to provide the convergence of the Galerkin scheme \eqref{eq:weak-formulation-h} and derive the corresponding rate of convergence. From now on we assume that the hypotheses of Theorem~\ref{theorem:unique-solution-weak1} and Theorem~\ref{theorem:unique-solution-weak1-h} hold.
%%%%%%%%%%%%%%%%%%%%%%%%%%%%%%%%%%%%%%%%%%%%%%%%%%%%%%%%%%%%%%%%%%
\subsection{Preliminaries}\label{sec:preliminaries}

Let the tuples $(\bsigma,\bu,\brho,\bvarphi,p)\in \bbH_\rN(\bdiv;\Omega)\times\bL^2(\Omega)\times\bbLskew\times \bH_\rN(\div;\Omega)\times\rL^2(\Omega)$ and $(\bsigma_h,\bu_h,\brho_h,\bvarphi_h,p_h)\in  \bbH^{\bsigma}_{h}\times\bH^{\bu}_{h}\times\bbH^{\brho}_{h}\times \bH^{\bvarphi}_{h}\times\rH^{p}_{h}$ be the unique solutions of \eqref{eq:weak1} and \eqref{eq:weak1-h}, respectively.

%Then, similarly to \cite[Section~4.2]{covy2022}, in order to simplify the subsequent analysis, 
Let us write $\texttt{e}_{\bsigma} = \bsigma - \bsigma_h$, $\texttt{e}_{\bu} = \bu - \bu_h$, $\texttt{e}_{\brho} = \brho - \brho_h$, $\texttt{e}_{\bvarphi} = \bvarphi - \bvarphi_h$ and $\texttt{e}_{p} = p - p_h$.
As usual, for a given $(\wh{\btau}_h,(\wh{\bv}_h,\wh{\bbeta}_h) )\in\bbH^{\bsigma}_{h}\times(\bH^{\bu}_{h}\times\bbH^{\brho}_{h})$ and $(\wh{\bpsi}_h,\wh{q}_h)\in \bH^{\bvarphi}_{h}\times\rH^{p}_{h}$, we shall then decompose these errors into
\begin{equation}\label{eq:decompositions}
\texttt{e}_{\bsigma} = \bxi_{\bsigma} + \bchi_{\bsigma},\qquad
\texttt{e}_{\bu} = \bxi_{\bu} + \bchi_{\bu},\qquad
\texttt{e}_{\brho} = \bxi_{\brho} + \bchi_{\brho},\qquad
\texttt{e}_{\bvarphi} = \bxi_{\bvarphi} + \bchi_{\bvarphi},\qquad
\texttt{e}_{p} = \xi_{p} + \chi_{p},
\end{equation}
with
$\bxi_{\bsigma} := \bsigma - \wh{\btau}_h$,
$\bchi_{\bsigma} := \wh{\btau}_h - \bsigma_h$,  
$\bxi_{\bu} := \bu - \wh{\bv}_h,$ 
$\bchi_{\bu} := \wh{\bv}_h - \bu_h$, 
$\bxi_{\brho} := \brho - \wh{\bbeta}_h$,  
$\bchi_{\brho} := \wh{\bbeta}_h - \brho_h$, 
$\bxi_{\bvarphi} := \bvarphi - \wh{\bpsi}_h$,  
$\bchi_{\bvarphi} := \wh{\bpsi}_h - \bvarphi_h$, 
$\xi_{p} := p - \wh{q}_h$, and  
$\chi_{p} := \wh{q}_h - p_h$.

Considering the first two equations of  problems \eqref{eq:weak-formulation} and \eqref{eq:weak-formulation-h}, the following identities hold
\begin{equation*}%label{eq:weak-formulation}
\begin{array}{rllll}
\ds a(\bsigma,\btau)  & +\quad \ds b(\btau,(\bu,\brho)) &= \ds H(\btau) - c(\btau,p)&\forall\, \btau\in\bbH_\rN(\bdiv;\Omega),\\ [1ex]
\ds b(\bsigma,(\bv,\bbeta))  &  & =\ds F(\bv,\bbeta)&\forall\,\bv\in\bL^2(\Omega),\, \forall\,\bbeta\in\bbLskew,
\end{array}
\end{equation*}
and
\begin{equation*}%\label{eq:weak-formulation-h}
\begin{array}{rllll}
\ds a(\bsigma_h,\btau_h)  & +\quad \ds b(\btau_h,(\bu_h,\brho_h))  &= \ds H(\btau_h) - c(\btau_h,p_h)&\forall\, \btau_h\in\bbH^{\bsigma}_{h},\\ [1ex]
\ds b(\bsigma_h,(\bv_h,\bbeta_h))  &  & =\ds F(\bv_h,\bbeta_h)&\forall(\bv_h,\bbeta_h)\in\bH^{\bu}_{h}\times\bbH^{\brho}_{h}.
\end{array}
\end{equation*}
From these relations we can obtain that for all $(\btau_h,(\bv_h,\brho_h))\in \bbH^{\bsigma}_{h}\times(\bH^{\bu}_{h}\times\bbH^{\brho}_{h})$, there holds
\begin{equation*}%\label{eq:weak-formulation-h}
\begin{array}{rllll}
\ds a(\texttt{e}_{\bsigma},\btau_h)  & +\quad \ds b(\btau_h,(\texttt{e}_{\bu},\texttt{e}_{\brho}))  &= \ds - c(\btau_h,\texttt{e}_{p}),\\ [1ex]
\ds b(\texttt{e}_{\bsigma},(\bv_h,\bbeta_h))  &  & =\ds 0,
\end{array}
\end{equation*}
which together with the definition of the errors in \eqref{eq:decompositions}, implies that 
\begin{equation}\label{eq:auxiliar-equation-1}
\begin{array}{ll}
\ds a(\bchi_{\bsigma},\btau_h) + b(\btau_h,(\bchi_{\bu},\bchi_{\brho})) + b(\bchi_{\bsigma},(\bv_h,\bbeta_h))\\
\ds\qquad = -a(\bxi_{\bsigma},\btau_h) - b(\btau_h,(\bxi_{\bu},\bxi_{\brho})) - b(\bxi_{\bsigma},(\bv_h,\bbeta_h)) - c(\btau_h,\bchi_{p}) - c(\btau_h,\bxi_{p}),
\end{array}
\end{equation}
for all $(\btau_h,(\bv_h,\brho_h))\in \bbH^{\bsigma}_{h}\times(\bH^{\bu}_{h}\times\bbH^{\brho}_{h})$.

Next, considering the last two equations of both problems \eqref{eq:weak-formulation} and \eqref{eq:weak-formulation-h}, we obtain
\begin{equation*}%\label{eq:weak-formulation}
\begin{array}{rllll}
\ds \wt a_{\bsigma,p}(\bvarphi,\bpsi)  & +\quad \ds \wt b(\bpsi,p)  &= \ds \wt H(\bpsi)&\forall\, \bpsi\in\bH_\rN(\div;\Omega),\\ [1ex]
\ds \wt b(\bvarphi,q)  &-\quad\wt c(p,q) &  =\ds \cblue{\tilde{F}}(q) + c(\bsigma,q)&\forall\, q\in\rL^2(\Omega).
\end{array}
\end{equation*}
and
\begin{equation*}%\label{eq:weak-formulation-h}
\begin{array}{rllll}
\ds \wt a_{\bsigma_h,p_h}(\bvarphi_h,\bpsi_h)  & +\quad \ds \wt b(\bpsi_h,p_h)  &= \ds \wt H(\bpsi_h)&\forall\, \bpsi_h\in\bH^{\bvarphi}_{h},\\ [1ex]
\ds \wt b(\bvarphi_h,q_h)  &-\quad\wt c(p_h,q_h) & =\ds \cblue{\tilde{F}}(q_h) + c(\bsigma_h,q_h)&\forall\, q_h\in\rH^{p}_{h}.
\end{array}
\end{equation*}
Then, using arguments similar to those in Lemma~\ref{lem:cotaS-1-2}, by adding  $\pm \wt a_{\bsigma_h, p_h}(\bvarphi,\bpsi_h)$, we have
\begin{align*}%\label{eq:weak-formulation-h}
& 
\ds \wt a_{\bsigma_h,p_h}(\texttt{e}_{\bvarphi_h},\bpsi_h) + \wt b(\bpsi_h,\texttt{e}_{p_h}) + \wt b(\texttt{e}_{\bvarphi_h},q_h) -\wt c(\texttt{e}_{p_h},q_h)\\
&\qquad = - \int_{\Omega} \Big(\bkappa(\bsigma,p)^{-1}-\bkappa(\bsigma_h,p_h)^{-1}\Big)\, \bvarphi \cdot \bpsi_h + c(\texttt{e}_{\bsigma},q_h),
\end{align*}
which together with \eqref{eq:decompositions}, implies that
\begin{equation}\label{eq:auxiliar-equation-2}
\begin{array}{lllll}
\ds \wt a_{\bsigma_h,p_h}(\bchi_{\bvarphi},\bpsi_h) + \wt b(\bpsi_h,\chi_{p_h}) + \wt b(\bchi_{\bvarphi},q_h) -\wt c(\chi_{p},q_h)+ \wt a_{\bsigma_h,p_h}(\bxi_{\bvarphi},\bpsi_h)\\
\ds \qquad =- \wt b(\bpsi_h,\xi_{p}) - \wt b(\bxi_{\bvarphi},q_h) +\wt c(\xi_{p},q_h) - \int_{\Omega} \Big(\bkappa(\bsigma,p)^{-1}-\bkappa(\bsigma_h,p_h)^{-1}\Big)\, \bvarphi \cdot \bpsi_h + c(\texttt{e}_{\bsigma},q_h).
\end{array}
\end{equation}
%%%%%%%%%%%%%%%%%%%%%%%%%%%%%%%%%%%%%%%%%%%%%%%%%%%%%%%%%%%%%%%%
\subsection{Derivation of C\'ea estimates}
\begin{lemma}\label{lem:CEA-1}
There exist $C_3^*,\, C_4^*>0$, independent of $h$, such that
\begin{equation}\label{eq:bound-chi-1}
\|\bchi_{\bsigma}\|_{\bdiv;\Omega} + \|\bchi_{\bu}\|_{0,\Omega} + \|\bchi_{\brho}\|_{0,\Omega} 
 \leq C_3^* (\|\bxi_{\bsigma}\|_{\bdiv;\Omega} + \|\bxi_{\bu}\|_{0,\Omega} + \|\bxi_{\brho}\|_{0,\Omega}  + \|\bxi_{p}\|_{0,\Omega}  ) + C_4^* \gamma \|\chi_{p}\|_{0,\Omega}.
\end{equation}    
\end{lemma}
\begin{proof}
From the properties of $a$ and $b$ (refer to \eqref{eq:ellipticity-a-h} and \eqref{eq:infsup-b-h}), and \cite[Proposition 2.36]{ernguermond}, we derive the following discrete global inf-sup condition
\begin{align*}%\label{eq:global-infsup-1}
%\begin{array}{ll}
& \|\bchi_{\bsigma}\|_{\bdiv;\Omega} + \|\bchi_{\bu}\|_{0,\Omega} + \|\bchi_{\brho}\|_{0,\Omega}\\
%\ds \qquad 
& \quad \leq (C_1^* +C_2^*) \sup_{\bzero\neq(\btau_h,\bv_h,\bpsi_h)\in \bbH^{\bsigma}_{h}\times\bH^{\bu}_{h}\times\bbH^{\brho}_{h} } \!\!\!\frac{a(\bchi_{\bsigma},\btau_h) + b(\btau_h,(\bchi_{\bu},\bchi_{\brho})) + b(\bchi_{\bsigma},(\bv_h,\bbeta_h))}{\|\btau_h\|_{\bdiv;\Omega} + \|\bv_h\|_{0,\Omega} + \|\bbeta_h\|_{0,\Omega}},
%\end{array}
\end{align*}
\revX{where $C_1^*,\,C_2^*>0$ independent of $h$ are} the discrete version of the constants $C_1,\,C_2$ defined in \eqref{eq:constant-C1}. Then, combining the last inequality with \eqref{eq:auxiliar-equation-1}, and the continuity properties of $a$ and $b$ (see \eqref{eq:bounded-a-b}), 
we obtain
\begin{equation*}
\begin{array}{ll}
\|\bchi_{\bsigma}\|_{\bdiv;\Omega} + \|\bchi_{\bu}\|_{0,\Omega} + \|\bchi_{\brho}\|_{0,\Omega}\\
\ds \qquad \leq (C_1^* + C_2^* )(\dfrac{1}{\mu}\|\bxi_{\bsigma}\|_{\bdiv;\Omega} + \|\bxi_{\bu}\|_{0,\Omega} + \|\bxi_{\brho}\|_{0,\Omega} + \|\bxi_{\bsigma}\|_{\bdiv;\Omega} + \gamma \|\chi_{p}\|_{0,\Omega} + \gamma \|\xi_{p}\|_{0,\Omega}  ),
\end{array}
\end{equation*}
which implies \eqref{eq:bound-chi-1} with $C_3^*:=(C_1^* + C_2^* )(\frac{1}{\mu}+1+\gamma)$ and $C_4^*:=C_1^* + C_2^* $. \revX{Note also, that the error estimate is robust with respect to $\lambda$.}
\end{proof}

\begin{lemma}\label{lem:CEA-2}
There exist $\wt C_5^*,\, \wt C_6^*>0$, independent of $h$, such that
\begin{equation}\label{eq:bound-chi-2}
\begin{array}{ll}
\ds \|\bchi_{\bvarphi}\|_{\div;\Omega} + \|\bchi_{p}\|_{0,\Omega} \leq \wt C^*_5\,(\|\bxi_{\bvarphi}\|_{\div;\Omega} + \|\xi_{p}\|_{0,\Omega} + \|\bxi_{\bsigma}\|_{\bdiv;\Omega})\\[1ex]
\qquad + \wt C^*_6 \big( (\|g\|_{0,\Omega} + \|p_{\rD}\|_{1/2,\Gamma_\rD} + \|\bu_\rD\|_{1/2,\Gamma_\rD} + \|\fb\|_{0,\Omega} + \gamma \,r\,) \|\chi_{p}\|_{0,\Omega} + \gamma\|\bchi_{\bsigma}\|_{\bdiv;\Omega} \big).
\end{array}
\end{equation}    
\end{lemma}
\begin{proof}
Similarly to Lemma~\ref{lem:CEA-1}, using the properties of $\wt a$, $\wt b$ and $\wt c$ (refer to \eqref{eq:ellipticity-hat-a-h}, \eqref{eq:infsup-hat-b-h} and \eqref{eq:ellipticity-hat-c}), and \cite[Th.~3.4]{cg2022}, we derive the following discrete global inf-sup condition
\begin{equation*}
\|\bchi_{\bvarphi}\|_{\div;\Omega} + \|\chi_{p}\|_{0,\Omega} 
\leq 2\wt C^* \sup_{\bzero\neq(\bpsi_h,q_h)\in \bH^{\bvarphi}_{h}\times\rH^{p}_{h} } \frac{\wt a_{\bsigma_h,p_h}(\bchi_{\bvarphi},\bpsi_h) + b(\bpsi_h,\chi_{p}) + b(\bchi_{\bvarphi},q_h) - \wt c(\chi_p,q_h)}{\|\bpsi_h\|_{\div;\Omega} + \|q_h\|_{0,\Omega}},
\end{equation*}
with $\wt C^*$ defined as in \eqref{eq:assumption-T-h}. Then, using the equation \eqref{eq:auxiliar-equation-2}, the bound properties of $\wt a$, $\wt b$ and $\wt c$ (see \eqref{eq:bound-a} and \eqref{eq:bounded-hat-b-c}), and the second assumption for $\bkappa$ (cf. \eqref{prop-kappa}), we obtain
\begin{align*}
&\|\bchi_{\bvarphi}\|_{\div;\Omega} + \|\chi_{p}\|_{0,\Omega} \\
& \leq 2\wt C^* (C_{\wt a} \|\bxi_{\bvarphi}\|_{\div;\Omega} + \|\xi_{p}\|_{0,\Omega} + \|\bxi_{\bvarphi}\|_{\div;\Omega} + \wt\gamma \|\xi_{p}\|_{0,\Omega} \\
&\qquad \quad + \|\bkappa(\bsigma,p)^{-1} - \bkappa(\bsigma_h,p_h)^{-1}\|_{\bbL^\infty(\Omega)}\|\bvarphi\|_{\div;\Omega} + \gamma\|\texttt{e}_{\bsigma}\|_{\bdiv;\Omega})\\
& \leq 2\wt C^* (C_{\wt a} \|\bxi_{\bvarphi}\|_{\div;\Omega} + \|\xi_{p}\|_{0,\Omega} + \|\bxi_{\bvarphi}\|_{\div;\Omega} + \wt\gamma \|\xi_{p}\|_{0,\Omega} + \kappa_2\|\texttt{e}_{p}\|_{0,\Omega}\|\bvarphi\|_{\div;\Omega} + \gamma\|\texttt{e}_{\bsigma}\|_{\bdiv;\Omega}),
\end{align*}
hence, using the fact that $\bvarphi$ satisfies \eqref{eq:stability} and the error decomposition \eqref{eq:decompositions}, we have 
\begin{align*}
\|\bchi_{\bvarphi}\|_{\div;\Omega} + \|\chi_{p}\|_{0,\Omega} &  \leq 2\wt C^* (C_{\wt a} \|\bxi_{\bvarphi}\|_{\div;\Omega} + \|\xi_{p}\|_{0,\Omega} + \|\bxi_{\bvarphi}\|_{\div;\Omega} + \wt\gamma \|\xi_{p}\|_{0,\Omega} + \kappa_2\|\xi_{p}\|_{0,\Omega}\|\bvarphi\|_{\div;\Omega}   \\
& \qquad \qquad + \gamma\|\bxi_{\bsigma}\|_{\bdiv;\Omega}) + 2\wt C^* ( \kappa_2\|\chi_{p}\|_{0,\Omega}\|\bvarphi\|_{\div;\Omega} + \gamma\|\bchi_{\bsigma}\|_{\bdiv;\Omega})\\
& \leq \wt C^*_5\,(\|\bxi_{\bvarphi}\|_{\div;\Omega} + \|\xi_{p}\|_{0,\Omega} + \|\bxi_{\bsigma}\|_{\bdiv;\Omega})  \\
&\qquad + 2\wt C^* \big( \kappa_2 C (\|g\|_{0,\Omega} + \|p_{\rD}\|_{1/2,\Gamma_\rD} + \|\bu_\rD\|_{1/2,\Gamma_\rD} + \|\fb\|_{0,\Omega} + \gamma \,r\,) \|\chi_{p}\|_{0,\Omega} \\
&\qquad \qquad + \gamma\|\bchi_{\bsigma}\|_{\bdiv;\Omega} \big),
\end{align*}
the last equation implies \eqref{eq:bound-chi-2}, with $\wt C^*_5 := 2\wt C^* \big( C_{\wt a}+1 +\wt\gamma +\gamma + \kappa_2 C (\|g\|_{0,\Omega} + \|p_{\rD}\|_{1/2,\Gamma_\rD} + \|\bu_\rD\|_{1/2,\Gamma_\rD} + \|\fb\|_{0,\Omega} + \gamma \,r\,)     \big)$ and $\wt C^*_6 := 2\wt C^*(\kappa_2 C + 1)$, and concludes the proof.
\end{proof}

\begin{theorem}\label{th:cea-estimate}
Assume that
\begin{equation}\label{eq:assumption-cea}
\ds (C_4^* + \wt C^*_6 + \wt C^*_6 r) \gamma + \wt C^*_6 (\|g\|_{0,\Omega} + \|p_{\rD}\|_{1/2,\Gamma_\rD} + \|\bu_\rD\|_{1/2,\Gamma_\rD} + \|\fb\|_{0,\Omega} )  \leq \dfrac{1}{2},
\end{equation}
with $C_4^*$ and $\wt C^*_6$ being the constants in Lemmas \ref{lem:CEA-1} and \ref{lem:CEA-2}. Furthermore, assume that the hypotheses of Theorem~\ref{theorem:unique-solution-weak1} and Theorem~\ref{theorem:unique-solution-weak1-h} hold. Let $(\bsigma,\bu,\brho,\bvarphi,p)\in \bbH_\rN(\bdiv;\Omega)\times\bL^2(\Omega)\times\bbLskew\times \bH_\rN(\div;\Omega)\times\rL^2(\Omega)$ and $(\bsigma_h,\bu_h,\brho_h,\bvarphi_h,p_h)\in  \bbH^{\bsigma}_{h}\times\bH^{\bu}_{h}\times\bbH^{\brho}_{h}\times \bH^{\bvarphi}_{h}\times\rH^{p}_{h}$ be the unique solutions of \eqref{eq:weak-formulation} and \eqref{eq:weak-formulation-h}, respectively.
Then, there exists $C_{\mathrm{C\acute{e}a}}>0$, independent of $h$, such that
\begin{equation}\label{eq:cea-estimate}
\begin{array}{ll}
\ds \|\texttt{e}_{\bsigma}\|_{\bdiv;\Omega} + \|\texttt{e}_{\bu}\|_{0,\Omega} + \|\texttt{e}_{\brho}\|_{0,\Omega} + \|\texttt{e}_{\bvarphi}\|_{\div;\Omega} + \|\texttt{e}_{p}\|_{0,\Omega}\\[1ex]
\ds \qquad \leq C_{\mathrm{C\acute{e}a}}\,\dist \big( (\bsigma,\bu,\brho,\bvarphi,p),\, \bbH^{\bsigma}_{h}\times\bH^{\bu}_{h}\times\bbH^{\brho}_{h}\times \bH^{\bvarphi}_{h}\times\rH^{p}_{h}   \big).
 \end{array}
\end{equation}
\end{theorem}
\begin{proof}
Combining \eqref{eq:bound-chi-1} and \eqref{eq:bound-chi-2}, and using the assumption \eqref{eq:assumption-cea}, we deduce 
\begin{align*}
& \|\bchi_{\bsigma}\|_{\bdiv;\Omega} + \|\bchi_{\bu}\|_{0,\Omega} + \|\bchi_{\brho}\|_{0,\Omega} + \|\bchi_{\bvarphi}\|_{\div;\Omega} + \|\chi_{p}\|_{0,\Omega} \\
&\quad \leq C_3^* (\|\bxi_{\bsigma}\|_{\bdiv;\Omega} + \|\bxi_{\bu}\|_{0,\Omega} + \|\bxi_{\brho}\|_{0,\Omega} +  \|\xi_{p}\|_{0,\Omega}  ) \\
&\qquad \ds + \wt C^*_5\,(\|\bxi_{\bvarphi}\|_{\div;\Omega} + \|\xi_{p}\|_{0,\Omega} + \|\bxi_{\bsigma}\|_{\bdiv;\Omega}) + \dfrac{1}{2} \|\chi_{p}\|_{0,\Omega} + \dfrac{1}{2} \|\bchi_{\bsigma}\|_{\bdiv;\Omega}.
\end{align*} 
And from the latter inequality we obtain
\begin{equation}\label{eq:auxiliar-equation-3}
\begin{array}{ll}
\ds \|\bchi_{\bsigma}\|_{\bdiv;\Omega} + \|\bchi_{\bu}\|_{0,\Omega} + \|\bchi_{\brho}\|_{0,\Omega} + \|\bchi_{\bvarphi}\|_{\div;\Omega} + \|\chi_{p}\|_{0,\Omega} \\[1ex]
\qquad \ds\leq 2(C_3^* + \wt C^*_5) (\|\bxi_{\bsigma}\|_{\bdiv;\Omega} + \|\bxi_{\bu}\|_{0,\Omega} + \|\bxi_{\brho}\|_{0,\Omega} + \|\bxi_{\bvarphi}\|_{\div;\Omega} +  \|\xi_{p}\|_{0,\Omega}  ).
\end{array}
\end{equation} 
In this way, from \eqref{eq:decompositions}, \eqref{eq:auxiliar-equation-3} and the triangle inequality we obtain
\begin{align*}
& \|\texttt{e}_{\bsigma}\|_{\bdiv;\Omega} + \|\texttt{e}_{\bu}\|_{0,\Omega} + \|\texttt{e}_{\brho}\|_{0,\Omega} + \|\texttt{e}_{\bvarphi}\|_{\div;\Omega} + \|\texttt{e}_{p}\|_{0,\Omega}\\
&\quad \leq \|\bchi_{\bsigma}\|_{\bdiv;\Omega} + \|\bxi_{\bsigma}\|_{\bdiv;\Omega} + \|\bchi_{\bu}\|_{0,\Omega} + \|\bxi_{\bu}\|_{0,\Omega}\\ &\qquad  + \|\bchi_{\brho}\|_{0,\Omega} + \|\bxi_{\brho}\|_{0,\Omega} + \|\bchi_{\bvarphi}\|_{\div;\Omega} + \|\bxi_{\bvarphi}\|_{\div;\Omega} + \|\chi_{p}\|_{0,\Omega}+  \|\xi_{p}\|_{0,\Omega}\\
&\quad \leq (2C_3^* + 2\wt C^*_5 + 1) (\|\bxi_{\bsigma}\|_{\bdiv;\Omega} + \|\bxi_{\bu}\|_{0,\Omega} + \|\bxi_{\brho}\|_{0,\Omega} + \|\bxi_{\bvarphi}\|_{\div;\Omega} +  \|\xi_{p}\|_{0,\Omega}  ),
 \end{align*}
which combined with the fact that $(\wh{\btau}_h,(\wh{\bv}_h,\wh{\bbeta}_h) )\in\bbH^{\bsigma}_{h}\times(\bH^{\bu}_{h}\times\bbH^{\brho}_{h})$ and $(\wh{\bpsi}_h,\wh{q}_h)\in \bH^{\bvarphi}_{h}\times\rH^{p}_{h}$ are arbitrary (see \eqref{eq:decompositions}), concludes the proof.
\end{proof}

%%%%%%%%%%%%%%%%%%%%%%%%%%%%%%%%%%%%%%%%%%%%%%%%%%%%%%%%%%%%%%%%%
\subsection{Rates of convergence}
In order to establish the rate of convergence of the Galerkin scheme \eqref{eq:weak-formulation-h}, we first recall the following approximation properties %$\mathbf{AP}$, 
associated with the FE spaces  specified in Section~\ref{sec:spaces}.
\medskip

For each $0 < m \leq k+1$ and for each $\btau\in \bbH^{m}(\Omega)\cap \bbH_\rN(\bdiv;\Omega)$ with $\bdiv\,\btau\in \bH^{m}(\Omega)$, 
$\bv\in \bH^{m}(\Omega)$,  $\bbeta\in \bbH^{m}(\Omega)\cap \bbL^2_{\rskew}(\Omega)$,  $\bpsi\in \bH^{m}(\Omega)\cap \bH_\rN(\revX{\div};\Omega)$ with $\div\,\bv\in \rH^{m}(\Omega)$, and $q\in \rH^{m}(\Omega)$, 
there holds 
\begin{subequations}\label{approx-prop}
%\noindent $(\mathbf{AP}^{\bsigma}_h)$ 
\begin{align}
\label{eq:rate-sigma}
\dist\big(\btau,\bbH^{\bsigma}_{h}\big) &:= \inf_{\btau_h \in \bbH^{\bsigma}_{h} } \|\btau - \btau_h\|_{\bdiv;\Omega} \lesssim h^{m}\,\Big\{ \|\btau\|_{m,\Omega} + \|\bdiv\,\btau\|_{m,\Omega} \Big\},\\
%\end{equation}
%\noindent $(\mathbf{AP}^{\bu}_h)$ For each $0\leq m \leq k+1$ and for each  there holds
%\begin{equation}
\label{eq:rate-u}
\dist\big(\bv,\bH^{\bu}_{h}\big) &:= \inf_{\bv_h \in \bH^{\bu}_{h}} \|\bv - \bv_h\|_{0,\Omega} \lesssim  h^{m}\,\|\bv\|_{m,\Omega},\\
%\end{equation}
%\noindent $(\mathbf{AP}^{\brho}_h)$ For each $0\leq m \leq k+1$ and for eachthere holds
%\begin{equation}
\label{eq:rate-rho}
\dist\big(\bbeta,\bbH^{\brho}_{h}\big) &:= \inf_{\bbeta_h \in \bbH^{\brho}_{h}} \|\bbeta - \bbeta_h\|_{0,\Omega} \lesssim  h^{m}\,\|\bbeta\|_{m,\Omega},\\
%\end{equation}
%
%\noindent $(\mathbf{AP}^{\bvarphi}_h)$ For each $0 < m \leq k+1$ and for eachthere holds
%\begin{equation}
\label{eq:rate-varphi}
\dist\big(\bpsi,\bH^{\bvarphi}_{h}\big) &:= \inf_{\bpsi_h \in \bH^{\bvarphi}_{h} } \|\bpsi - \bpsi_h\|_{\div;\Omega} \lesssim h^{m}\,\Big\{ \|\bpsi\|_{m,\Omega} + \|\div\,\bpsi\|_{m,\Omega} \Big\},\\
%\end{equation}
%\noindent $(\mathbf{AP}^{p}_h)$ For each $0\leq m \leq k+1$ and for each , there holds
%\begin{equation}
\label{eq:rate-p}
\dist\big(q,\rH^{p}_{h}\big) &:= \inf_{q_h \in \rH^{p}_{h}} \|q - q_h\|_{0,\Omega} \lesssim  h^{m}\,\|q\|_{m,\Omega}.
\end{align}
\end{subequations}
For \eqref{eq:rate-sigma}, \eqref{eq:rate-u} and \eqref{eq:rate-rho} we refer to \cite[Th. 2.4]{gatica2013priori}, whereas \eqref{eq:rate-varphi} and \eqref{eq:rate-p} are provided in \cite[Th. 3.6]{gatica14} and \cite[Proposition 1.134]{ernguermond}, respectively.
%[23, Proposition 1.134, Section 1.6.3].
%
%, whereas for \eqref{eq:rate-p-2} is the same as \eqref{eq:rate-p}.
%
With these steps we are now in \revX{a} position to state the rates of convergence associated with the Galerkin scheme \eqref{eq:weak-formulation-h}.

\begin{theorem}\label{th:rate1}
Assume that the hypotheses of Theorem~\ref{th:cea-estimate} hold and let $(\bsigma,\bu,\brho,\bvarphi,p)\in \bbH_\rN(\bdiv;\Omega)\times\bL^2(\Omega)\times\bbLskew\times \bH_\rN(\div;\Omega)\times\rL^2(\Omega)$ and $(\bsigma_h,\bu_h,\brho_h,\bvarphi_h,p_h)\in  \bbH^{\bsigma}_{h}\times\bH^{\bu}_{h}\times\bbH^{\brho}_{h}\times \bH^{\bvarphi}_{h}\times\rH^{p}_{h}$ be the unique solutions of the continuous and discrete problems \eqref{eq:weak-formulation} and \eqref{eq:weak-formulation-h}, respectively.
Assume further that  $\bsigma\in \bbH^{m}(\Omega)$, $\bdiv \,\bsigma \in \bH^{m}(\Omega)$, $\bu \in \bH^{m}(\Omega)$, $\brho\in \bbH^{m}(\Omega)$, $\bvarphi\in \bH^{m}(\Omega)$, $\div \,\bvarphi \in \rH^{m}(\Omega)$ and $p\in \rH^{m}(\Omega)$, for $1\leq m \leq k+1$.
Then there exists $C_{\mathrm{rate}}>0$, independent of $h$, such that 
\begin{equation*}
\begin{array}{ll}
\ds \|\texttt{e}_{\bsigma}\|_{\bdiv;\Omega} + \|\texttt{e}_{\bu}\|_{0,\Omega} + \|\texttt{e}_{\brho}\|_{0,\Omega} + \|\texttt{e}_{\bvarphi}\|_{\div;\Omega} + \|\texttt{e}_{p}\|_{0,\Omega}\\[1ex]
\ds \quad  \leq\
C_{\mathrm{rate}}\,h^{m}\Big\{  \|\bsigma\|_{m,\Omega} + \|\bdiv\,\bsigma\|_{m,\Omega} + \|\bu\|_{m,\Omega} + \|\brho\|_{m,\Omega} + \|\bvarphi\|_{m,\Omega} +\|\div\,\bvarphi\|_{m,\Omega}  + \|p\|_{m,\Omega}\Big\}.
\end{array}
\end{equation*}
\end{theorem} 
\begin{proof}
The result follows from C\'ea estimate \eqref{eq:cea-estimate}
and the approximation properties \eqref{approx-prop}. 
\end{proof}	

\begin{remark}
Similarly to \cite{lamichhane24}, the analysis  developed in Sections \ref{sec:weak}-\ref{sec:error} can be adapted to a formulation without the variable $\brho$ ($\brho_h$ in the discrete problem), imposing  symmetry of $\bsigma$ by taking  $\bsigma\in\bbHsym\,:=\,\big\{ \btau \in \bbLsym:\  \bdiv\,\btau \in \bL^2(\Omega) \big\}$ and $\bbLsym:=\{\btau \in \bbL^2(\Omega):\,\btau = \btau^{\tt t} \}$, utilizing results from \cite[Section 2.2]{lamichhane24} (\cite[Section 4.1]{lamichhane24} for the discrete problem), and adapting the strategy used in, e.g., \cite[Sections 3 and 4]{lamichhane24}.
\end{remark}
%%%%%%%%%%%%%%%%%%%%%%%%%%%%%%%%%%%%%%%%%%%%%%%%%%%%%%%%%%%%%%%%%%%%%%%%%
\section{A posteriori error estimates}\label{sec:aposteriori}
In this section we derive residual-based a posteriori error estimates  and  demonstrate the reliability and efficiency of the proposed estimators. Mainly due to notational convenience, we confine our analysis to the two-dimensional case. The extension to three-dimensional case should be quite straightforward (see, e.g., \cite{cgor2023}).
Similarly to \cite[Section~4]{ccov2022}, we introduce additional notation. Let $\mathcal{E}_{h}$ be the set of edges of $\mathcal{T}_{h}$, whose corresponding diameters are denoted $h_E$, and define
$$
\textcolor{lightblue}{\mathcal{E}_h(\star)\,:=\,\big\{\,E\in\mathcal{E}_h\,:\; E\subseteq \star \, \big\},\quad \star \in \{ \Omega, \Gamma_\rD, \Gamma_\rN\}}.
%\mathcal{E}_h(\Gamma_\rD)\,:=\,\big\{\,E\in\mathcal{E}_{h}\,:\; E\subseteq \Gamma_\rD \, \big\}, \qan
%\mathcal{E}_h(\Gamma_\rN)\,:=\,\big\{\,E\in\mathcal{E}_{h}\,:\; E\subseteq \Gamma_\rN \, \big\}\,.
$$
On each $E\in\mathcal{E}_{h}$, we also define the unit normal vector $\bn_E:=(n_1,n_2)^\rt$  and the tangential vector $\bs_E:=(-n_2,n_1)^\rt$. However, when no confusion arises, we will simply write $\bn$ and $\bs$ instead of $\bn_E$ and $\bs_E$, respectively. Also, by $\frac{\mathrm{d}}{\mathrm{d}\bs}$ we denote the tangential derivative. The usual jump operator $[\![ \cdot ]\!]$ across internal edges are defined for piecewise continuous matrix, vector, or scalar-valued functions. For sufficiently smooth scalar $\psi$, vector $\bv := (v_1,v_2)^\rt$, and tensor fields $\btau := (\tau_{ij})_{1\leq i,j\leq2}$, we let
\begin{gather*}
\curl(\psi):=\Big(\frac{\partial\psi}{\partial x_2}\,,\,-\frac{\partial\psi}{\partial x_1}\Big)^{\tt t}
\,,\quad
\rot(\bv):=\frac{\partial v_2}{\partial x_1}-\frac{\partial v_1}{\partial x_2}\,,\quad \bcurl(\bv)=\begin{pmatrix}\curl (v_1)^\rt\\[1ex] 
\curl (v_2)^\rt\end{pmatrix},\\
\qan \underline{\bcurl}(\btau)=\begin{pmatrix}\rot(\btau_1)\\
\rot(\btau_2)\end{pmatrix}.
\end{gather*}
In addition, we denote by $\Pi_h$ the Raviart--Thomas interpolator and by $I_{h}$ the Cl\'ement operator (see, e.g., \cite[Section~3]{alvarez16} for their properties). In what follows, we denote by $\bPi_h$ the tensor version of $\Pi_h$, which is defined row-wise by $\Pi_h$ and by $\bI_h$ the corresponding vector version of $I_h$ which is defined componentwise by $I_h$.

In what follows, we will assume that the hypotheses of Theorems \ref{theorem:unique-solution-weak1} and \ref{theorem:unique-solution-weak1-h} are satisfied. Let $\bsigma_h$, $\bu_h$, $\brho_h$, $\bvarphi_h$, $p_h$ denote the FE solutions of \eqref{eq:weak-formulation-h}. We define the residual-based and fully computable local contributions to the error estimator $\Xi_K^2$, defined as the sum of $\Xi_{s,K}^2$ and $\Xi_{f,K}^2$, where $\Xi_{s,K}$ and $\Xi_{f,K}$ pertain to the solid (mixed elasticity) and fluid (mixed Darcy) components, respectively: 
\begin{subequations}
\begin{align}\label{eq:res-1}
%\begin{array}{ll}
%\ds 
\Xi_{s,K}^2 &:=\|\fb +\bdiv \bsigma_h \|_{0,K}^2+\|\bsigma_h-\bsigma_h^{\tt t}\|^2_{0,K}+h_K^2\|\mathcal{C}^{-1}\bsigma_h + \dfrac{\alpha}{d\lambda+2\mu} p_h \mathbf{I} - \bnabla \bu_h +\brho_h\|_{0,K}^2\nonumber\\
&\quad +h_K^2\|\underline{\mathbf{curl}}(\mathcal{C}^{-1}\bsigma_h + \dfrac{\alpha}{d\lambda+2\mu} p_h \mathbf{I}  +\brho_h)\|_{0,K}^2+\sum_{E\in \partial K \cap \mathcal{E}_h(\Gamma_\rD)}h_E \|\bu_\rD - \bu_h \|_{0,E}^2\nonumber\\
&\quad +\sum_{E\in \partial K \cap \mathcal{E}_h(\Omega)}h_E \|[\![(\mathcal{C}^{-1}\bsigma_h + \dfrac{\alpha}{d\lambda+2\mu} p_h \mathbf{I} +\brho_h)\bs]\!]\|_{0,E}^2\nonumber\\
&\quad +\sum_{E\in \partial K \cap \mathcal{E}_h(\Gamma_\rD)}h_E \|(
\mathcal{C}^{-1}\bsigma_h + \dfrac{\alpha}{d\lambda+2\mu} p_h \mathbf{I} +\brho_h)\bs -\frac{\mathrm{d} \bu_\rD}{\mathrm{d}\bs}\|_{0,E}^2 ,
\\
\label{eq:res-2}
 \Xi_{f,K}^2 &:= \|c_0 p_h + \dfrac{\alpha}{d\lambda+2\mu}\tr\bsigma_h + \dfrac{d\alpha^2}{d\lambda+2\mu} p_h - \div \bvarphi_h-g\|_{0,K}^2
+h_K^2\|\bkappa(\bsigma_h,p_h)^{-1}\bvarphi_h- \nabla p_h\|_{0,K}^2\nonumber\\
&\quad +h_K^2\|\rot(\bkappa(\bsigma_h,p_h)^{-1}\bvarphi_h)\|_{0,K}^2+\sum_{E\in \partial K \cap \mathcal{E}_h(\Omega)}h_E \|[\![(\bkappa(\bsigma_h,p_h)^{-1}\bvarphi_h)\cdot\bs]\!]\|_{0,E}^2 \nonumber\\
&\quad + \sum_{E\in \partial K \cap \mathcal{E}_h(\Gamma_\rD)}h_E \|\bkappa(\bsigma_h,p_h)^{-1}\bvarphi_h\cdot\bs -\frac{\mathrm{d} p_\rD}{\mathrm{d}\bs}\|_{0,E}^2 
 + \sum_{E\in \partial K \cap \mathcal{E}_h(\Gamma_\rD)}h_E \|p_\rD - p_h\|_{0,E}^2.
\end{align}
\end{subequations}
%where $p_\rD$ is the possibly non-homogeneous pressure datum.
%
Then, we define the global estimator 
\begin{equation}\label{eq:estimator}
\Xi^2:= \sum_{K\in\mathcal{T}_h} \Xi_{s,K}^2 + \Xi_{f,K}^2.
\end{equation}

\subsection{Reliability of the a posteriori error estimator}%\label{section:reliability}

First we prove preliminary results that will be key in showing the reliability of the global estimator. 
\begin{lemma}\label{lem-res-1}
There exists $C_1>0$, such that
    \begin{equation*}
\|\bsigma-\bsigma_h\|_{\bdiv;\Omega} + \|\bu-\bu_h\|_{0,\Omega} + \|\brho-\brho_h\|_{0,\Omega}\leq C_1 \big(\|\cR_1\|+ \|\fb+\bdiv(\bsigma_h)\|_{0,\Omega} + \|\bsigma_h-\bsigma_h^t\|_{0,\Omega} \big),
\end{equation*}
    where 
\begin{equation}\label{eq:def-R-1}
 \ds \cR_1 (\btau):=a(\bsigma-\bsigma_h,\btau) + b(\btau,(\bu-\bu_h,\brho-\brho_h)),   
\end{equation}    
with $\cR_1 (\btau_h)=-c(\btau_h,p-p_h)$ for all $\btau_h \in \bbH^{\bsigma}_h$, and $\ds ||\cR_1||=\sup_{\bzero\neq\btau  \in \bbH_{\rN}(\bdiv;\Omega)}\frac{\mathcal{R}_1({\btau})}{\|\btau\|_{\bdiv;\Omega} }.$ 
\end{lemma}
\begin{proof}
Using the properties of bilinear forms $a$ and $b$, as outlined in equations \eqref{eq:ellipticity-a} and \eqref{eq:infsup-b}, along with the insight from \cite[Proposition 2.36]{ernguermond}, there exists $C_1>0$ depending on $\mu,\, c_a,\, \beta$ such that
\begin{align*}%\label{eq:global-infsup-1}
&\|{\tt e}_{\bsigma}\|_{\bdiv;\Omega} + \|{\tt e}_{\bu}\|_{0,\Omega} + \|{\tt e}_{\brho}\|_{0,\Omega} 
 \leq C_1 \!\!\!\!\!\!\!\!\sup_{\substack{\bzero\neq(\btau,\bv,\bbeta) \\ \quad \in \bbH_\rN(\bdiv;\Omega)\times\bL^2(\Omega)\times\bbLskew} } \!\!\!\!\frac{a({\tt e}_{\bsigma},\btau) + b(\btau,({\tt e}_{\bu},{\tt e}_{\brho})) + b({\tt e}_{\bsigma},(\bv,\bbeta))}{\|\btau\|_{\bdiv;\Omega} + \|\bv\|_{0,\Omega} + \|\bbeta\|_{0,\Omega}}\\
&\qquad \leq C_1 \;\Big(\sup_{\bzero\neq \btau\in \bbH_{\rN}(\bdiv;\Omega) } \frac{\cR_1({\btau})}{\|\btau\|_{\bdiv;\Omega} }+\;\sup_{\bzero\neq(\bv,\bbeta)\in  \bL^2(\Omega)\times\bbLskew } \frac{b({\tt e}_{\bsigma},(\bv,\bbeta))}{ \|\bv\|_{0,\Omega} + \|\bbeta\|_{0,\Omega}}\Big).
\end{align*}
Then, recalling the definitions of the bilinear form $b$ (cf. \eqref{eq:forms}), using the equation \eqref{eq:2.2b}, along with the fact that $\int_{\Omega} \bsigma_h:\bbeta =\frac{1}{2} \int_{\Omega} (\bsigma_h-\bsigma_h^{\tt t}):\bbeta$ for $\bbeta \in \bbLskew $, the following estimate holds
\[|b({\tt e}_{\bsigma},(\bv,\bbeta))|\le ( ||\fb+\bdiv\,\bsigma_h||_{0,\Omega} + ||\bsigma_h-\bsigma_h^{\tt t}||_{0,\Omega})(\|\bv\|_{0,\Omega} + \|\bbeta\|_{0,\Omega}),\]
and this gives the asserted inequality. 
\end{proof}

\begin{lemma}\label{lem-res-2}
There exists $C_2>0$ such that
\begin{align*}
&\|\bvarphi-\bvarphi_h\|_{\div;\Omega} + \|p-p_h\|_{0,\Omega}\\
& \quad \leq C_2 \big(\|\cR_2\|+ \|g-c_0 p_h - \dfrac{\alpha}{d\lambda+2\mu}\tr\bsigma_h - \dfrac{d\alpha^2}{d\lambda+2\mu} p_h + \div \bvarphi_h \|_{0,\Omega} + \gamma\|\bsigma-\bsigma_h \|_{\bdiv;\Omega}\big),
\end{align*}
where
\begin{equation*}%\label{eq:def-R-2}
\cR_2 (\bpsi):=\wt a_{\bsigma,p}(\bvarphi-\bvarphi_h,\bpsi) + \wt b(\bpsi,p-p_h),  
\end{equation*}
satisfies $\cR_2 (\bpsi_h)=0$ for all $\bpsi_h \in \bH^{\bvarphi}_h$, and $\ds ||\cR_2||=
    \sup_{\bzero\neq\bpsi  \in \bH_{\rN}(\div;\Omega) }   \frac{\cR_2(\bpsi)}{\|\bpsi\|_{\div;\Omega} }.$
\end{lemma}

\begin{proof}
Similarly to Lemma~\ref{lem:CEA-1}, using the properties of bilinear forms $\wt a_{\bsigma,p}$ and $\wt b$ (as outlined in equations \eqref{eq:ellipticity-hat-a}, \eqref{eq:infsup-hat-b} and \eqref{eq:ellipticity-hat-c}), along with the insight from \cite[Th.~3.4]{cg2022}, we establish that there exists $C_2>0$ depending on $\kappa_1$, $\kappa_2$, $C_{\wt a}$, $\gamma$, $\wt\gamma$, $\wt\beta$ such that
\begin{align*}
\|{\tt e}_{\bvarphi}\|_{\div;\Omega} + \|{\tt e}_{p}\|_{0,\Omega}&\leq C_2 \sup_{\bzero\neq(\bpsi,q)\in \bH_{\rN}(\div;\Omega)\times \rL^2(\Omega) } \frac{\wt a_{\bsigma,p}({\tt e}_{\bvarphi},\bpsi) +\wt  b(\bpsi,{\tt e}_{p}) +\wt  b({\tt e}_{\bvarphi},q) - \wt c({\tt e}_p,q)}{\|\bpsi\|_{\div;\Omega} + \|q\|_{0,\Omega}}\\
&\leq C_2 \;\Big(\sup_{\bzero\neq \bpsi\in \bH_{\rN}(\div;\Omega) } \frac{\cR_2(\bpsi)}{\|\bpsi\|_{\div;\Omega} }+\;\sup_{0\neq q\in \rL^2(\Omega)} \frac{ \wt b({\tt e}_{\bvarphi},q) - \wt c({\tt e}_p,q)}{ \|q\|_{0,\Omega}}\Big).
\end{align*}
Hence, recalling the definitions of   $\wt b,\,\wt c$, adding $\pm\, c\, (\bsigma_h,p)$, and using \eqref{eq:2.2d}, we arrive at 
\[|\wt b({\tt e}_{\bvarphi},q) - \wt c({\tt e}_p,q)|\le  \|g-c_0 p_h - \dfrac{\alpha}{d\lambda+2\mu}\tr\bsigma_h - \dfrac{d\alpha^2}{d\lambda+2\mu} p_h + \div \bvarphi_h \|_{0,\Omega}\|q\|_{0,\Omega} + \gamma\|\bsigma-\bsigma_h \|_{\bdiv;\Omega}\|q\|_{0,\Omega},\]
and therefore, we obtain the desired result. 
\end{proof}

Throughout the rest of this section, we provide suitable upper bounds for $\cR_1$ and $\cR_2$. We begin by establishing the corresponding estimates for $\cR_1$, which are based on a suitable Helmholtz decomposition of the space $\bbH_\rN(\bdiv;\Omega)$ (see \cite[Lemma 3.9]{alvarez16} for details), along with the following two technical results.
\begin{lemma}\label{eq:rel-bound1}
There exists a positive constant $C_3$, independent of $h$, such that for each $\bxi \in \bbH^1(\Omega)$ there holds 
\begin{equation*}
\begin{array}{ll}
|\cR_1(\bxi-\bPi_h(\bxi))|\,\le  C_3 \,\gamma \|p-p_h\|_{0,\Omega}\|\bxi\|_{1,\Omega}\\
\quad \ds + C_3\biggr(\sum_{K\in\mathcal{T}_h} h_K^2||\C^{-1}\bsigma_h + \dfrac{\alpha}{d\lambda+2\mu} p_h \mathbf{I} - \bnabla \bu_h +\brho_h||_{0,K}^2 + \sum_{E \in \mathcal{E}_h(\Gamma_\rD)}h_E||\bu_\rD-\bu_h||_{0,E}^2 \biggr)^{1/2}\|\bxi\|_{1,\Omega}.
\end{array}
\end{equation*}
\end{lemma}
\begin{proof}
From the definition of $\cR_1$ (cf. \eqref{eq:def-R-1}), adding $\pm\, c\, (\bxi-\bPi_h(\bxi),p_h)$, and using equation \eqref{eq:2.2a}, we have
\begin{equation*}
\begin{array}{ll}
\ds \cR_1(\bxi-\bPi_h(\bxi))= H(\bxi-\bPi_h(\bxi)) - c(\bxi-\bPi_h(\bxi),p) -a(\bsigma_h,\bxi-\bPi_h(\bxi)) 
-b(\bxi-\bPi_h(\bxi),(\bu_h,\brho_h))\\
\ds \quad = \langle (\bxi-\bPi_h(\bxi))\bn,\bu_\rD\rangle_{\Gamma_\rD} - \dfrac{\alpha}{d\lambda+2\mu} \int_{\Omega} (p-p_h)\, \tr(\bxi-\bPi_h(\bxi)) - \dfrac{\alpha}{d\lambda+2\mu} \int_{\Omega} p_h\, \tr(\bxi-\bPi_h(\bxi))\\
\ds \qquad - \int_{\Omega} \C^{-1}\bsigma_h:(\bxi-\bPi_h(\bxi)) - \int_{\Omega} \brho_h:(\bxi-\bPi_h(\bxi))  - \int_{\Omega} \bu_h\cdot\bdiv(\bxi-\bPi_h(\bxi)),
\end{array}
\end{equation*}
then, applying a local integration by parts to the last term above, using the identity
$\int_E \bPi_h(\btau)\bn\cdot\bxi=\int_E \btau\bn\cdot \bxi$, for all $\bxi\in \bP_k(E)$, for all edge $E$ of $\mathcal T_h$, the fact that $\bu_\rD\in \bL^2(\Gamma_\rD)$, and the Cauchy-Schwarz inequality, we obtain
\begin{equation*}
\begin{array}{ll}
\ds \cR_1(\bxi-\bPi_h(\bxi)) = \sum_{K\in\mathcal{T}_h} \int_K (-\C^{-1}\bsigma_h - \dfrac{\alpha}{d\lambda+2\mu} p_h \mathbf{I} + \bnabla \bu_h -\brho_h):(\bxi-\bPi_h(\bxi))\\
\qquad \qquad \ds +\sum_{E \in \mathcal{E}_h(\Gamma_\rD)}\langle (\bxi-\bPi_h(\bxi))\bn,\bu_\rD - \bu_h\rangle_E   - \dfrac{\alpha}{d\lambda+2\mu} \int_{\Omega} (p-p_h)\, \tr(\bxi-\bPi_h(\bxi))\\
\qquad \ds \leq \sum_{K\in\mathcal{T}_h} \|\C^{-1}\bsigma_h + \dfrac{\alpha}{d\lambda+2\mu} p_h \mathbf{I} - \bnabla \bu_h +\brho_h\|_{0,K} \|\bxi-\bPi_h(\bxi)\|_{0,K}\\
\qquad \qquad \ds +\sum_{E \in \mathcal{E}_h(\Gamma_\rD)} \|\bu_\rD - \bu_h\|_{0,E} \|\bxi-\bPi_h(\bxi)\|_{0,E} + \gamma \|p-p_h\|_{0,\Omega}\|\bxi-\bPi_h(\bxi)\|_{0,\Omega},
\end{array}
\end{equation*}
with $\gamma$ defined as in \eqref{eq:def-gamma}. Therefore, using the approximations properties of $\bPi_h$ (see, e.g., \cite[Section~3]{alvarez16}) and the Cauchy--Schwarz inequality, we obtain the desired result.
\end{proof}
\begin{lemma}\label{eq:rel-bound2}
Let $\bchi \in \bH^1_{\Gamma_\rN}(\Omega):= \{\bw\in \bH^1(\Omega):\, \bw = \bzero \ \text{on}\ \Gamma_\rN\}$ and assume that $\bu_\rD\in \bH^1(\Gamma_\rD)$. Then,  there exists $C_4>0$, independent of $h$, such that 
\begin{equation*}
\begin{array}{ll}
\ds  |\cR_1(\bcurl (\bchi - \bI_h \bchi) )| 
 \leq C_4 \gamma \| p- p_h||_{0,\Omega} \|\bchi\|_{1,\Omega} \\[2ex]
\ds + C_4\biggr(\,\sum_{K\in\mathcal{T}_h} h_K^2||\underline{\mathbf{curl}}(\C^{-1}\bsigma_h + \dfrac{\alpha}{d\lambda+2\mu} p_h \mathbf{I}  +\brho_h)||_{0,K}^2 \\[2ex]
\ds\qquad +\sum_{E\in \mathcal{E}_h(\Omega)}h_E ||[\![(\C^{-1}\bsigma_h + \dfrac{\alpha}{d\lambda+2\mu} p_h \mathbf{I} +\brho_h)\bs]\!]||_{0,E}^2\\[2ex]
\ds \qquad+\sum_{E\in \mathcal{E}_h(\Gamma_\rD)}h_E ||(C^{-1}\bsigma_h + \dfrac{\alpha}{d\lambda+2\mu} p_h \mathbf{I}  +\brho_h)\bs -\dfrac{\mathrm{d} \bu_\rD}{\mathrm{d}\bs}||_{0,E}^2\biggr)^{1/2} \|\bchi\|_{1,\Omega}.
\end{array}
\end{equation*}
\end{lemma}
%(\bpl{$\bnabla_{\Gamma}$ is not defined.})
\begin{proof}
Similarly to Lemma~\ref{eq:rel-bound1}, adding $\pm\, c\, (\bcurl (\bchi - \bI_h \bchi),p_h)$, we have
\begin{equation*}
\begin{array}{ll}
\cR_1(\bcurl  (\bchi - \bI_h \bchi))\\[1ex]
\quad\ds = H(\bcurl (\bchi - \bI_h \bchi)) - c(\bcurl (\bchi - \bI_h \bchi),p) -a(\bsigma_h,\bcurl (\bchi - \bI_h \bchi)) \\[1ex]
\qquad \ds - b(\bcurl (\bchi - \bI_h \bchi),(\bu_h,\brho_h))\\[1ex]
\ds \quad = \langle (\bcurl (\bchi - \bI_h \bchi))\bn,\bu_\rD\rangle_{\Gamma_\rD} - \dfrac{\alpha}{d\lambda+2\mu} \int_{\Omega} (p-p_h)\, \tr(\bcurl (\bchi - \bI_h \bchi)) \\[2ex]
\ds \quad  - \int_{\Omega} \C^{-1}\bsigma_h:\bcurl (\bchi - \bI_h \bchi) - \int_{\Omega} \brho_h:\bcurl (\bchi - \bI_h \bchi) - \dfrac{\alpha}{d\lambda+2\mu} \int_{\Omega} p_h\, \tr(\bcurl (\bchi - \bI_h \bchi)).
\end{array}
\end{equation*}
Then, applying a local integration by parts, using that $\bu_\rD\in \bH^1(\Gamma_\rD)$, the identity $\langle \bcurl(\bchi-\bI_h \bchi)  \bn ,\bu_\rD \rangle_{\Gamma_\rD}= -\langle \bchi-\bI_h \bchi, \frac{\mathrm{d}\bu_\rD}{\mathrm{d}\bs} \rangle_{\Gamma_\rD}$, and the Cauchy--Schwarz inequality, we obtain
\begin{align*}
\ds \cR_1(\bcurl  (\bchi - \bI_h \bchi)) &= -\sum_{K\in\mathcal{T}_h} \int_K \underline{\bcurl}(\C^{-1}\bsigma_h + \dfrac{\alpha}{d\lambda+2\mu} p_h \mathbf{I} +\brho_h):(\bchi - \bI_h \bchi)\\
\ds & \qquad  +\sum_{E\in \mathcal{E}_h(\Omega)} \int_E [\![(\C^{-1}\bsigma_h + \dfrac{\alpha}{d\lambda+2\mu} p_h \mathbf{I} +\brho_h)\bs]\!]\cdot(\bchi - \bI_h \bchi) \\
\ds &\qquad - \dfrac{\alpha}{d\lambda+2\mu} \int_{\Omega} (p-p_h)\, \tr(\bcurl (\bchi - \bI_h \bchi))\\
\ds &\qquad  +\sum_{E\in \mathcal{E}_h(\Gamma_\rD)} \int_E (\C^{-1}\bsigma_h + \dfrac{\alpha}{d\lambda+2\mu} p_h \mathbf{I} +\brho_h -\bnabla\bu_\rD)\bs \cdot(\bchi - \bI_h \bchi)\\
 &\ds  \leq \sum_{K\in\mathcal{T}_h} \|\underline{\bcurl}(\C^{-1}\bsigma_h + \dfrac{\alpha}{d\lambda+2\mu} p_h \mathbf{I} +\brho_h)\|_{0,K}\|\bchi - \bI_h \bchi\|_{0,K}\\
\ds &\qquad  +\sum_{E\in \mathcal{E}_h(\Omega)} \|[\![(\C^{-1}\bsigma_h + \dfrac{\alpha}{d\lambda+2\mu} p_h \mathbf{I} +\brho_h)\bs]\!]\|_{0,E}\|\bchi - \bI_h \bchi\|_{0,E} \\
&\qquad \ds +\gamma \|p-p_h\|_{0,\Omega}\, \|\bcurl (\bchi - \bI_h \bchi)\|_{0,\Omega}\\
\ds &\qquad  +\sum_{E\in \mathcal{E}_h(\Gamma_\rD)} \| (\C^{-1}\bsigma_h + \dfrac{\alpha}{d\lambda+2\mu} p_h \mathbf{I} +\brho_h)\bs -\dfrac{\mathrm{d} \bu_\rD}{\mathrm{d}\bs}\|_{0,E} \|\bchi - \bI_h \bchi\|_{0,E}.
\end{align*}
As in the previous result, the approximation properties of the  Cl\'ement interpolation (see, e.g., \cite[Section~3]{alvarez16}) in conjunction with the Cauchy--Schwarz inequality, produces the desired result.
\end{proof}

The following lemma establishes the desired estimate for $\cR_1$.

\begin{lemma}\label{lem:cota-R1}
Assume that there exists a convex domain $B$ such that $\Omega\subseteq B$ and $\Gamma_\rN\subseteq \partial B$, and that $\bu_\rD\in \bH^1(\Gamma_\rD)$. Then, there exists a constant $C_5>0$, independent of $h$, such that
\begin{equation*}
\|\bsigma-\bsigma_h\|_{\bdiv;\Omega} + \|\bu-\bu_h\|_{0,\Omega} + \|\brho-\brho_h\|_{0,\Omega}\,\leq\, C_5 \Bigg\{\,\sum_{T\in\mathcal{T}_{h}}\Xi_{s,K}^2\Bigg\}^{1/2} + C_5\gamma\|p-p_h\|_{0,\Omega}.
\end{equation*}
%with $\Xi_{s,K}$ defined as in \eqref{eq:res-1}.
\end{lemma}

\begin{proof}
Let $\btau \in \bbH_\rN(\bdiv;\Omega)$.  From \cite[Lemma 3.9]{alvarez16}, there exist $\bxi \in \bbH^1(\Omega)$ and $\bchi \in \bH_{\Gamma_\rN}^1(\Omega)$, such that
\begin{equation}\label{eq:descHelm-1}
\btau=\bxi + \bcurl \bchi \qan  \|\bxi\|_{1,\Omega}+\|\bchi\|_{1,\Omega}\leq   C_{\mathrm{Helm}}\|\btau\|_{\bdiv; \Om},   
\end{equation}
Using that $\cR_1(\btau_h)=c(\btau_h,p_h-p)$ for all $\btau_h\in\bbH^{\bsigma}_{h}$, we have
$$\cR_1(\btau)=\cR_1(\btau-\btau_h) + c(\btau_h,p-p_h) \quad \forall\,\btau_h\in\bbH^{\bsigma}_{h}.$$
In particular, this holds for $\btau_h$ defined as $\ds\btau_h=\bPi_h\bxi+\bcurl(\bI_h \bchi)$, whence
$$\cR_1(\btau)=\cR_1(\bxi-\bPi_h\bxi)+\cR_1(\curl(\bchi-\bI_h \bchi)) + c(\btau_h,p-p_h).$$
Hence, the proof follows from Lemmas~\ref{lem-res-1}, \ref{eq:rel-bound1} and \ref{eq:rel-bound2}, and estimate \eqref{eq:descHelm-1}.
\end{proof}

The following lemma establishes the estimate for $\cR_2$.

\begin{lemma}\label{lem:cota-R2}
Assume that there exists a convex domain $B$ such that $\Omega\subseteq B$ and $\Gamma_\rN\subseteq \partial B$, and that $p_\rD\in \rH^1(\Gamma_\rD)$. Then, there exists a constant $C_8>0$, independent of $h$, such that
\begin{equation*}
\|\bvarphi-\bvarphi_h\|_{\div;\Omega} + \|p-p_h\|_{0,\Omega}\,\leq\, C_8 \Bigg\{\,\sum_{T\in\mathcal{T}_{h}}\Xi_{f,K}^2\Bigg\}^{1/2} + C_8\big(\|\bvarphi_h\|_{0,\Omega}\|p-p_h\|_{0,\Omega} + \gamma\|\bsigma-\bsigma_h \|_{0,\Omega} \big).
\end{equation*}
%with $\Xi_{f,K}$ defined as in \eqref{eq:res-2}.
\end{lemma}
\begin{proof}
It follows the steps of Lemma \ref{lem:cota-R1}. From \cite[Lemma~4.4]{cov2020c}, we have that for all $\bpsi\in \bH_\rN(\div;\Omega)$ there exist $\bz \in \bH^1(\Omega)$ and $\phi \in \rH_{\Gamma_\rN}^1(\Omega)$, such that $\bpsi=\bz + \curl \phi$ and $\|\bz\|_{1,\Omega}+\|\phi\|_{1,\Omega}\leq   \wt C_{\mathrm{Helm}}\|\bpsi\|_{\bdiv; \Om}$.
Thus, proceeding similarly to Lemmas \ref{eq:rel-bound1} and \ref{eq:rel-bound2}, applying local integration by parts and the approximation properties of $\Pi_h$ and $I_h$ along with the Cauchy--Schwarz inequality, we can obtain the following estimates
\begin{equation}\label{eq:stimates-R2}
\begin{array}{ll}
|\cR_2(\bz-\Pi_h(\bz))|\,\le  C_6 \,\|\bvarphi_h\|_{0,\Omega}\|p-p_h\|_{0,\Omega}\|\bz\|_{1,\Omega}\\
\qquad \ds + C_6\biggr(\sum_{K\in\mathcal{T}_h} h_K^2\|\bkappa(\bsigma_h,p_h)^{-1}\bvarphi_h- \nabla p_h\|_{0,K}^2 + \sum_{E \in \mathcal{E}_h(\Gamma_\rD)}h_E \|p_\rD - p_h\|_{0,E}^2 \biggr)^{1/2}\|\bz\|_{1,\Omega},\\[4ex]
\ds  |\cR_2(\curl (\phi - I_h \phi) )| 
 \leq C_7 \|\bvarphi_h\|_{0,\Omega}\|p-p_h\|_{0,\Omega}\|\phi\|_{1,\Omega} \\[2ex]
\qquad \ds + C_7\biggr(\,\sum_{K\in\mathcal{T}_h} h_K^2\|\rot(\bkappa(\bsigma_h,p_h)^{-1}\bvarphi_h)\|_{0,K}^2 +\sum_{E\in \mathcal{E}_h(\Omega)}h_E \|[\![(\bkappa(\bsigma_h,p_h)^{-1}\bvarphi_h)\cdot\bs]\!]\|_{0,E}^2\\[2ex]
\qquad \ds +\sum_{E\in \mathcal{E}_h(\Gamma_\rD)}h_E \|\bkappa(\bsigma_h,p_h)^{-1}\bvarphi_h\cdot\bs -\frac{\mathrm{d} p_\rD}{\mathrm{d}\bs} \|_{0,E}^2\biggr)^{1/2} \|\phi\|_{1,\Omega}.
\end{array}
\end{equation}
Then, noting that $\cR_2(\bpsi_h)=0$ for all $\bpsi_h\in \bH^{\bvarphi}_{h}$, and defining $\bpsi_h$ as $\ds\bpsi_h=\Pi_h\bz+\curl(I_h \phi)$, we have
$$\cR_2(\bpsi)=\cR_2(\bpsi-\bpsi_h)=\cR_2(\bz-\Pi_h\bz)+\cR_2(\curl(\phi-I_h \phi)).$$
Hence, the proof follows from Lemma~\ref{lem-res-2}, estimates \eqref{eq:stimates-R2}, and the Helmholtz decomposition of $\bH_\rN(\div;\Omega)$.
\end{proof}

Finally we state the main reliability bound for the proposed estimator.
\begin{theorem} \label{eq:rel-2}
Assume the hypotheses stated in Theorem~\ref{th:cea-estimate} and Lemmas \ref{lem:cota-R1}-\ref{lem:cota-R2}. Let $(\bsigma,\bu,\brho,\bvarphi,p)\in \bbH_\rN(\bdiv;\Omega)\times\bL^2(\Omega)\times\bbLskew\times \bH_\rN(\div;\Omega)\times\rL^2(\Omega)$ and $(\bsigma_h,\bu_h,\brho_h,\bvarphi_h,p_h)\in  \bbH^{\bsigma}_{h}\times\bH^{\bu}_{h}\times\bbH^{\brho}_{h}\times \bH^{\bvarphi}_{h}\times\rH^{p}_{h}$ be the unique solutions of %the continuous and discrete problems 
\eqref{eq:weak-formulation} and \eqref{eq:weak-formulation-h}, respectively.
Assume further that
\begin{equation}\label{eq:assumption-rel}
(C_5+C_8)\gamma + C_8 C^* (\|g\|_{0,\Omega} + \|p_{\rD}\|_{1/2,\Gamma_\rD} + \|\bu_\rD\|_{1/2,\Gamma_\rD} + \|\fb\|_{0,\Omega} + \gamma \,r\,) \leq\frac{1}{2}.
\end{equation}
Then, there exists $C_{\mathrm{rel}}>0$, independent of $h$, such that 
\begin{align*}
\ds \|\texttt{e}_{\bsigma}\|_{\bdiv;\Omega} + \|\texttt{e}_{\bu}\|_{0,\Omega} + \|\texttt{e}_{\brho}\|_{0,\Omega} + \|\texttt{e}_{\bvarphi}\|_{\div;\Omega} + \|\texttt{e}_{p}\|_{0,\Omega}
\leq\,C_{\mathrm{rel}}\, \Xi.
\end{align*}

\end{theorem}
\begin{proof}
It follows directly from  Lemmas \ref{lem:cota-R1} and \ref{lem:cota-R2}, using the fact that $\bvarphi_h$ satisfies the estimate \eqref{eq:stability-h}, and the assumption \eqref{eq:assumption-rel}.
\end{proof}
%%%%%%%%%%%%%%%%%%%%%%%%%%%%%%%%%%%%%%%%%%%%%%%%%%%%%%%%%%%%%%%%%%%%%%%%
\subsection{Efficiency of the a posteriori error estimator}%\label{section:efficiency}

In this section we derive the efficiency estimate of the estimator defined in \eqref{eq:estimator}. The main result of this section is stated  as follows. %To that end, and for the sake of simplicity, from now on we assume that the Dirichlet boundary data $\bu_\Gamma$ and $p_\Sigma$ are a piecewise polynomial. By doing that, we prevent, on the one hand, the introduction of additional terms in the estimator involving the projection of  $\bu_\Gamma$ and $p_\Sigma$ and/or the appearance of high-order terms (${\rm h.o.t.}$). More precisely, we obtain the following main result of this section.
\begin{theorem}\label{th:efficiency}
There exists $C_{\mathrm{eff}}>0$, independent of $h$, such that
\begin{equation}\label{eq:efficiency}
C_{\mathrm{eff}}\,\Xi \,\leq\, \|\bsigma-\bsigma_h\|_{\bdiv;\Omega} + \|\bu-\bu_h\|_{0,\Omega} + \|\brho-\brho_h\|_{0,\Omega} + \|\bvarphi-\bvarphi_h\|_{\div;\Omega} + \|p-p_h\|_{0,\Omega} + {\rm h.o.t},
\end{equation}
where ${\rm h.o.t.}$ stands for one or several terms of higher order.
\end{theorem}
We begin with the estimates for the zero order terms appearing in the definition of $\Xi_{s,K}$ (cf. \eqref{eq:res-1}).
\begin{lemma}\label{lem:cota-eff-1}
For all $K\in \cT_h$ there holds
\begin{equation*}%\label{ef1}
\|\fb +\bdiv \bsigma_h \|_{0,K}\lesssim \|\bsigma-\bsigma_h\|_{\bdiv;K}\, \qan  \|\bsigma_h-\bsigma^{\tt t}_h\|_{0,K}\lesssim  \|\bsigma-\bsigma_h\|_{\bdiv;K}.
\end{equation*}
\end{lemma}
\begin{proof}
By employing the same arguments as in \cite[Theorem~3.2]{ccov2022}, we can conclude that $ \fb= -\bdiv\,\bsigma$, which, together with the symmetry of $\bsigma$,  implies the desired result. Further details are omitted.
\end{proof}

In order to derive the upper bounds for the remaining terms
defining the error estimator $\Xi_{s,K}$, we use results from \cite{carstensen1998posteriori}, inverse inequalities, and the localisation technique based on element-bubble and edge-bubble functions. The main properties that we will use can be found in \cite[Lemmas~4.4-4.7]{gatica22}.

\begin{lemma}\label{lem:cota-eff-2}
For all $K\in \cT_h$ there holds
\begin{equation*}%\label{ef2}
\begin{array}{ll}
h_K\|\mathcal{C}^{-1}\bsigma_h + \dfrac{\alpha}{d\lambda+2\mu} p_h \mathbf{I} - \bnabla \bu_h +\brho_h\|_{0,K}\\
\ds\qquad \lesssim h_K\big(\|\bsigma-\bsigma_h\|_{\bdiv;K}  + \|\brho-\brho_h\|_{0,K} + \|p-p_h\|_{0,K}\big) + \|\bu-\bu_h\|_{0,K}.
\end{array}
\end{equation*}
\end{lemma}
\begin{proof}
It follows from an application of \cite[Lemma~4.4]{gatica22} with $\bq= \mathcal{C}^{-1}\bsigma_h + \dfrac{\alpha}{d\lambda+2\mu} p_h \mathbf{I} - \bnabla \bu_h +\brho_h$, using that $\C^{-1}\bsigma + \dfrac{\alpha}{d\lambda+2\mu} p \mathbf{I} = \bnabla \bu -\brho$ and 
\cite[Lemma~4.5]{gatica22}. We refer to \cite[Lemma~6.6]{carstensen1998posteriori} for further details.
\end{proof}

\begin{lemma}\label{lem:cota-eff-3}
For all $K\in \cT_h$ and $E\in  \mathcal{E}_h(\Omega)$, there holds
\begin{equation*}%\label{ef3}
\begin{array}{ll}
h_K\|\underline{\mathbf{curl}}(\mathcal{C}^{-1}\bsigma_h + \dfrac{\alpha}{d\lambda+2\mu} p_h \mathbf{I}  +\brho_h)\|_{0,K} \lesssim \|\bsigma-\bsigma_h\|_{\bdiv;K} + \|\brho-\brho_h\|_{0,K} + \|p-p_h\|_{0,K},\\[3ex]
h_E^{1/2} \|[\![(\mathcal{C}^{-1}\bsigma_h + \dfrac{\alpha}{d\lambda+2\mu} p_h \mathbf{I}  +\brho_h)\bs]\!]\|_{0,E} \lesssim \|\bsigma-\bsigma_h\|_{\bdiv;\omega_E} + \|\brho-\brho_h\|_{0,\omega_E} + \|p-p_h\|_{0,\omega_E},
\end{array}
\end{equation*}
where the patch of elements sharing the edge $E$ is denoted as $\omega_E:=\cup\{K'\in \mathcal{T}_h\colon E\in \mathcal{E}_h(K')\}$.
\end{lemma}
\begin{proof}
It suffices to apply \cite[Lemma~4.7]{gatica22} with $\bxi:=\C^{-1}\bsigma + \dfrac{\alpha}{d\lambda+2\mu} p \mathbf{I} +\brho = \bnabla \bu $ and $\bxi_h:=\C^{-1}\bsigma_h + \dfrac{\alpha}{d\lambda+2\mu} p_h \mathbf{I} +\brho_h$.
\end{proof}

\begin{lemma}\label{lem:cota-eff-4}
Assume that $\bu_\rD$ is piecewise polynomial. Then, for all $E\in  \mathcal{E}_h(\Gamma_\rD)$, there holds
\begin{equation*}%\label{ef4}
\begin{array}{ll}
h_E^{1/2} \|(
\mathcal{C}^{-1}\bsigma_h + \dfrac{\alpha}{d\lambda+2\mu} p_h \mathbf{I} +\brho_h)\bs -\frac{\mathrm{d} \bu_\rD}{\mathrm{d}\bs}\|_{0,E} \lesssim  \|\bsigma-\bsigma_h\|_{\bdiv;K_E}\! +\! \|\brho-\brho_h\|_{0,K_E}\! +\! \|p-p_h\|_{0,K_E},\\[2ex]
h_E^{1/2} \|\bu_\rD-\bu_h\|_{0,E} \lesssim  h_{K_E}\big(\|\bsigma-\bsigma_h\|_{\bdiv;K_E}  + \|\brho-\brho_h\|_{0,K_E} + \|p-p_h\|_{0,K_E}\big) + \|\bu-\bu_h\|_{0,K_E},
\end{array}
\end{equation*}
where $K_E$ is a triangle in $\mathcal{T}_h$  that contains  $E$ on its boundary.
\end{lemma}
\begin{proof}
The first estimate follows as in \cite[Lemma~4.18]{gatica22}, defining $\bxi$ and $\bxi_h$ as in Lemma \ref{lem:cota-eff-3}. On the other hand, the second estimate follows from an application of the discrete trace inequality (see \cite[Lemma~4.6]{gatica22}), using that $\C^{-1}\bsigma + \dfrac{\alpha}{d\lambda+2\mu} p \mathbf{I} = \bnabla \bu -\brho$, and the fact that $\bu=\bu_\rD$ on $\Gamma_\rD$. See also  \cite[Lemma~4.14]{gms2010}.
\end{proof}

A direct application of Lemmas \ref{lem:cota-eff-1}-\ref{lem:cota-eff-4} yields 
\begin{equation}\label{eq:bound-eff-1}
\sum_{K\in\mathcal{T}_h}\Xi_{s,K}\lesssim  \|\texttt{e}_{\bsigma}\|_{\bdiv;\Omega} + \|\texttt{e}_{\bu}\|_{0,\Omega} + \|\texttt{e}_{\brho}\|_{0,\Omega} + \|\texttt{e}_{\bvarphi}\|_{\div;\Omega} + \|\texttt{e}_{p}\|_{0,\Omega}.
\end{equation}
Similarly, using the same arguments as in Lemmas \ref{lem:cota-eff-1}-\ref{lem:cota-eff-4}, along with algebraic manipulations as in Section \ref{sec:error}, assuming that $p_\rD$ is piecewise polynomial, together with the Lipschitz continuity of $\bkappa$ (cf. \eqref{prop-kappa}), we can bound each of the terms that appear in the estimator $\Xi_{f,K}$ and obtain the following result
\begin{equation}\label{eq:bound-eff-2}
\sum_{K\in\mathcal{T}_h}\Xi_{f,K}\lesssim  \|\texttt{e}_{\bsigma}\|_{\bdiv;\Omega} + \|\texttt{e}_{\bu}\|_{0,\Omega} + \|\texttt{e}_{\brho}\|_{0,\Omega} + \|\texttt{e}_{\bvarphi}\|_{\div;\Omega} + \|\texttt{e}_{p}\|_{0,\Omega}.
\end{equation}
We remark that the efficiency of $\Xi$ (cf. \eqref{eq:estimator}) in Theorem \ref{th:efficiency} is now a straightforward consequence of estimates \eqref{eq:bound-eff-1} and \eqref{eq:bound-eff-2}. In turn, we emphasize that the resulting constant, denoted by $C_{\mathrm{eff}}>0$ is independent of $h$.
\begin{remark}
For simplicity, we have assumed that $\bu_\rD$ and $p_\rD$ are piecewise polynomial in the derivation of \eqref{eq:bound-eff-1} and \eqref{eq:bound-eff-2}. However, similar estimates can also be obtained by assuming $\bu_\rD$ and $p_\rD$ are sufficiently smooth (taking, for example, $\bu_\rD\in \bH^1(\Gamma_\rD)$ and $p_\rD\in \rH^1(\Gamma_\rD)$, as in Lemmas~\ref{lem:cota-R1}-\ref{lem:cota-R2}), and proceeding as in \cite[Section~6.2]{cmo2016}. In such cases, higher-order terms, stemming from errors in the polynomial approximations, would appear in \eqref{eq:bound-eff-1} and \eqref{eq:bound-eff-2}, accounting for the presence of ${\rm h.o.t}$ in \eqref{eq:efficiency}.
\end{remark}
\begin{remark}
We conclude this section by noting that the a posteriori error estimation analysis developed here can be adapted to the three-dimensional case. In particular, in \cite[Th~3.2]{gatica2020heldes} and \cite[Lemma~4.4]{cov2020c}, one can find the suitable Helmholtz decompositions for the spaces $\bbH_\rN(\bdiv;\Omega)$ and $\bH_\rN(\div;\Omega)$, respectively.
\end{remark}

%%%%%%%%%%%%%%%%%%%%%%%%%%%%%%%%%%%%%%%%%%%%%%%%%%%%%
\section{Numerical results}\label{sec:numer}
The computational examples in this section  verify the theoretical properties (optimal convergence, conservativity,  \revX{and parameter robustness}) of the proposed schemes. The implementation has been carried out using the FE library FEniCS \cite{alnaes15}. The nonlinear  systems were solved with Newton--Raphson's method with  a residual tolerance of $10^{-7}$, and the linear systems were solved using the sparse LU factorisation of MUMPS \cite{Amestoy2000}. 

\subsection{Optimal convergence to smooth solutions and conservativity in 2D}
We first consider a simple planar problem setup with manufactured exact solution. We take the unit square domain $\Omega = (0,1)^2$, the bottom and left segments represent $\Gamma_\rD$ and the top and right sides are $\Gamma_\rN$. We choose the body load $\boldsymbol{f}$, mass source $g$, boundary displacement $\bu_\rD$, boundary pressure $p_\rD$, as well as (not necessarily homogeneous,  but standard arguments can be used to extend the theory to the inhomogeneous case) boundary data $\bvarphi\cdot\bn = \varphi_\rN$ and \revX{$\bsigma\bn = \bsigma_\rN$}, such that the exact displacement and fluid pressure are 
\[ \bu(x,y) = \frac{1}{20}\begin{pmatrix}
    \cos\bigl[\frac{3\pi}{2}(x+y)\bigr]\\[1ex] \sin\bigl[\frac{3\pi}{2}(x-y)\bigr]
\end{pmatrix}, \quad p(x,y) = \sin(\pi x) \sin (\pi y).\]
These exact primary variables are used to construct exact mixed variables of stress, rotation, and discharge flux. We choose the second constitutive relation for the permeability in \eqref{const-kappa-2} and choose the following arbitrary model parameters (all \revX{nondimensional})
$k_0 = k_1 = c_0 = \alpha = 0.1, \quad \lambda = \mu = \mu_f = 1$.
These values indicate a mild permeability variation and it is expected that the nonlinear solver (in this case, Newton--Raphson) converges in only a few iterations. We construct six levels of uniform mesh refinement of the domain, on which we compute approximate solutions and the associated errors for each primal and mixed variable in their natural norms. Convergence rates are calculated as usual: 
\[
\text{rate} =\log({e}/\widehat{{e}})[\log(h/\widehat{h})]^{-1}\,,\]
where ${e}$ and $\widehat{{e}}$ denote errors produced on two consecutive meshes of sizes $h$ and $\widehat{h}$, respectively. Table \ref{tab:my_label} reports on this error history focusing on the methods defined by the PEERS$_k$ family with $k=0$ and $k=1$, showing a $O(h^{k+1})$ convergence for all unknowns as expected from the theoretical error bound of Theorem~\ref{th:rate1} (except for the rotation approximation that shows a slight superconvergence for the lowest-order case and only in 2D -- a well-known phenomenon associated with PEERS$_k$ elements). With the purpose of illustrating the character of the chosen manufactured solution and the parameter regime, we show sample discrete solutions in Figure~\ref{fig:enter-label}.

\begin{table}[t!]
\setlength{\tabcolsep}{4pt}
    \centering
    \begin{tabular}{|c|rc|cccccccccc|}
    \hline
$k$ & DoFs & $h$ & $e(\bsigma)$ & rate & $e(\bu)$  & rate & $e(\brho)$  & rate & $e(\bvarphi)$  & rate & $e(p)$  & rate \\
    \hline 
     \multirow{6}{*}{0} &  98 & 0.7071 & 3.5e+0 & $\star$ & 4.5e-02 & $\star$ & 3.7e-01 & $\star$ & 4.9e-01 & $\star$ & 2.4e-01 & $\star$\\
&   354 & 0.3536 & 1.9e+0 & 0.86 & 2.0e-02 & 1.20 & 8.0e-02 & 2.19 & 2.6e-01 & 0.92 & 1.3e-01 & 0.92\\
&  1346 & 0.1768 & 9.9e-01 & 0.96 & 1.0e-02 & 0.98 & 3.3e-02 & 1.30 & 1.3e-01 & 0.98 & 6.5e-02 & 0.98\\
&  5250 & 0.0884 & 5.0e-01 & 0.99 & 5.0e-03 & 1.00 & 1.3e-02 & 1.37 & 6.6e-02 & 0.99 & 3.3e-02 & 1.00\\
& 20738 & 0.0442 & 2.5e-01 & 1.00 & 2.5e-03 & 1.00 & 4.7e-03 & 1.43 & 3.3e-02 & 1.00 & 1.6e-02 & 1.00\\
& 82434 & 0.0221 & 1.2e-01 & 1.00 & 1.3e-03 & 1.00 & 1.7e-03 & 1.47 & 1.6e-02 & 1.00 & 8.2e-03 & 1.00\\
 \hline 
 \multirow{6}{*}{1} &          290 & 0.7071 & 1.3e+0 & $\star$ & 1.3e-02 & $\star$ & 4.1e-02 & $\star$ & 1.5e-01 & $\star$ & 7.4e-02 & $\star$\\
&  1090 & 0.3536 & 4.1e-01 & 1.61 & 4.2e-03 & 1.68 & 1.4e-02 & 1.50 & 3.9e-02 & 1.92 & 2.0e-02 & 1.93\\
&  4226 & 0.1768 & 1.1e-01 & 1.94 & 1.1e-03 & 1.95 & 5.8e-03 & 1.30 & 9.9e-03 & 1.98 & 5.0e-03 & 1.98\\
& 16642 & 0.0884 & 2.7e-02 & 1.98 & 2.7e-04 & 1.99 & 2.0e-03 & 1.55 & 2.5e-03 & 1.99 & 1.2e-03 & 2.00\\
& 66050 & 0.0442 & 6.8e-03 & 2.00 & 6.8e-05 & 2.00 & 5.8e-04 & 1.78 & 6.2e-04 & 2.00 & 3.1e-04 & 2.00 \\
& 263170 & 0.0221 & 1.7e-03 & 2.00 & 1.7e-05 & 2.00 & 1.5e-04 & 1.92 & 1.6e-04 & 2.00 & 7.8e-05 & 2.00  \\
\hline
    \end{tabular}
    \caption{Example 1. Error history (degrees of freedom, mesh size, individual errors and experimental rates of convergence) in 2D for the formulation using the two lowest-order FE families with PEERS$_k$ elements.}
    \label{tab:my_label}
\end{table}

\begin{figure}[t!]
    \centering
    \includegraphics[width=0.325\textwidth]{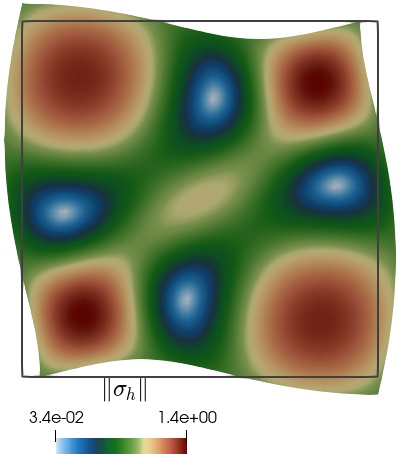}
    \includegraphics[width=0.325\textwidth]{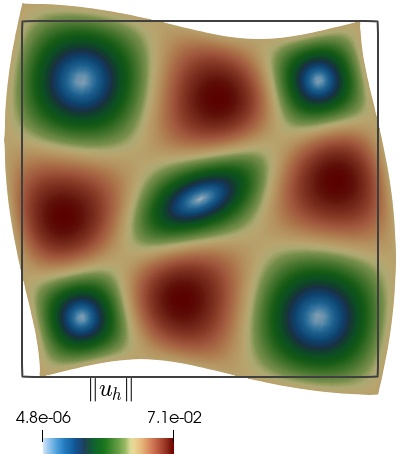}
    \includegraphics[width=0.325\textwidth]{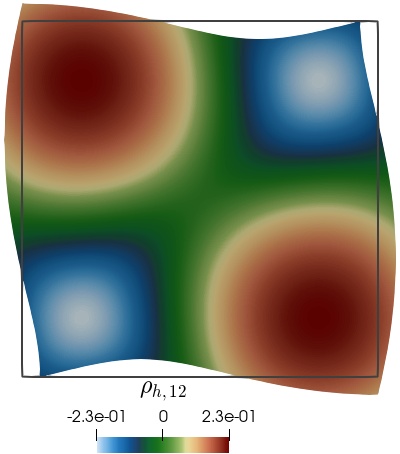}\\
    \includegraphics[width=0.325\textwidth]{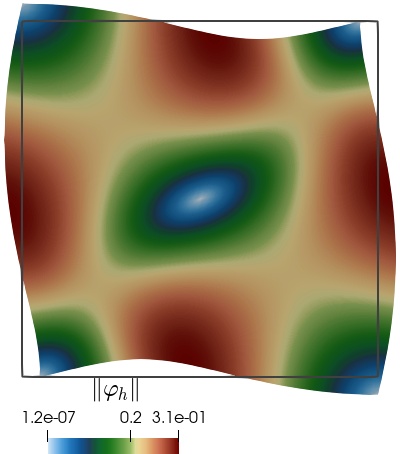}
    \includegraphics[width=0.325\textwidth]{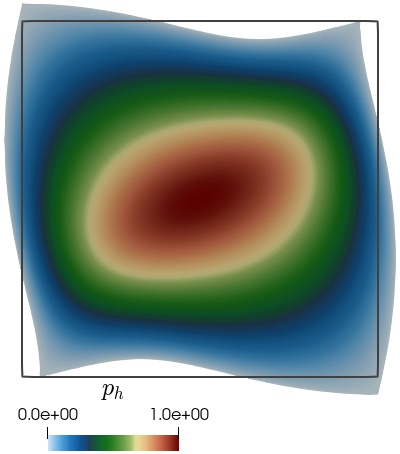}
    \caption{Example 1. Sample of approximate solutions (stress magnitude, displacement magnitude, non-zero entry of rotation, flux magnitude, and fluid pressure) computed with the second-order scheme and plotted on the deformed domain (for reference we also show the contour of the undeformed domain).}
    \label{fig:enter-label}
\end{figure}

We also exemplify the \revX{equilibrium} and mass conservativity of the formulation. To do so we % let $\mathcal P^k_h : \mathrm L^2(\Omega)\to \mathrm P_k(\mathcal{T}_h)$ be the projector defined, for each $v\in \mathrm L^2(\Omega)$, as the unique element $\mathcal P^k_h(v) \in 
%\mathrm P_k(\mathcal{T}_h)$ such that
%\begin{equation*}%\label{identity-mathcal-p-k}
%\int_\Omega \mathcal P^k_h(v) \,q_h \,=\, \int_\Omega v\,q_h \qquad\forall\,q_h\in P_k(\mathcal{T}_h)\,,
%\end{equation*}
%and let $\boldsymbol{\mathcal P}^k_h : \mathbf L^2(\Omega)\to \mathbf P_k(\mathcal{T}_h)$ be its corresponding vector version \cred{[check if already defined before]}. Then, the following quantities 
represent the loss of \revX{equilibrium} and mass as 
\[
  \revX{\texttt{equ}_h} := \bigl\|\boldsymbol{\mathcal P}_h[\bdiv (\bsigma_h) + \boldsymbol{f}]\bigr\|_{\ell^\infty}, \ 
  \texttt{mass}_h := \biggl\|\mathcal{P}_h\biggl[\bigl(c_0+ \frac{d\alpha^2}{d\lambda + 2\mu}\bigr) p_h + \frac{\alpha}{d\lambda + 2\mu}\tr\bsigma_h + \div(\bvarphi_h) + g\biggr]\biggr\|_{\ell^\infty},
\]
where $\mathcal{P}_h: \mathrm L^2(\Omega)\to \mathrm P_k(\mathcal{T}_h)$ is the scalar version of $\boldsymbol{\mathcal{P}}_h$. 
They are computed at each refinement level and tabulated in Table~\ref{table:conserv} together with the total error $\texttt{e} := e(\bsigma) + e(\bu) + e(\brho) + e(\bvarphi) + e(p)$, and its experimental convergence rate. We report on the nonlinear iteration count as well. The expected optimal convergence of the total error, and the announced local conservativity are confirmed. We also note that, at least for this parameter regime, for all the refinements and polynomial degrees  the nonlinear solver takes  three iterations to get a residual below the tolerance. In the last column of the same table we report on the efficiency of the global a posteriori error estimator designed in Section~\ref{sec:aposteriori} $\texttt{eff}(\Xi) = \frac{\texttt{e}}{\Xi}$, which -- in this case of smooth solutions -- is asymptotically constant (approximately 0.98 for $k=0$ and 1.52 for $k=1$).

\begin{table}[t!]
\centering
\begin{tabular}{|rc|cccccc|}
\hline       
DoFs & $h$ & \texttt{e} & rate & \texttt{equ}$_h$ & \texttt{mass}$_h$ & iter & \texttt{eff}($\Xi$) \\
\hline  
\multicolumn{8}{|c|}{$k=0$}\\
\hline
    97 &  0.707 & 4.83e+0 & $\star$ & 9.30e-16 & 1.94e-16 &  3 & 0.93 \\
   353 &  0.354 & 2.45e+0 & 0.98 & 9.44e-16 & 2.98e-16 &  3 & 0.94 \\
  1345 &  0.177 & 1.25e+0 & 0.98 & 2.64e-15 & 8.26e-16 &  3 & 0.96 \\
  5249 &  0.088 & 6.22e-01 & 1.00 & 5.58e-15 & 2.49e-15 &  3 & 0.97 \\
 20737 &  0.044 & 3.10e-01 & 1.00 & 1.46e-14 & 7.03e-15 &  3 & 0.98 \\
 82433 &  0.022 & 1.55e-01 & 1.00 & 1.72e-13 & 1.09e-13 &  3& 0.98 \\
 328705 & 0.011 & 7.63e-02 & 1.00 & 4.25e-13 & 3.46e-13 & 3 & 0.98 \\
\hline 
\multicolumn{8}{|c|}{$k=1$}\\
\hline
   289 &  0.707 & 1.57e+0 & $\star$ & 2.88e-15 & 1.20e-15 &  3 & 1.45\\
  1089 &  0.354 & 5.08e-01 & 1.63 & 6.77e-15 & 2.25e-15 &  3 & 1.50\\
  4225 &  0.177 & 1.34e-01 & 1.92 & 2.21e-14 & 6.60e-15 &  3& 1.53\\
 16641 &  0.088 & 3.42e-02 & 1.97 & 8.00e-14 & 1.31e-14 &  3& 1.52\\
 66049 &  0.044 & 8.62e-03 & 1.99 & 3.91e-13 & 3.83e-14 &  3& 1.51\\
263169 &  0.022 & 2.16e-03 & 2.00 & 8.20e-12 & 7.61e-14 &  3& 1.51\\
 1050625 &  0.011 & 5.04e-04 & 2.00 & 2.17e-11 & 4.90e-13 &  3& 1.52\\
 \hline 
\end{tabular}
\caption{Example 1. Total error, experimental rates of convergence, $\ell^\infty$-norm of the projected residual of the \revX{equilibrium} and mass balance equations, Newton--Raphson iteration count, and efficiency index of the global a posteriori error estimator. Tabulated results correspond to the two lowest-order polynomial degrees.}
\label{table:conserv}
\end{table}

\subsection{Convergence  in 3D using physically relevant  parameters} Next we investigate the behaviour of the proposed numerical methods in a 3D setting and taking model parameters more closely related to applications in tissue poroelasticity. We still use manufactured solutions to assess the accuracy of the formulation, but take an exact  displacement that satisfies $\div \bu \to 0$ as $\lambda \to \infty$. The domain is the 3D box $\Omega=(0,L)\times(0,L)\times (0,2L)$ with $L=0.01$\,m,  and mixed boundary conditions were taken analogously as before, separating the domain boundary between $\Gamma_\rN$ defined \revX{by} the planes $x=0$, $y=0$ and $z=0$,  and $\Gamma_\rD$ as the remainder of the boundary. The manufactured displacement and pressure head are 
\[ \bu = \frac{L}{4}\begin{pmatrix}
\sin(x/L)\cos(y/L)\sin(z/(2L)) + x^2/\lambda\\
-2\cos(x/L)\sin(y/L)\cos(z/(2L)) + y^2/\lambda\\
2\cos(x/L)\cos(y/L)\sin(z/(2L)) - 2z^2/\lambda    
\end{pmatrix}, \quad p = \sin(x/L)\cos(y/L)\sin(z/(2L)).\]

\begin{table}[t!]
\setlength{\tabcolsep}{4pt}
    \centering
    \begin{tabular}{|rc|cccccccccc|}
      \hline
DoFs & $h$ & $e(\bsigma)$ & rate & $e(\bu)$  & rate & $e(\brho)$  & rate & $e(\bvarphi)$  & rate & $e(p)$  & rate \\
    \hline 
    \multicolumn{12}{|c|}{unity parameters} \\
    \hline
%   328 & 0.0173 & 1.6e-04 & $\star$ & 5.0e-06 & $\star$ & 2.7e-06 & $\star$ & 2.6e-06 & $\star$ & 1.1e-08 & $\star$\\
%  2311 & 0.0087 & 8.2e-05 & 1.00 & 2.5e-06 & 1.00 & 1.0e-06 & 1.37 & 1.3e-06 & 0.97 & 5.7e-09 & 0.96\\
% 17443 & 0.0043 & 4.1e-05 & 1.00 & 1.2e-06 & 1.00 & 4.0e-07 & 1.35 & 6.7e-07 & 0.99 & 2.9e-09 & 0.99\\
%135715 & 0.0022 & 2.1e-05 & 1.00 & 6.2e-07 & 1.00 & 1.7e-07 & 1.22 & 3.4e-07 & 0.99 & 1.4e-09 & 1.00\\
%1071043 & 0.0011 & 1.0e-05 & 1.00 & 3.1e-07 & 1.00 & 8.0e-08 & 1.11 & 1.7e-07 & 1.00 & 7.2e-10 & 1.00\\
   328 & 0.0173 & 2.9e-02 & $\star$  & 1.1e-06 & $\star$  & 4.8e-04 & $\star$  & 3.1e+0 & $\star$  & 1.3e-04 & $\star$ \\
  2311 & 0.0087 & 1.5e-02 & 0.94 & 5.7e-07 & 0.96 & 1.3e-04 & 1.90 & 1.6e+0 & 1.00 & 6.3e-05 & 1.00\\
 17443 & 0.0043 & 7.5e-03 & 0.99 & 2.8e-07 & 1.00 & 4.4e-05 & 1.54 & 7.8e-01 & 1.00 & 3.2e-05 & 1.00\\
135715 & 0.0022 & 3.8e-03 & 1.00 & 1.4e-07 & 1.00 & 1.8e-05 & 1.34 & 3.9e-01 & 1.00 & 1.6e-05 & 1.00\\
1071043 & 0.0011 & 1.9e-03 & 1.00 & 7.0e-08 & 1.00 & 7.7e-06 & 1.19 & 1.9e-01 & 1.00 & 7.9e-06 & 1.00\\
      \hline 
    \multicolumn{12}{|c|}{physically relevant parameters} \\
    \hline
       328 & 0.0173 & 2.1e+02 & $\star$ & 1.1e-06 & $\star$ & 4.2e-04 & $\star$ & 4.8e-07 & $\star$ & 1.9e-04 & $\star$ \\
  2311 & 0.0087 & 1.1e+02 & 0.91 & 5.6e-07 & 0.99 & 1.1e-04 & 1.93 & 3.9e-07 & 0.55 & 9.1e-05 & 1.05\\
 17443 & 0.0043 & 5.6e+01 & 0.98 & 2.8e-07 & 1.00 & 3.3e-05 & 1.74 & 2.2e-07 & 0.89 & 4.4e-05 & 1.04\\
135715 & 0.0022 & 2.8e+01 & 1.00 & 1.4e-07 & 1.00 & 1.2e-05 & 1.42 & 1.2e-07 & 0.93 & 2.2e-05 & 0.99\\
1071043 & 0.0011 & 1.4e+01 & 1.00 & 7.0e-08 & 1.00 & 5.3e-06 & 1.23 & 6.0e-08 & 0.97 & 1.1e-05 & 0.98\\
%  328 & 0.0173 & 1.5e+0 & $\star$  & 5.5e-06 & $\star$  & 2.3e-06 & $\star$  & 3.4e-09 & $\star$  & 1.1e-06 & $\star$ \\
%  2311 & 0.0087 & 7.5e-01 & 1.00 & 2.8e-06 & 1.00 & 8.5e-07 & 1.41 & 2.9e-09 & 0.19 & 5.1e-07 & 1.09\\
% 17443 & 0.0043 & 3.8e-01 & 1.00 & 1.4e-06 & 1.00 & 3.1e-07 & 1.47 & 1.8e-09 & 0.73 & 2.5e-07 & 1.07\\
%135715 & 0.0022 & 1.9e-01 & 1.00 & 6.9e-07 & 1.00 & 1.3e-07 & 1.27 & 9.5e-10 & 0.90 & 1.2e-07 & 1.00\\
%1071043 & 0.0011 & 9.4e-02 & 1.00 & 3.5e-07 & 1.00 & 5.8e-08 & 1.14 & 4.9e-10 & 0.95 & 6.1e-08 & 1.00\\
\hline
   \end{tabular}
    \caption{Example 2. Error history (degrees of freedom, mesh size, individual errors and experimental rates of convergence) in 3D for the formulation using the  lowest-order FE family with PEERS$_k$ elements and changing from unity (top) to physically relevant (bottom) parameters.}
    \label{tab:my_label2}
\end{table}

%Note that the performance of the nonlinear solver (in terms of the required number of iterations in the Newton--Raphson algorithm to get the residuals below the prescribed tolerance)  changes with respect to the previous tests. 
First we set again the model parameters to mild values $\lambda = \mu = c_0=k_0=\alpha=\mu_f =1$, $k_1=k_2=0.1$, \revX{we use the exponential permeability constitutive law}, and we compare them against the following values (from, e.g., \cite{barnafi22,ruiz22}) 
\begin{gather*} k_0 = 2.28\times 10^{-11}\,\text{m}^3,\quad k_1 = 5\times10^{-12}\,\text{m}^3, \quad \lambda =1.44\times10^{6}\,\text{Pa},\\  \mu = 9.18\times10^{3}\,\text{Pa}, \quad \mu_f = 7.5\times10^{-4}\,\text{Pa}\cdot\text{s},\quad c_0 = 0, \quad \alpha = 0.99.\end{gather*}Table~\ref{tab:my_label2} reports on the convergence of the method. While the magnitude of the stress  errors is much higher for the second parameter regime, the discharge flux error magnitude is smaller than in the first case and the displacement, rotation, and fluid pressure errors remain roughly of the same magnitude. In any case, the table confirms that the optimal slope of the error decay is not affected by a vanishing storativity nor large Lam\'e constants. The \revX{components of the numerical solution} are displayed in Figure~\ref{fig:enter-label2}.

\begin{figure}[t!]
    \centering
    \includegraphics[width=0.325\textwidth]{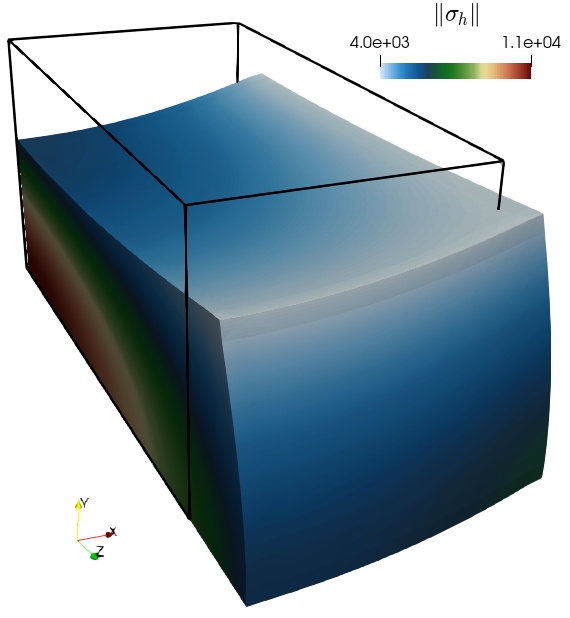}
    \includegraphics[width=0.325\textwidth]{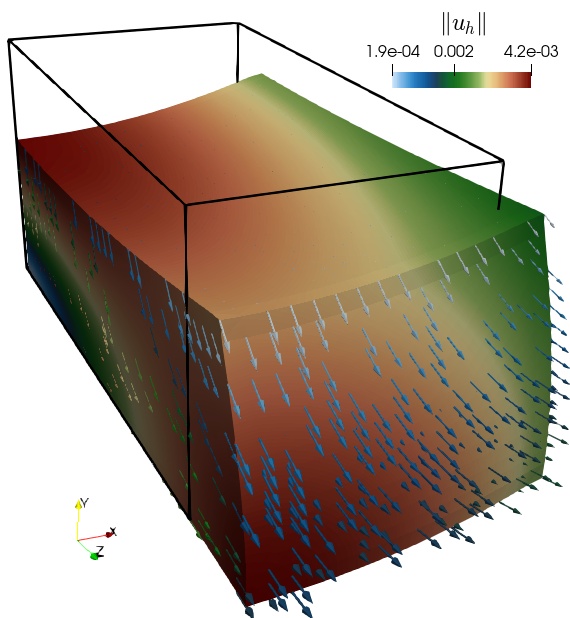}
    \includegraphics[width=0.325\textwidth]{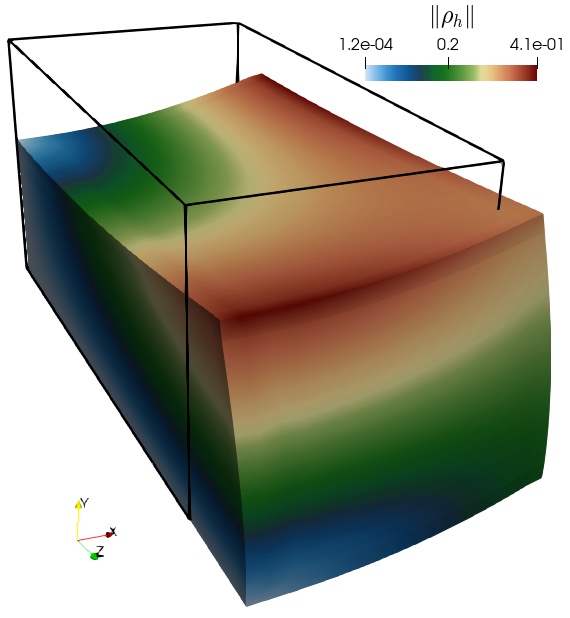}\\
    \includegraphics[width=0.325\textwidth]{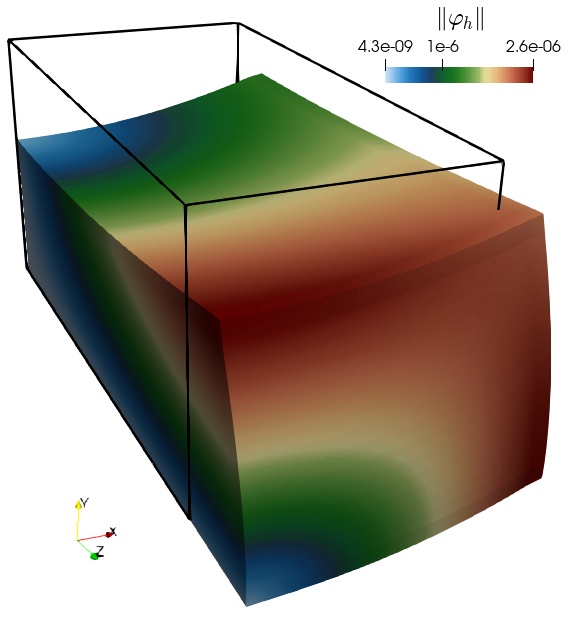}
    \includegraphics[width=0.325\textwidth]{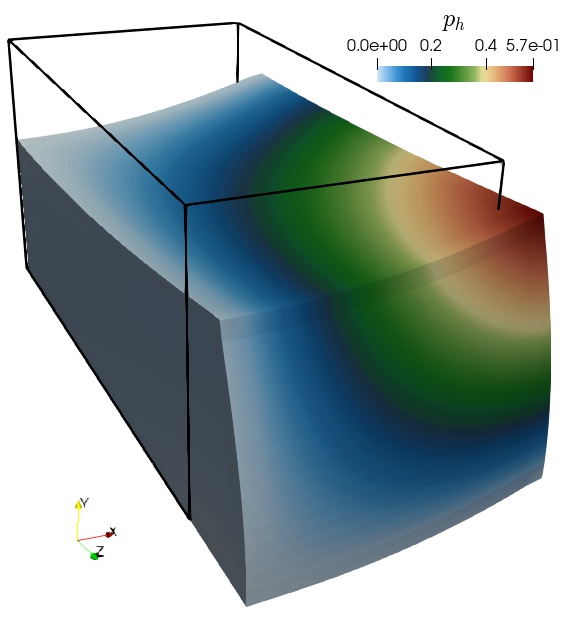}
    \caption{Example 2. Sample of approximate solutions (stress magnitude, displacement magnitude,   rotation magnitude, flux magnitude, and fluid pressure) computed with the first-order scheme and plotted on the deformed domain (for reference we also show the outline of the undeformed domain).}
    \label{fig:enter-label2}
\end{figure}

\subsection{Convergence in the case of adaptive mesh refinement}
We continue with a test targeting the recovery of optimal convergence through adaptive mesh refinement guided by the a posteriori error estimator proposed in Section~\ref{sec:aposteriori}.  
We employ the well-known adaptive mesh refinement approach  of solving, then computing the estimator, marking, and refining. Marking is done as follows \cite{dorfler_sinum96}: a given $K\in \mathcal{T}_h$ is added to the marking set $\mathcal{M}_h\subset\mathcal{T}_h$  whenever the local error indicator $\Xi_K$ satisfies 
\[ \sum_{K \in \mathcal{M}_h} \Xi^2_K \geq \zeta \sum_{K\in\mathcal{T}_h} \Xi_K^2,\]
where $\zeta$ is a user-defined bulk density. All  elements in $\mathcal{M}_h$ are marked for refinement and also some neighbours are marked for the sake of closure. For convergence rates we use the alternative form 
\[\text{rate} = -2\log(e/\widehat{e})[\log(\text{DoFs}/\widehat{\text{ DoF}})]^{-1}.\]
 
  Let us consider the non-convex rotated L-shaped domain $\Omega = (-1,1)^2\setminus (-1,0)^2$ and use manufactured displacement and fluid pressure with sharp gradients near the domain re-entrant corner (see, e.g., \cite{carstensen2000locking} for the displacement and \cite{bulle2023hierarchical} for the fluid pressure)
\begin{gather*}%\label{eq:non-smooth}
\bu(r,\theta)  =\frac{r^{\chi}}{2\mu} \begin{pmatrix} 
-(\chi + 1) \cos([\chi+1]\theta) + (M_2 - \chi - 1)M_1\cos([\chi-1]\theta)\\
(\chi + 1) \sin([\chi+1]\theta) + (M_2 + \chi - 1)M_1\sin([\chi-1]\theta)
\end{pmatrix}, \\
p(r,\theta)  =  r^{1/3}\sin\biggl(\frac{1}{3}(\frac{\pi}{2} + \theta)\biggr),
\end{gather*}
with polar coordinates $r = \sqrt{x_1^2 + x_2^2}$, $\theta= \arctan(x_2,x_1)$, $\chi \approx 0.54448373$, $M_1 = -\cos([\chi+1]\omega)/\cos([\chi-1]\omega)$, and $M_2 = 2(\lambda + 2\mu)/(\mu + \lambda)$. The boundary conditions (taking as $\Gamma_\rN$ the segments at $x=\pm1$ and $y =\pm 1$ and $\Gamma_\rD$ the remainder of the boundary) and forcing data are constructed from these solutions and the model parameters are  $\lambda = 10^3$, $\mu = 10$, $k_0 = \frac12$, $\mu_f =c_0=k_1= 0.1$, $\alpha = \frac14$, where for this test we consider a Kozeny--Carman permeability form. As in \cite{carstensen2000locking}, a sub-optimal rate of convergence is expected for the mixed elasticity sub-problem in its energy norm. Note that since the exact pressure is in $\rH^{4/3-\epsilon}(\Omega)$  for any
$\epsilon > 0$ (cf. \cite[Chapter 5]{grisvard2011elliptic}), it is still regular enough to have optimal convergence. However its  gradient (and therefore also the exact discharge flux $\bvarphi$) has a singularity located at the reentrant corner and therefore we expect an order of convergence of approximately $O(h^{1/3})$.

\begin{table}[t!]
\setlength{\tabcolsep}{3.5pt}
    \centering
    \begin{tabular}{|rc|ccccccccccc|}
      \hline
DoFs & $h$ & $e(\bsigma)$ & rate & $e(\bu)$  & rate & $e(\brho)$  & rate & $e(\bvarphi)$  & rate & $e(p)$  & rate & $\texttt{eff}(\Xi)$ \\
    \hline 
    \multicolumn{13}{|c|}{uniform mesh refinement} \\
    \hline
       77 & 1.4142 & 1.7e+3 & $\star$ & 1.7e+1 & $\star$ & 7.6e+1 & $\star$ & 9.1e+0 & $\star$ & 1.6e+0 & $\star$ & 2.67 \\
   273 & 0.7071 & 1.5e+3 & 0.27 & 7.8e+0 & 1.10 & 2.5e+1 & 1.61 & 5.8e+0 & 0.65 & 7.8e-01 & 1.06 & 2.65 \\
  1025 & 0.3536 & 1.2e+3 & 0.31 & 5.0e+0 & 0.63 & 2.1e+1 & 0.27 & 4.3e+0 & 0.44 & 3.9e-01 & 1.00 & 3.12 \\
  3969 & 0.1768 & 8.9e+2 & 0.40 & 3.6e+0 & 0.50 & 1.7e+1 & 0.26 & 3.3e+0 & 0.38 & 2.0e-01 & 0.99 & 2.82 \\
 15617 & 0.0884 & 6.6e+2 & 0.43 & 2.6e+0 & 0.45 & 1.4e+1 & 0.32 & 2.6e+0 & 0.35 & 9.9e-02 & 0.99 & 2.60 \\
 61953 & 0.0442 & 4.8e+2 & 0.45 & 1.9e+0 & 0.45 & 1.1e+1 & 0.35 & 2.0e+0 & 0.34 & 5.0e-02 & 1.00 & 2.42 \\
246785 & 0.0221 & 3.5e+2 & 0.46 & 1.4e+0 & 0.46 & 8.4e+0 & 0.38 & 1.6e+0 & 0.34 & 2.5e-02 & 1.00 & 2.27 \\
985089 & 0.0011 & 2.9e+2 & 0.45 & 1.1e+0 & 0.45 & 6.7e+0 & 0.38 & 1.2e+0 & 0.34 & 1.3e-02 & 1.00 & 2.40 \\
      \hline 
    \multicolumn{13}{|c|}{adaptive mesh refinement} \\
    \hline
   77 & 1.4142 & 1.7e+3 & $\star$ & 1.7e+1 & $\star$ & 7.6e+1 & $\star$ & 9.1e+0 & $\star$ & 1.6e+0 & $\star$ & 2.67 \\
   273 & 0.7071 & 1.5e+3 & 0.29 & 7.8e+0 & 1.20 & 2.5e+1 & 1.76 & 5.8e+0 & 0.72 & 7.8e-01 & 1.16 & 2.65\\
  1025 & 0.3536 & 1.2e+3 & 0.33 & 5.0e+0 & 0.66 & 2.1e+1 & 0.28 & 4.3e+0 & 0.47 & 3.9e-01 & 1.04 & 3.12\\
  3813 & 0.3536 & 8.9e+2 & 0.42 & 3.6e+0 & 0.52 & 1.7e+1 & 0.27 & 3.3e+0 & 0.40 & 2.0e-01 & 1.04 & 2.82\\
  7113 & 0.2500 & 6.7e+2 & 0.89 & 2.6e+0 & 0.98 & 1.4e+1 & 0.68 & 2.6e+0 & 0.77 & 1.1e-01 & 1.91 & 2.55\\
 11013 & 0.2500 & 5.2e+2 & 1.20 & 2.0e+0 & 1.31 & 1.1e+1 & 0.99 & 2.0e+0 & 1.07 & 6.0e-02 & 1.74 & 2.58\\
 17449 & 0.2500 & 4.0e+2 & 1.10 & 1.5e+0 & 1.18 & 9.1e+0 & 0.95 & 1.6e+0 & 1.01 & 3.7e-02 & 1.79 & 2.52\\
 27225 & 0.1768 & 3.1e+2 & 1.14 & 1.2e+0 & 1.18 & 7.2e+0 & 1.05 & 1.3e+0 & 1.05 & 2.1e-02 & 1.80 & 2.52\\
 38081 & 0.1768 & 2.5e+2 & 1.37 & 9.0e-01 & 1.53 & 5.7e+0 & 1.38 & 1.0e+0 & 1.37 & 1.5e-02 & 1.57 & 2.55\\
 54797 & 0.1250 & 2.0e+2 & 1.18 & 7.2e-01 & 1.21 & 4.6e+0 & 1.22 & 8.1e-01 & 1.26 & 1.2e-02 & 1.48 & 2.53\\
\hline
   \end{tabular}
    \caption{Example 3. Convergence history (degrees of freedom, mesh size, individual errors,  experimental rates of convergence, and effectivity index) in a rotated L-shaped domain for the formulation using the  lowest-order FE family with PEERS$_k$ elements and changing from uniform (top) to adaptive mesh refinement (bottom) guided by the a posteriori error indicator from \eqref{eq:estimator}.}
    \label{tab:adaptive}
\end{table}

The numerical results of this test are reported in Table~\ref{tab:adaptive}. We observe the expected sub-optimal convergence under an uniform mesh refinement while the optimal convergence in all variables is attained as the mesh is locally refined (the first three rows are very similar as most of the elements are refined in the first three steps. This can be controlled by the bulk density, here taken as  $\zeta = 9.5\cdot10^{-5}$). We also note that the individual errors are approximately of the same magnitude in the last row of each section of the table, but for the adaptive case this is achieved using approximately 5.5\% of the number of degrees of freedom needed in the uniformly refined case. The last column of the table again confirms the reliability and efficiency of the a posteriori error estimator. Note that for this case we compute the divergence part of the error norm in the stress and fluxes as projections of the \revX{equilibrium} and mass residuals onto the displacement and pressure discrete spaces, respectively. We plot in Figure~\ref{fig:adaptive} the approximate displacement and pressure as well as sample triangulations obtained after a few adaptive refinement steps that confirm the expected agglomeration of vertices near the reentrant corner.

\begin{figure}[t!]
    \centering
    \includegraphics[width=0.325\textwidth]{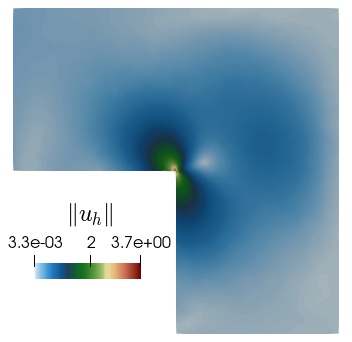}
    \includegraphics[width=0.325\textwidth]{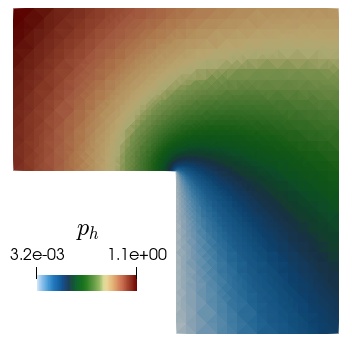}\\
    \includegraphics[width=0.325\textwidth]{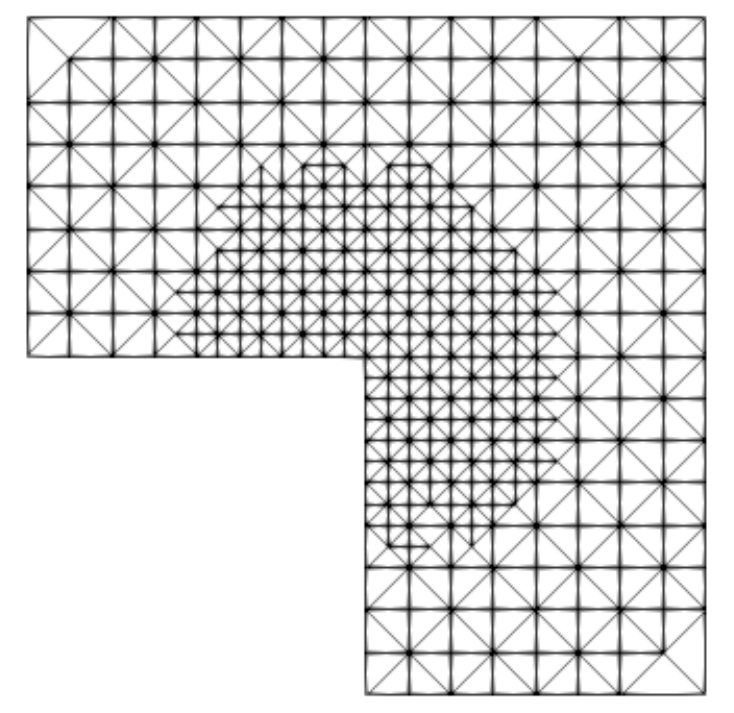}
    \includegraphics[width=0.325\textwidth]{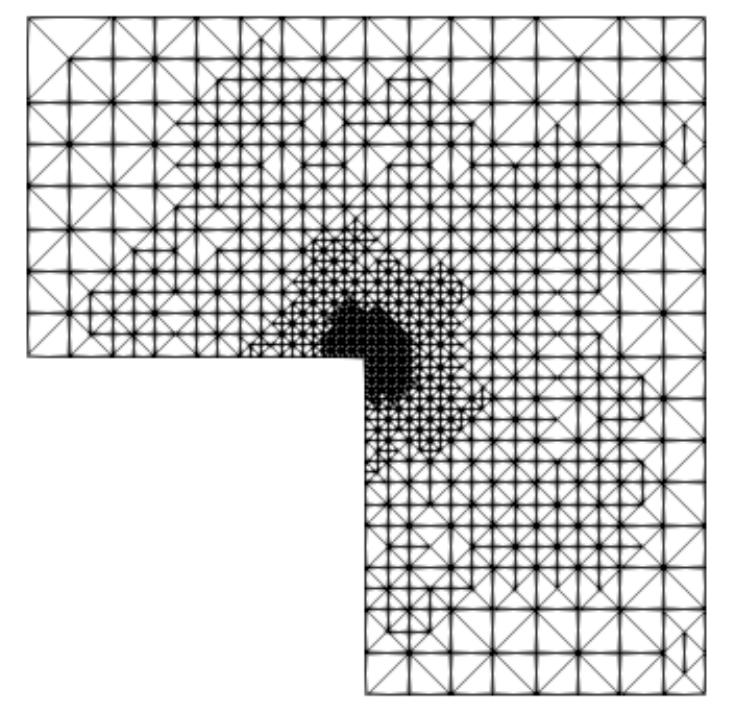}
    \includegraphics[width=0.325\textwidth]{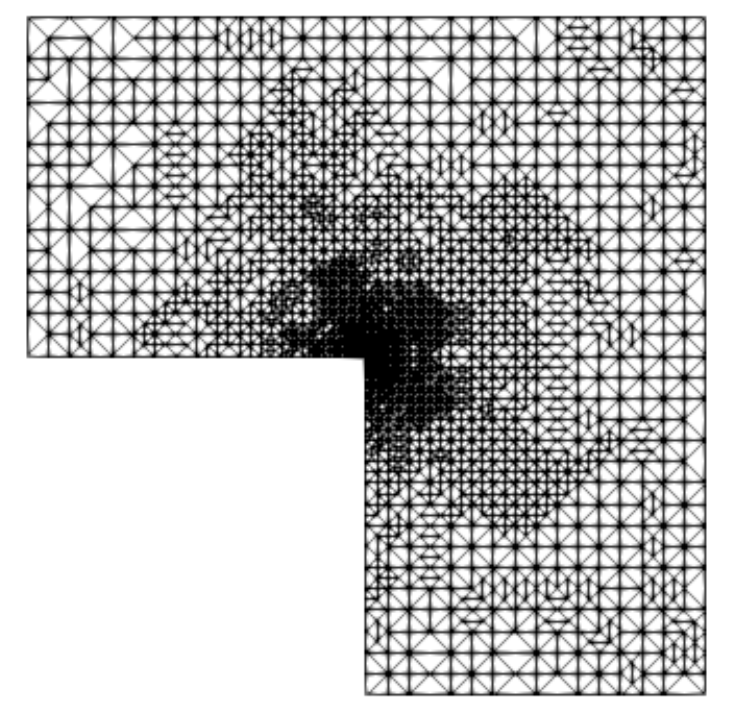}
    \caption{Example 3. Approximate primal variable solutions (solid displacement  and fluid pressure) computed with the first-order scheme,  and  meshes generated after two, three, and four adaptive refinement steps.}
    \label{fig:adaptive}
\end{figure}

%%%%%%%%%%%%%
\subsection{Adaptive computation of cross-sectional flow and deformation  in a soft tissue specimen} Finally, we apply the proposed methods to  simulate the localisation of stress, deformation, and flow patterns in a multi-layer cross-section of cervical spinal cord. We follow the setup in \cite{meddahi23,stover16}.  The geometry and unstructured mesh have been generated using GMSH \cite{geuzaine09} from the  images in \cite{stover16}. The heterogeneous porous material consists of white and grey matter  surrounded by the pia mater (a thin layer, also considered poroelastic. See Figure~\ref{fig:ex4}, top two left panels). All components are assumed fully saturated with \revX{cerebrospinal} fluid. The transversal cross-section  has 1.3\,cm in maximal diameter and  the indentation region is  a curved subset of the anterior part of the pia mater (a sub-boundary of length {0.4\,cm}). Boundary conditions are of mixed load-traction type, but slightly different than the ones analysed in the previous section. We conduct an indentation test applying a traction $\bsigma\bn = (0,-P)^{-\tt t}$, with $P$ a constant solid pressure of {950\,dyne/cm$^2$}. The posterior part of the pia mater  acts as a rigid posterior support where we prescribe zero displacement. The remainder of the boundary  of the pia mater  is   stress-free. For the fluid phase we impose a constant inflow pressure of cerebrospinal fluid of 1.1\,dyne/cm$^2$ and zero outflow pressure at the stress-free sub-boundary, as well as zero normal discharge flux at the posterior support.  For the three different layers of the domain we use the following values for   Young modulus, Poisson ratio, and lower bound for permeability (some values from \cite{bloom98,meddahi23,stover16}) 
%\begin{gather*}
$E^{\mathrm{pia}} = 23'000\,\text{dyne/cm}^{2}, \ \nu^{\mathrm{pia}} = 0.3$, $E^{\mathrm{white}} = 8'400\,\text{dyne/cm}^{2}, \ \nu^{\mathrm{white}} = 0.479,\  E^{\mathrm{grey}} = 16'000\,\text{dyne/cm}^{2}$,  $\nu^{\mathrm{grey}} = 0.49$,  $k_0^{\mathrm{grey}} = 1.4\cdot10^{-9}\,\text{dyne/cm}^{2}$,  $k_0^{\mathrm{white}} = 1.4\cdot10^{-6}\,\text{dyne/cm}^{2}$, $k_0^{\mathrm{pia}} = 3.9\cdot10^{-10}\,\text{dyne/cm}^{2}$. Further, we take $\fb=\boldsymbol{0}$, $g=0$, $\mu_f = 70\,\text{dyne/cm}^{2}\cdot$\,s (for cerebrospinal fluid at 37$^\circ$), $k_1 = \frac12k_0$, $\alpha = \frac14$, and $c_0 = 10^{-3}$. 
%{The density in the three layers is taken as $\rho^{\mathrm{grey}} =1.045$\,g/cm$^3$,  $\rho^{\mathrm{white}} =1.041$\,g/cm$^3$,  $\rho^{\mathrm{pia}} =1.133$\,g/cm$^3$}.

The initial and the final adapted mesh, together with samples of  solutions are shown in Figure~\ref{fig:ex4}, where we have used the mesh density parameter $\zeta = 5.5\cdot10^{-4}$. After each adaptation iteration guided by the a posteriori error indicator \eqref{eq:estimator}, a mesh smoothing step was included. The figure indicates that most of the refinement occurs near the interface between the heterogeneous components of the porous media, and the plots also confirm a flow pattern moving slowly from top to bottom, consistently with a typical indentation test.

\begin{figure}[t!]
    \centering
    \includegraphics[width=0.34\linewidth]{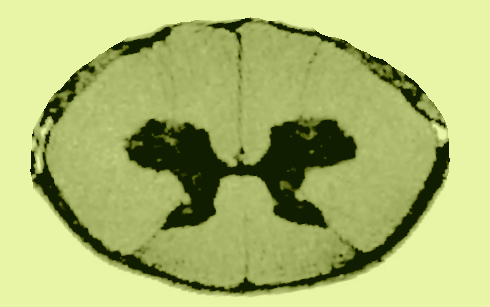} \includegraphics[width=0.32\linewidth]{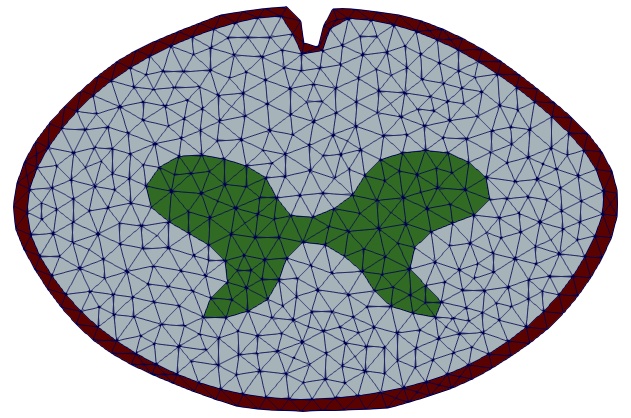}
    \includegraphics[width=0.32\linewidth]{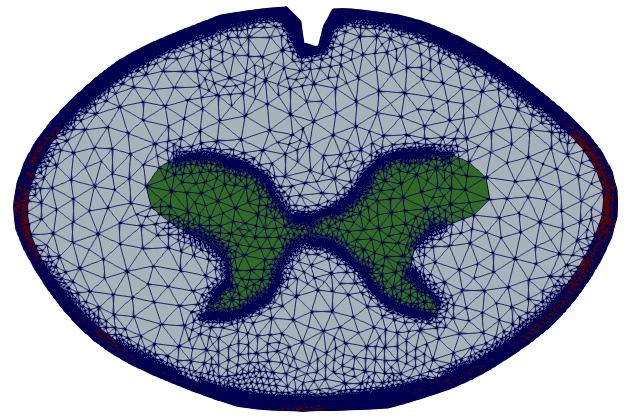}\\
 \includegraphics[width=0.24\textwidth]{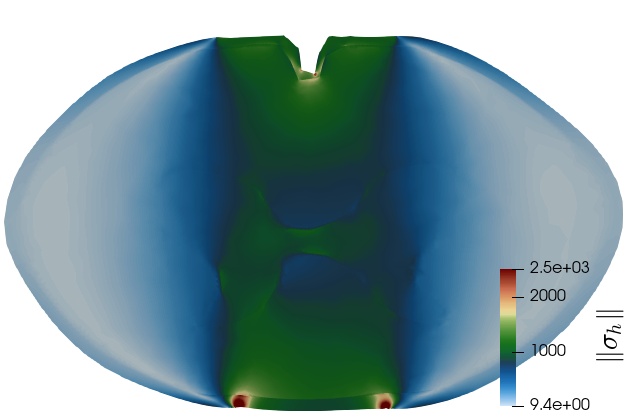}
    \includegraphics[width=0.24\textwidth]{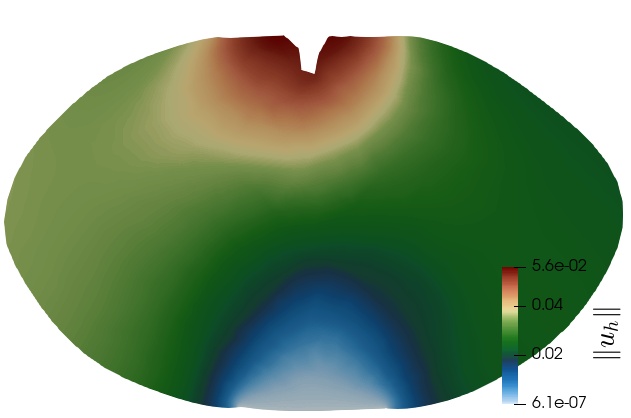}
    \includegraphics[width=0.24\textwidth]{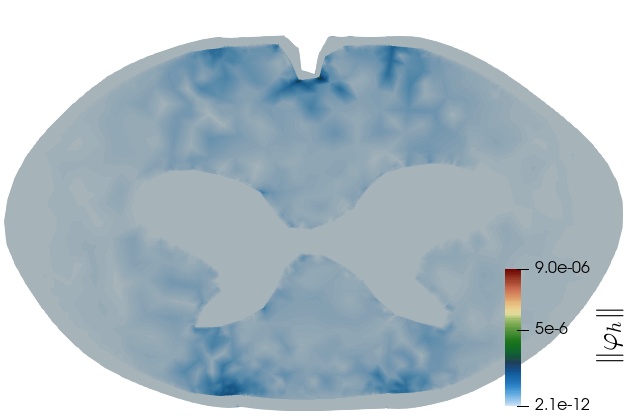}
    \includegraphics[width=0.24\textwidth]{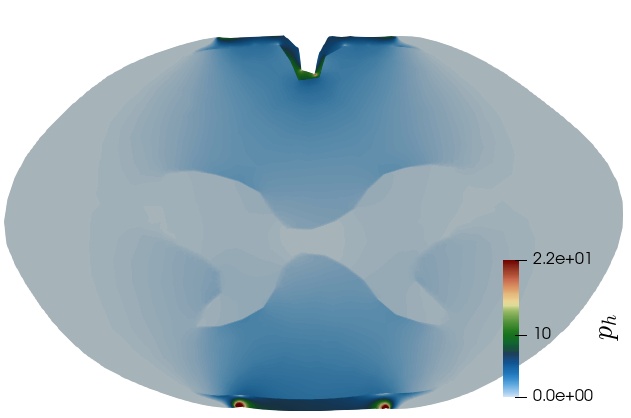}    
    \caption{Example 4. Cross sectional area of cervical spinal cord segmented from sheep imaging data in \cite{stover16}, initial coarse mesh indicating subdomains with distinct material properties (outer pia mater in red, mid white matter in grey, inner grey matter in green), final adapted mesh after six refinement steps, and sample of stress, displacement, %rotation, 
    fluid flux and fluid pressure at the indentation test (bottom row figures are obtained with the lowest-order scheme and rendered on the deformed configuration).}
    \label{fig:ex4}
\end{figure}

\revX{\section{Concluding remarks}\label{sec:concl}
In this work, we developed a family of mixed finite element methods for a nonlinear poroelasticity model with stress-dependent permeability. By reformulating the constitutive equations, we enabled the use of a Hellinger--Reissner-type mixed formulation, ensuring robustness with respect to near-incompressibility and vanishing storativity limits. The problem structure was analysed using fixed-point theory and saddle-point formulations, leading to well-posedness results for both the continuous and discrete settings. The chosen finite element spaces (PEERS$_k$ elements for elasticity and Raviart--Thomas elements for fluid flow) provided exact equilibrium and mass conservation.  

We derived a priori error estimates, demonstrating optimal convergence properties. Furthermore, a residual-based a posteriori error estimator was introduced and shown to be both reliable and efficient. This estimator guided adaptive mesh refinement strategies, which were validated through numerical experiments in both two and three dimensions. 
}

\revZ{The proposed nonlinear dependence of permeability on stress in poroelasticity is particularly relevant for soft tissues, where permeability changes due to deformation, and for subsurface reservoirs, where stress variations influence fluid transport. The current formulation can be adapted to different material systems by choosing appropriate functional forms for the permeability-stress relationship including also  anisotropic permeability variations. 

We are keen to follow other extensions, for instance exploring different abstract results that would impose a less restrictive assumption on smallness of data, and extend the formulation to accommodate a nonlinear stress-strain constitutive law,   transient effects, as well as fully coupled multiphysics systems, such as thermo-poroelasticity. Note however that these generalisations (useful to broaden the applicability of the proposed methods to more complex real-world scenarios) will require a substantially different theoretical framework. 
}

\section*{Funding}
This work has been supported in part by the Australian Research Council through the \textsc{Future Fellowship} grant FT220100496 and \textsc{Discovery Project} grant DP22010316; by the National Research and Development Agency (ANID) of the Ministry of Science, Technology, Knowledge and Innovation of Chile through the postdoctoral program \textsc{Becas Chile} grant 74220026.

%%%%%%%%%%%%%%%%%%%%%%%%%%%%%%%
%\small 
\bibliographystyle{siam}%amsplain}
\bibliography{klrv}

\end{document}